\documentclass[11pt]{article}
\usepackage[margin=0.875in]{geometry}
\usepackage{amsmath,amssymb,amsfonts,amsthm}
\usepackage{mathtools} 
\usepackage{bbm}
\usepackage{titling}
\usepackage{enumitem}
\usepackage{bm}
\usepackage{tikz}
\usepackage{hyperref}
\usepackage{cleveref}
\usepackage{fullpage}
\usepackage{setspace}
\usepackage{tikz-cd}
\usepackage{cjhebrew}
\usepackage[maxbibnames=99,style=alphabetic,doi=false,isbn=false,url=false,eprint=false,backend=bibtex]{biblatex}
\usepackage{stmaryrd}
\usepackage{mathrsfs}

\definecolor{emerald}{rgb}{0.31, 0.78, 0.47}
\definecolor{bleudefrance}{rgb}{0.19, 0.55, 0.91}
\definecolor{brandeisblue}{rgb}{0.0, 0.44, 1.0}
\hypersetup{colorlinks=true,linkcolor=brandeisblue,citecolor=emerald}
\newcommand{\p}{\mathbb{P}}
\newcommand{\e}{\mathbb{E}}
\newcommand{\eps}{\varepsilon}
\newcommand{\bra}[1]{\left(#1\right)}

\newcommand{\abs}[1]{\left\lvert#1\right\rvert}
\newcommand{\op}[1]{\operatorname{#1}}

\newcommand{\Aut}{\operatorname{Aut}}

\newcommand{\Z}{\mathbb{Z}}

\newcommand{\cU}{\mathcal U}

\newcommand{\be}{\begin{equation}}
\newcommand{\ee}{\end{equation}}

\newcommand{\cG}{\mathcal G}

\makeatletter
\DeclareFontFamily{OMX}{MnSymbolE}{}
\DeclareSymbolFont{MnLargeSymbols}{OMX}{MnSymbolE}{m}{n}
\SetSymbolFont{MnLargeSymbols}{bold}{OMX}{MnSymbolE}{b}{n}
\DeclareFontShape{OMX}{MnSymbolE}{m}{n}{
    <-6>  MnSymbolE5
   <6-7>  MnSymbolE6
   <7-8>  MnSymbolE7
   <8-9>  MnSymbolE8
   <9-10> MnSymbolE9
  <10-12> MnSymbolE10
  <12->   MnSymbolE12
}{}
\DeclareFontShape{OMX}{MnSymbolE}{b}{n}{
    <-6>  MnSymbolE-Bold5
   <6-7>  MnSymbolE-Bold6
   <7-8>  MnSymbolE-Bold7
   <8-9>  MnSymbolE-Bold8
   <9-10> MnSymbolE-Bold9
  <10-12> MnSymbolE-Bold10
  <12->   MnSymbolE-Bold12
}{}

\let\llangle\@undefined
\let\rrangle\@undefined
\DeclareMathDelimiter{\llangle}{\mathopen}%
                     {MnLargeSymbols}{'164}{MnLargeSymbols}{'164}
\DeclareMathDelimiter{\rrangle}{\mathclose}%
                     {MnLargeSymbols}{'171}{MnLargeSymbols}{'171}

\makeatletter
\pgfdeclareshape{slash underlined}
{
  \inheritsavedanchors[from=rectangle] 
  \inheritanchorborder[from=rectangle]
  \inheritanchor[from=rectangle]{north}
  \inheritanchor[from=rectangle]{north west}
  \inheritanchor[from=rectangle]{north east}
  \inheritanchor[from=rectangle]{center}
  \inheritanchor[from=rectangle]{west}
  \inheritanchor[from=rectangle]{east}
  \inheritanchor[from=rectangle]{mid}
  \inheritanchor[from=rectangle]{mid west}
  \inheritanchor[from=rectangle]{mid east}
  \inheritanchor[from=rectangle]{base}
  \inheritanchor[from=rectangle]{base west}
  \inheritanchor[from=rectangle]{base east}
  \inheritanchor[from=rectangle]{south}
  \inheritanchor[from=rectangle]{south west}
  \inheritanchor[from=rectangle]{south east}
  \inheritanchorborder[from=rectangle]
  \foregroundpath{
    \southwest \pgf@xa=\pgf@x \pgf@ya=\pgf@y
    \northeast \pgf@xb=\pgf@x \pgf@yb=\pgf@y
    \pgf@xc=\pgf@xa
    \advance\pgf@xc by .5\pgf@xb
    \pgf@yc=\pgf@ya
    \advance\pgf@xc by 0.4pt
    \advance\pgf@yc by -1.8pt
    \pgfpathmoveto{\pgfqpoint{\pgf@xc}{\pgf@yc}}
    \advance\pgf@xc by  2.6pt
    \advance\pgf@yc by  3.6pt
    \pgfpathlineto{\pgfqpoint{\pgf@xc}{\pgf@yc}}
    \pgfpathmoveto{\pgfqpoint{\pgf@xa}{\pgf@ya}}
    \pgfpathlineto{\pgfqpoint{\pgf@xb}{\pgf@ya}}
 }
}
\tikzset{nomorepostaction/.code=\let\tikz@postactions\pgfutil@empty}

\newcommand\nxleftrightarrow[2][]{%
  \mathrel{\tikz[baseline=-.7ex] \path node[slash underlined,draw,<->,anchor=south] {\(\scriptstyle #2\)} node[anchor=north] {\(\scriptstyle #1\)};}}
\makeatother

\newcommand\upsc[1]{\text{\textup{\textsc{#1}}}}
\newcommand\burnin{\op{Burn}}
\newcommand\sprinkle{\op{Spr}}

\crefname{thm}{Theorem}{Theorems}
\crefname{lem}{Lemma}{Lemmas}
\crefname{clm}{Claim}{Claims}
\crefname{rk}{Remark}{Remarks}
\crefname{prop}{Proposition}{Propositions}
\crefname{defn}{Definition}{Definitions}
\crefname{cor}{Corollary}{Corollaries}
\crefname{conj}{Conjecture}{Conjectures}
\crefname{question}{Question}{Questions}
\crefname{section}{Section}{Sections}
\crefname{ineq}{inequality}{inequalities}

\theoremstyle{plain}
\newtheorem{thm}{Theorem}[section]
\newtheorem*{thm*}{Theorem}
\newtheorem{lem}[thm]{Lemma}
\newtheorem*{lem*}{Lemma}

\newtheorem*{clm*}{Claim}
\newtheorem{cor}[thm]{Corollary}
\newtheorem*{cor*}{Corollary}
\newtheorem{prop}[thm]{Proposition}
\newtheorem*{prop*}{Proposition}
\newtheorem{conj}[thm]{Conjecture}
\newtheorem*{conj*}{Conjecture}
\newtheorem{problem}[thm]{Problem}
\theoremstyle{definition}
\newtheorem{defn}[thm]{Definition}
\newtheorem{defn*}{Definition}

\theoremstyle{remark}
\newtheorem{rk}{Remark}[section]
\newtheorem*{rk*}{Remark}

\usepackage[T1]{fontenc}

\pretitle{\begin{flushleft}\Large}
\posttitle{\par\end{flushleft}\vspace{1em}}
\preauthor{\begin{flushleft}}
\postauthor{\\
\vspace{0.5em}
\small{The Division of Physics, Mathematics and Astronomy, California Institute of Technology.\\ Email: \href{mailto:peaso@caltech.edu}{peaso@caltech.edu} and
\href{mailto:t.hutchcroft@caltech.edu}{t.hutchcroft@caltech.edu}}
\end{flushleft}}
\predate{\vspace{0.5em} \begin{flushleft}}
\postdate{\par\end{flushleft}}
\renewenvironment{abstract}
 {\par\noindent\textbf{\abstractname.}\ \ignorespaces}
 {\par\medskip}
\title{\bf The critical percolation probability is local}
\author{{\bf Philip Easo and Tom Hutchcroft }}
\date{\small{\today}}
\begin{document}
\maketitle
\begin{abstract}
We prove Schramm's locality conjecture for Bernoulli bond percolation on transitive graphs: If $(G_n)_{n\geq 1}$ is a sequence of infinite vertex-transitive graphs converging locally to a vertex-transitive graph $G$ and $p_c(G_n) \neq 1$ for every $n \geq 1$ then $\lim_{n\to\infty} p_c(G_n)=p_c(G)$. Equivalently, the critical probability $p_c$ defines a continuous function on the space $\mathcal{G}^*$ of infinite vertex-transitive graphs that are not one-dimensional.  
 As a corollary of the proof, we obtain a new proof that $p_c(G)<1$ for every infinite vertex-transitive graph that is not one-dimensional.
\end{abstract}

\newgeometry{margin=1.2in}

\newpage

\tableofcontents

\newgeometry{margin=0.88in}

\newpage

\section{Introduction}

In \textbf{Bernoulli bond percolation}, the edges of a connected, locally finite graph $G$ are chosen to be either retained (\textbf{open}) or deleted (\textbf{closed}) independently at random, with probability $p\in [0,1]$ of retention. The law of the resulting random subgraph is denoted $\mathbb{P}_p=\mathbb{P}_p^G$. Percolation theory is concerned primarily with the geometry of the connected components of this random subgraph, which are known as \textbf{clusters}. Much of the interest in the model arises from the fact that it undergoes a \emph{phase transition}: If we define the \textbf{critical probability}
\[
p_c(G) = \inf\bigl\{p\in [0,1]: \text{there exists an infinite cluster $\mathbb{P}_p$-almost surely}\bigr\}
\]
then ``most'' interesting graphs have $p_c(G)$ strictly between $0$ and $1$, so that there is a non-trivial phase where infinite clusters do not exist followed by non-trivial phase where they do exist.
 We will be primarily interested  in percolation on \textbf{(vertex-)transitive} graphs, i.e., graphs for which any vertex can be mapped to any other vertex by an automorphism of the graph. We allow our graphs to contain loops and multiple edges, and make the implicit assumption throughout the paper that all transitive graphs are connected and locally finite (i.e.\ have finite vertex degrees).

\medskip

Many interesting features of percolation on an infinite transitive graph at and near the critical point are expected to be \emph{universal}, meaning that they depend only on the graph's \emph{large-scale} geometry and not its microscopic structure. For example, the \emph{critical exponents} governing the power-law behaviour of various interesting quantities at and near criticality are believed to depend only on the volume-growth dimension of the graph, and should therefore take the same values on e.g.\ the square and triangular lattice.
In contrast, Schramm conjectured around 2008 \cite[Conjecture 1.2]{MR2773031} that the \emph{value} of the critical probability $p_c(G)$ should be entirely determined by the \emph{local} (microscopic) geometry of the graph, \emph{subject to the global constraint that $p_c(G)<1$}. 
More precisely, he conjectured that if $G_n$ is a sequence of infinite transitive graphs converging to an infinite transitive graph $G$ in the local topology (defined below) and $\limsup_{n\to\infty} p_c(G_n)<1$ then $p_c(G_n)\to p_c(G)$ as $n\to\infty$. 
The assumption that $\limsup_{n\to\infty} p_c(G) <1$ is needed to rule out degenerate one-dimensional examples such as the cylinder $\mathbb{Z} \times (\mathbb{Z}/n\mathbb{Z})$ (which converges to the square grid $\mathbb{Z}^2$ but which has $p_c(\mathbb{Z} \times (\mathbb{Z}/n\mathbb{Z}))=1 \nrightarrow p_c(\mathbb{Z}^2)=1/2$), and is now known to be equivalent to the graphs $G_n$ having superlinear volume growth for all sufficiently large $n$ \cite{MR4181032,hutchcroft2021nontriviality}. The fact that $p_c$ is \emph{lower} semi-continuous 
 follows straightforwardly from standard facts about percolation on transitive graphs as observed in \cite[\S 14.2]{Gabor} and \cite[p.4]{duminil2016new} and does not require the assumption that the graphs are not one-dimensional; the difficult part of the conjecture is to prove \emph{upper} semi-continuity. 

\medskip

 The locality conjecture has inspired a great deal of subsequent work, including both partial progress on the original conjecture \cite{MR2773031,MR3630298,hutchcroft2019locality,hermon2021no,MR4529920}, which we review in detail below, and analogous results in other settings including self-avoiding walk \cite{grimmett2017connective,grimmett2018locality}, the random cluster model \cite{duminil2019note}, finite random graphs \cite{MR2773031,krivelevich2020asymptotics,sarkar2021note,van2021giant,ren2022locality,alimohammadi2022algorithms,borgs2023locality,alimohammadi2023locality}, and geometric random graphs \cite{hansen2021poisson,lichev2023first}.


\newpage

In this paper we give a complete proof of Schramm's locality conjecture. 

\begin{thm} \label{thm:main}
      Let $\mathcal G^*$ be the set of all infinite transitive graphs that are not one-dimensional, endowed with the local topology. Then the function $p_c : \mathcal G^* \to (0,1)$ is continuous. 
\end{thm}

Here, the \textbf{local topology} (a.k.a.\ the \textbf{Benjamini-Schramm topology}) on the space of transitive graphs, denoted by $\mathcal{G}$, is defined so that $(G_n)_{n=1}^{\infty}$ converges to $G$ if and only if, for each $r\geq 1$, the balls of radius $r$ in $G_n$ and $G$ are isomorphic as rooted graphs for all sufficiently large $n$. We say that an infinite transitive graph is \textbf{one-dimensional} if it has linear volume growth (i.e., if its balls $B_n$ satisfy $|B_n|=O(n)$ as $n\to \infty$);
it follows from (a simple special case of) the structure theory of transitive graphs of polynomial growth that 
an infinite transitive graph is one-dimensional if and only if it is rough-isometric to $\mathbb{Z}$ \cite{MR3573922}, while the results of \cite{MR4181032} imply that an infinite transitive graph has $p_c<1$ if and only if it is not one-dimensional. (In fact the proof of \cref{thm:main} also yields a new proof of this theorem as we discuss in detail in \cref{sec:induction_step}.)

\begin{rk}
In our forthcoming paper \cite{easo2021supercritical2}, we prove related results on the locality of the \emph{density} of the infinite cluster, implying in particular that the percolation probability $\theta(p,G)=\mathbb{P}^G_p(o \leftrightarrow \infty)$ is a continuous function of $(G,p)$ in the supercritical set $\{(G,p): G\in \mathcal G^*, p>p_c(G)\}$. (\cref{thm:main} implies that this set is open.) 
An alternative proof of this result using the methods developed in the present paper is sketched in \cref{sec:continuity_of_density}.
\end{rk}

\subsection{Previous work}

In this section we overview previous work on locality, the $p_c<1$ problem, and the structure theory of transitive graphs of polynomial growth. Our proof will employ many ideas and methods from these earlier papers, including most notably the works \cite{hutchcroft2019locality,hutchcroft2020nonuniqueness,contreras2022supercritical} and the structure theory developed in \cite{breuillard2011structure,MR4253426}.

\medskip

\noindent \textbf{Euclidean lattices.} Well before the formulation of Schramm's conjecture, the first significant work on locality was carried out in the seminal work of Grimmett and Marstrand \cite{MR1068308}, who proved that the critical probability for percolation on a ``slab'' $\Z^{d-k} \times \{0,\ldots,n\}^k$ converges to $p_c(\Z^d)$ as $n\to\infty$ provided that $d-k \geq 2$. (In this context $\mathbb Z^d$ and $\Z^{d-k} \times \{0,\ldots,n\}^k$ refer to the Cayley graphs of these groups \emph{with their standard generating sets}.) This theorem and the methods developed to prove it are of central importance to the study of supercritical percolation in three and more dimensions. (Moreover, one of the motivations for the work of Grimmett and Marstrand was to get closer to proving that the percolation phase transition on $\mathbb Z^d$ is continuous, and indeed similar methods were used in \cite{MR1124831} to prove continuity for half-spaces.)  Although it is not strictly an instance of Schramm's conjecture since slabs are not transitive, the Grimmett--Marstrand theorem trivially implies that the analogous statement holds for ``slabs with periodic boundary conditions'' (i.e., $\Z^{d-k} \times (\Z/n\Z)^k$ with its standard generating set), which are transitive. The proof of the Grimmett--Marstrand theorem relies heavily on renormalization techniques exploiting the full symmetries of $\Z^d$ and scale-invariance of Euclidean space $\mathbb{R}^d$, and does not readily generalize to other transitive graphs. Let us also mention that a quantitative version of the Grimmett--Marstand theorem was proven in the more recent work of Duminil-Copin, Kozma, and Tassion \cite{duminil2021upper}  which was very influential in both our work and \cite{contreras2022supercritical}.

\begin{rk}
A further classical Euclidean result in the spirit of the locality conjecture was established by Kesten \cite{kesten1990asymptotics}, who proved that $p_c(\Z^d)\sim 1/(2d-1)$ as $d\to \infty$. See \cite{alon2004percolation} for a simple proof and \cite{slade2006lace} for more refined results. While not strictly an instance of the locality conjecture since $\Z^d$ does not converge in the local topology, the intuitive reason for this result to hold is that $\Z^d$ is ``locally tree-like'' when $d\to \infty$ in the sense that small cycles have a negligible effect on the behaviour of the percolation model, so that one can define a kind of ``local limit'' of $\Z^d$ as $d\to\infty$ (in a different technical sense to the one we consider here) in terms of Aldous's \emph{Poisson weighted infinite tree} (PWIT) \cite{aldous2004objective}. 
\end{rk}

\noindent \textbf{Progress on locality.} Previous works on the locality conjecture can be divided into two strains, with completely different set of methods and domains of application associated to them: The first strain concerns graphs that satisfy various strong, ``infinite-dimensional'' expansion conditions, while the second concerns ``finite-dimensional'' graphs (i.e., graphs of polynomial volume growth) where one can hope to develop appropriate generalizations of Grimmett--Marstrand theory.
  The first strain splits further into two cases according to whether or not the graphs in question are \emph{unimodular}, a technical condition\footnote{Here is the definition: A transitive graph $G=(V,E)$ is unimodular if it satisfies the \textbf{mass-transport principle}, meaning that $\sum_x F(o,x) = \sum_x F(x,o)$ for every $F:V^2\to [0,\infty]$ that is diagonally invariant in the sense that $F(x,y)=F(\gamma x, \gamma y)$ for every automorphism $\gamma$ of $G$. We will not directly engage with unimodularity in this paper, but it will appear as a hypothesis in many of our intermediate results since it is needed to apply the two-ghost inequality of \cite{hutchcroft2019locality}.} that holds for most familiar examples of transitive graphs including every Cayley graph and every amenable transitive graph \cite{soardi1988amenability}.  Although nonunimodular graphs are often considered to be ``pathological'' compared to their unimodular cousins, it turns out that \emph{nonunimodularity is actually a very helpful assumption}: in \cite{hutchcroft2020nonuniqueness} the second author carried out a very detailed analysis of critical percolation on nonunimodular transitive graphs, which he then used to prove the nonunimodular case of locality in \cite{hutchcroft2019locality}. Moreover, it was proven in \cite[Corollary 5.5]{hutchcroft2019locality} that the set of nonunimodular transitive graphs is both closed and open in $\cG$, so that to prove \cref{thm:main} it now suffices to consider the case that all graphs in question are unimodular.

Let us now discuss previous results for unimodular graphs in the ``infinite-dimensional'' setting.
The first result in this direction was due to Benjamini, Nachmias, and Peres \cite{MR2773031}, who proved the conjecture for nonamenable graph sequences satisfying a certain high girth condition (e.g., uniformly nonamenable graph sequences of divergent girth; unimodularity is not required).
More recently, the second author \cite{hutchcroft2019locality} proved the conjecture for  graph sequences of uniform exponential growth (meaning that the balls of radius $r$ in the graphs $G_n$ all have volume lower-bounded by $e^{cr}$ for some constant $c$ independent of $n$ and $r$), and Hermon and the second author \cite{hermon2021no} proved the conjecture for sequences of graphs satisfying a certain uniform stretched-exponential heat kernel upper bound, a class that includes certain examples of intermediate volume growth (i.e., volume growth that is superpolynomial but subexponential; note however that the spectral condition of \cite{hermon2021no} is not implied by any growth condition). The works \cite{hutchcroft2019locality,hutchcroft2020nonuniqueness,hermon2021no} all establish locality for the families of graphs they consider by proving quantitative tail estimates on critical percolation clusters that hold uniformly for all graphs in the family, yielding much more than just locality. (In particular, they also imply that the graphs in question have continuous percolation phase transitions.) 

\begin{rk}
Although the techniques developed in \cite{hutchcroft2019locality,hermon2021no} have yet to be made to work for \emph{nearest-neighbour} percolation models in finite dimension, versions of these arguments have been used in \cite{hutchcroft2021power} to analyze certain \emph{long-range} percolation models in finite-dimensional spaces. 
 The results of \cite{hutchcroft2021power} can be used to prove versions of the locality conjecture for certain large families of long-range percolation models on unimodular transitive graphs (under the assumption that the long-range edge kernel has a sufficiently heavy tail uniformly throughout the sequence). Further results on locality for long-range percolation can be found in \cite{MR1405960,MR1896880}.
\end{rk}

\noindent
\textbf{Polynomial growth and structure theory.}
We now discuss the second strain of results, concerning graphs of polynomial volume growth. Let us first briefly review the \emph{structure theory} of transitive graphs of polynomial growth, which plays an important role in these developments.
Recall that $\mathcal G$ is the space of all infinite transitive graphs and that $G \in \mathcal G$ is said to have \textbf{polynomial growth} if for some positive reals $C$ and $d$, the number of vertices contained in a ball of radius $n$, denoted $\op{Gr}(n) := \abs{B_n(o)}$, satisfies $\op{Gr}(n) \leq Cn^d $ for all $n\geq 1$. The geometry of such graphs is highly constrained: it is a consequence of Gromov's theorem \cite{MR623534} and Trofimov's theorem \cite{MR735714} that every $G \in \mathcal G$ with polynomial growth is necessarily quasi-isometric to the Cayley graph of a nilpotent group. In particular, for every such graph $G$, there is a positive real $C$ and a unique positive integer $d$ such that $C^{-1} n^d \leq \op{Gr}(n) \leq Cn^d$ for all $n \geq 1$. The integer $d$ is called the (volume growth) \textbf{dimension} of $G$; it coincides with the \emph{isoperimetric dimension} and \emph{spectral dimension} of $G$ by a theorem of Coulhon and Saloff-Coste \cite{MR1232845}. These results are often used in the study of probability on transitive graphs as part of a ``structure vs.\ expansion dichotomy'', wherein each graph either satisfies a high-dimensional isoperimetric inequality (which is often a helpful assumption) or else is quasi-isometric to a nilpotent group of bounded step and rank (which is useful because these graphs are highly explicit and well-behaved); a detailed overview of the structure theory of transitive graphs of polynomial growth and its applications to probability is given in the introduction to \cite{EHStructure}.

More recently, \emph{finitary} versions of these results have been established, first for groups in the landmark work of Breuillard, Green, and Tao \cite{breuillard2011structure}, then for transitive graphs by Tessera and Tointon~\cite{MR4253426}. These results imply, for instance, that for each constant $K <\infty$ there exists $N<\infty$ such that if we observe that $\op{Gr}(3n) \leq K \op{Gr}(n)$ for some $n \geq N$, then $G$ is $(1,Cn)$-quasi isometric the Cayley graph of a virtually nilpotent group where the constant $C$ along with the rank, step, and index of the nilpotent subgroup are all bounded above by some function of $K$ (see \cref{subsection:structure_theory} for further details). 
This finitary structure theory is extremely useful in applications to problems such as locality in which one wishes to argue in a way that is uniform over some family of graphs. For example, 
it follows from \cite[Corollary 1.5]{MR4253426} that for each $d\geq 1$ the set of transitive graphs of polynomial growth with dimension at most $d$ is an open subset of $\cG$, and moreover that 
  if $G_n \to G$ with $G$ of polynomial growth of dimension $d$ then there exists $n_0<\infty$ and a constant $C$ such that the ball of radius $r$ in $G_n$ has volume at most $Cr^d$ for every $n\geq n_0$, with constants independent of $n$. As such, to prove locality in the case that the limit has polynomial growth, it suffices to consider the case that all graphs in the sequence satisfy a \emph{uniform} polynomial upper bound on their growth as well as various other forms of strong uniform control on their geometry.

 Besides the original work of Grimmett and Marstrand, the first result on locality for graphs of polynomial growth was due to Martineau and Tassion \cite{MR3630298}, who proved that locality holds for Cayley graphs of abelian groups.
  Their proof employs a variation on the Grimmett--Marstrand argument, overcoming significant technical difficulties arising due to the loss of rotational and reflection symmetry. This result was greatly extended in the recent work of Contreras, Martineau, and Tassion, who developed a version of Grimmett--Marstrand theory for transitive graphs of polynomial growth in \cite{contreras2022supercritical} and used this theory together with the finitary structure theory discussed above to deduce the polynomial growth case of the locality conjecture in \cite{MR4529920}. As with the aforementioned works in the infinite-dimensional setting, the works \cite{contreras2022supercritical,MR4529920} establish not just locality but also many further strong quantitative results about  percolation  on the classes of graphs they consider. However, while \cite{hutchcroft2019locality,hermon2021no} established quantitative estimates on finite clusters in \emph{critical} percolation, \cite{contreras2022supercritical,MR4529920} instead establish strong results about the geometry of the infinite cluster in \emph{supercritical} percolation. This reflects a fundamental distinction between the two approaches, with the analysis of the polynomial growth case involving estimates on uniqueness of annuli crossings etc.\ that are simply not true for ``big'' graphs like the 3-regular tree.

\medskip
\noindent \textbf{What challenges remain?} Given the previous results discussed above, it appears that there are two main cases of the locality conjecture left to consider: arbitrary sequences of superpolynomial growth transitive graphs converging to a superpolynomial growth graph, and the ``diagonal'' case in which a sequence of polynomial-growth graphs converges to a graph of superpolynomial growth. Moreover, the second case might be split further according to whether the graphs in the sequence have bounded or divergent dimension. (As we will soon explain, our proof will in fact work through a different and less obvious kind of case analysis.) In the first case, a key difficulty is that the geometry of transitive graphs of superpolynomial growth can be highly arbitrary, with the space of all such graphs being an ineffably complex object in some senses: the challenge is precisely that we need  an argument that is robust enough to work for all possible transitive graphs, for which there is nothing like a general classification.
 Moreover, the best known uniform lower bound on the growth of groups of superpolynomial growth, due to Shalom and Tao \cite{shalom2010finitary}, is of the form $n^{(\log \log n)^c}$ for a small constant $c>0$. This growth lower bound (which has not yet been proven for transitive graphs that are not Cayley) is \emph{vastly} weaker\footnote{While a well-known conjecture of Grigorchuk \cite{grigorchuk2014gap} states that every superpolynomial growth transitive graph has growth at least $\exp[c n^{1/2}]$, this conjecture is completely open, somewhat controversial, and in any case would still require a significant advance on the methods of \cite{hutchcroft2019locality,hermon2021no} to be applicable towards the locality conjecture.} than the assumptions used in the works \cite{MR2773031,hutchcroft2019locality,hermon2021no} discussed above.  In the diagonal case, one must contend with these same difficulties again together with the total incompatibility of the methods that have thus far been used to handle the high-growth and polynomial growth cases. As such, while the main difficulties in the locality conjecture arise from ``unknown enemies'' hiding deep within the unknowable expanse of the space of all transitive graphs, there are also explicit examples that seem difficult to handle within existing frameworks. (One such example is the standard Cayley graph of the free step-$s$ nilpotent group on two generators, which converges to a $4$-regular tree as $s\to \infty$ but has finite dimension for each finite $s$.)

\medskip

\noindent \textbf{Parallels with $p_c<1$.} Before we begin to describe our proof of the locality conjecture, let us first discuss how its history closely parallels the (older) history of the $p_c<1$ problem. It follows from the classical work of Peierls \cite{peierls1936ising} that $p_c(\Z^d)<1$ for every $d\geq 2$. In their highly influential work \cite{MR1423907}, Benjamini and Schramm conjectured that $p_c<1$ for every transitive graph that is not one-dimensional. Benjamini and Schramm also proved in the same paper that $p_c<1$ for every (not necessarily transitive) nonamenable graph, while earlier results of Lyons \cite{lyons1995random} implied that exponential growth suffices in the transitive case. On the other hand, it is a simple consequence of the structure theory that every transitive graph of polynomial volume growth that is not one-dimensional contains a subgraph that is quasi-isometric to $\Z^2$, which easily implies that every such graph has $p_c<1$ (see e.g. \cite[Section 3.4]{hutchcroft2021nontriviality} for details). As such, for many years the problem remained open only for groups of intermediate growth. Even in this case the problem was solved for most ``known'' examples of graphs of intermediate growth, such as the Grigorchuk group \cite{muchnik2001percolation,raoufi2017indicable}, with the main remaining difficulty coming from ``unknown enemies'' as discussed above. See the introduction of \cite{MR4181032} for a detailed account of this partial progress including several further references.

The $p_c<1$ problem was eventually solved in full generality by Duminil-Copin, Goswami, Raoufi, Severo, and Yadin \cite{MR4181032}. More precisely, they established that $p_c<1$ for any (not necessarily transitive) bounded degree graph satisfying a $(4+\eps)$-dimensional isoperimetric inequality, with the structure theory and classical results above handling all remaining transitive graphs. Their proof uses a comparison between percolation and the Gaussian free field which works only for values of $p$ very close to $1$, making their methods unsuitable for the locality problem. A finitary version of the results of \cite{MR4181032} was developed by the second author and Tointon in \cite{hutchcroft2021nontriviality}, who proved in particular that sequences of \emph{finite} transitive graphs have a non-trivial phase in which a \emph{giant} cluster exists provided that they are ``not one dimensional'' in an appropriate quantitative sense. This finitary approach also allowed them to prove a \emph{uniform} version of the main result of \cite{MR4181032}, stating that for each $d\geq 1$ there exists $\eps>0$ such that every infinite transitive graph that has degree at most $d$ and is not one-dimensional has $p_c<1-\eps$. (For Cayley graphs it is now known that $\eps$ can be taken independently of the degree \cite{panagiotis2021gap}; see \cref{sec:pc_gap} for some related conjectures.) 

\medskip

\noindent 
\textbf{Can we do something similar?}
Continuing to follow the path set by this previous work on the $p_c<1$ problem, one might hope to prove locality via a similar dichotomy, finding some method that handles all graphs that are ``high-dimensional'' in some sense, then using the structure theory to separately analyze the remaining ``low-dimensional'' examples. One technical problem with this approach, which was already a major hurdle in \cite{hutchcroft2021nontriviality}, is that one is forced to consider sequences of graphs that may look high-dimensional up to some divergently large scale then switch to looking low-dimensional, meaning that one must find a way to ``patch together'' the outputs of the two different case analyses at the crossover scale. A more fundamental problem, however, is that to date there have simply been no viable approaches to prove locality under the assumption that the graphs are high dimensional. 

As we will see, our proof will instead follow a more subtle approach in which we first dichotomize into two much less obvious cases according to whether or not the graph has \emph{quasi-polynomial growth} on the relevant scale. (Here, a function is said to have \emph{quasi-polynomial growth} if it is bounded by a function of the form $\exp[(\log n)^{O(1)}]$.)  In the low growth case, we then employ a second, subordinate dichotomization according to whether the \emph{rate of growth} on the relevant scale is low-dimensional or high-dimensional; it is in this second dichotomy that we can make use of the structure theory.
%
 %
%
%
A key technical difficulty when arguing this way is that (as far as we know with the current structure theory), the same graph might oscillate between the quasi-polynomial and super-quasi-polynomial regimes infinitely many times as we go up the scales, so that any ``case analysis'' we do via this dichotomy must be able to handle this oscillation. Moreover, the delicate nature of our proof leads us to engage with the structure theory literature in a deeper way than had previously been necessary in applications to probability. Indeed, our proof relies in part on our structure-theoretic companion paper \cite{EHStructure} in which we prove a ``uniform finite presentation'' theorem for groups of polynomial volume growth which we use to make some of the arguments from \cite{contreras2022supercritical} finitary. Besides this, we must also contend with the fact that the assumption of quasi-polynomial growth is highly non-standard, so that we must spend a significant amount of the paper studying the deterministic geometry of transitive graphs of quasi-polynomial growth (at some scale) with a view to eventually generalizing the methods of \cite{contreras2022supercritical} from polynomial to quasi-polynomial growth.


Interestingly, our proof of locality yields as an immediate corollary a new proof that $p_c<1$ for transitive graphs that are not one-dimensional, recovering the main result of \cite{MR4181032}. This proof works directly with Bernoulli percolation and does not rely on the comparison to the GFF in any way. (On the other hand it is also \emph{much} more complicated than the original proof!)

\subsection{About the proof}

In this section we give an overview of our proof. Let us first establish some relevant notation that will be used throughout the paper. Recall that $\cG$  denotes the space of all (vertex-)transitive graphs (which we always take to be connected and locally finite) and that $\cG^* \subseteq \cG$ is the space of infinite transitive graphs that are not one-dimensional. We also write $\cU \subseteq \cG$ for the space of all unimodular transitive graphs and write $\cU^*=\cU \cap \cG^*$. 
Given $d\in \mathbb{N}$, we write $\cG_d$, $\cG_d^*$, and $\cU^*_d$ for the subsets of these spaces in which every graph has degree $d$. Given sets of vertices $A$ and $B$ in a graph, we write $\{A \leftrightarrow B\}$ for the event that there is a path from $A$ to $B$ in the given configuration. We also use the notation $A \leftrightarrow \infty$ to mean that there is an infinite cluster that intersects $A$. 
When $A = \{u\}$ and $B = \{v\}$ are singletons, we may simply write $u \leftrightarrow v$ and $u \leftrightarrow \infty$ instead.
 For each transitive graph $G$ we will write $o$ for an arbitrarily chosen root vertex of $G$ which we will refer to as \emph{the origin}. We write $B_n$ for the graph-distance ball of radius $n$ around $o$ in $G$ and write $S_n$ for the set of vertices at distance exactly $n$ from $o$.

\medskip

\noindent \textbf{Uniform estimates and finite-size criteria.}
Recall that percolation is said to undergo a \emph{continuous phase transition} on an infinite graph $G$ if the function $\theta_G: p \mapsto \p_p^G(o \leftrightarrow \infty)$ is continuous (it is a theorem of Schonmann \cite{MR1676831} that $\theta_G$ is always continuous at every point other than $p_c$). One of the best-known conjectures in the study of percolation on general infinite transitive graphs is that percolation should undergo a continuous phase transition on every $G \in \mathcal G^*$ \cite[Conjecture 4]{MR1423907}; this is famously still open when $G$ is the three-dimensional cubic lattice. Although it might not be obvious from their statements, the locality conjecture and the continuity conjecture are very closely related. To see why, let us give two equivalent formulations of \cref{thm:main}, which will inform the remainder of our analysis. (For notational convenience we define $\p_{p} := \p_1$ when $p > 1$.)

\theoremstyle{plain}
\newtheorem{reform}{Reformulation of the locality conjecture}

\begin{reform} \label{prop:locality_phrased_as_FSC}
      For each $d \in \mathbb N$ and $\eps,\delta >0$, there exists $r \in \mathbb N$ such that 
      \begin{equation} \label{eq:main_locality_as_fsc}
            \p_p^G( o \leftrightarrow S_r ) \geq \delta \quad \implies \quad \p_{p+\eps}^G( o \leftrightarrow \infty) > 0
      \end{equation}
      for every $G \in \mathcal G_d^*$ and $p \in [0,1]$.
\end{reform}

\begin{reform} \label{prop:locality_phrased_as_FSC2}
      For each $d \in \mathbb N$ and $\eps >0$, there exists a function $h_\eps=h_{d,\eps}:\mathbb N\to (0,\infty)$ with $h_\eps(r)\to 0$ as $r\to\infty$ such that
      \begin{equation} \label{eq:main_locality_as_fsc2}
            \p_{p_c-\eps}^G( o \leftrightarrow S_r ) \leq h_\eps(r)
      \end{equation}
      for every $G\in \mathcal G_d^*$ and $r\geq 1$.
\end{reform}

These two reformulations are trivially equivalent to one another (the difference amounts to considering either $h_\eps(n)$ or its inverse), but we find the two different viewpoints on the same statement to be illuminating: the first is formulated in terms of a ``finite-size criterion for not being very subcritical'' while the second is formulated in terms of ``uniform estimates on subcritical percolation''. 
In either formulation, the $\eps=0$ analogue of the same statement would imply both the locality conjecture and a strengthened, ``uniform in $G$'' version of the conjecture concerning the continuity of the phase transition; this is precisely what the arguments of \cite{hutchcroft2019locality,hermon2021no} establish for the classes of graphs they consider. (In fact this uniform version of continuity is implied by locality together with the non-uniform version of continuity by a simple compactness argument.) This suggests a close connection between the two problems, while the freedom to use ``sprinkling'' (i.e., to increase $p$ by small amounts in an appropriate manner) may make locality significantly more tractable.

Let us now briefly explain why Reformulation~\ref{prop:locality_phrased_as_FSC} is equivalent to the locality conjecture (or more accurately to the upper semi-continuity of $p_c$, which is the difficult part of locality). In one direction, suppose that $G_n \to G$ and that $p>p_c(G)$. Since $p>p_c(G)$, the connection probability $\p_{p}^G( o \leftrightarrow S_r)$ does not decay as $r\to \infty$. Since for each fixed $r$ we also have that $\p_{p}^{G_n}( o \leftrightarrow S_r)=\p_{p}^G( o \leftrightarrow S_r)$ for every sufficiently large $n$ by the definition of local convergence, we may apply Reformulation~\ref{prop:locality_phrased_as_FSC} with $\eps = \p_{p}^G( o \leftrightarrow \infty)>0$ and $\delta$ an arbitrary positive number to deduce that $p_c(G_n)\leq p+\delta$ for all sufficiently large $n$. This implies the desired upper semi-continuity since $p>p_c(G)$ and $\delta>0$ were arbitrary. In the other direction, suppose that Reformulation~\ref{prop:locality_phrased_as_FSC} is \emph{false}, so that there exists $d \in \mathbb N$, $\eps>0$, and $\delta>0$ such that for each $r\geq 1$ there exists $G_r\in \cG_d^*$ and $p_r\in [0,1]$ with $p_c(G_r)\geq p_r+\eps$ and $\mathbb P_p^{G_r}(o\leftrightarrow S_r) \geq \delta$. Since $\cG_d^*$ is compact, there exists a subsequence along which $G_r$ converges locally to some transitive graph $G\in \cG_d^*$ and $p_r$ converges to some $p\in [0,1]$, which must satisfy $\mathbb P_p^G(o\leftrightarrow r)\geq \delta$ for every $r\geq 1$ and hence that $p_c(G)\leq p$. Since the liminf of $p_c(G_r)$ along this subsequence is at least $p+\eps$ we deduce that $p_c$ is \emph{not} continuous on $\cG^*$, completing the proof of the equivalence.

\begin{rk}
The lower semi-continuity of $p_c$ follows by a very similar argument to the deduction of upper semi-continuity from Reformulation~\ref{prop:locality_phrased_as_FSC} that we just gave, but where the relevant finite-size criteria or uniform bounds follow easily from standard facts about the sharpness of the phase transition: The argument of \cite{duminil2016new} is formulated in terms of finite-size criteria for subcriticality, while that of \cite{Gabor} is formulated in terms of uniform bounds on supercritical percolation.
\end{rk}

\begin{rk}
\label{rk:we_dont_actually_prove_uniform}
Although the perspective on the locality problem we have just discussed is highly influential on our approach, we will in fact follow a slightly different approach in order to circumvent some technical problems related to estimating the ``burn-in'', i.e., the amount of sprinkling needed to perform the base case of our multi-scale induction scheme. As such, we do not obtain an explicit function $h_\eps$ as in Reformulation~\ref{prop:locality_phrased_as_FSC2} in this paper. The additional steps required to make our argument completely quantitative and obtain an explicit bound on the function $h_\eps$ will be carried out in a forthcoming companion paper \cite{EH_quantitative}.
\end{rk}





\medskip

\noindent\textbf{A non-trivial reformulation.}
Let us now begin to go into more detail about our methods of proof. As discussed above, the results of \cite{hutchcroft2019locality,hutchcroft2020nonuniqueness} completely resolve the nonunimodular case of the conjecture, and since the space $\cU^*$ is both closed and open in $\cG^*$ we may restrict from now on to the case that all graphs are unimodular. 
In this case, the sharpness of the phase transition \cite{MR852458,MR874906,duminil2016new} together with the methods of \cite{hutchcroft2019locality} allow us to further reformulate \cref{thm:main} as follows:

\begin{reform}
\label{reform3}
For all $d \in \mathbb N$ and all $\eps,\delta >0$, there exists $n \in \mathbb N$ such that 
\[
      \min_{u,v \in B_n} \p_p( u \leftrightarrow v) \geq \eps \quad \implies \quad \lim_{m \to \infty} \frac{1}{m} \log \p_{p+\delta}(o \leftrightarrow S_m) = 0
\]
for every $G \in \mathcal U_d^*$ and $p \in [0,1]$,
\end{reform}

The fact that we can replace the statement that $\theta(p+\delta)>0$ appearing on the right hand side of Reformulation~\ref{prop:locality_phrased_as_FSC} with the statement that the radius has a subexponential tail appearing on the right hand side of Reformulation~\ref{reform3} follows directly from the sharpness of the phase transition \cite{MR852458,MR874906,duminil2016new}, which implies that the radius has an exponential tail whenever $p<p_c$. (This does not require unimodularity.) The fact that we can replace the lower bound on the tail of the radius appearing on the left hand side of Reformulation~\ref{prop:locality_phrased_as_FSC} with the lower bound on point-to-point connection probabilities appearing on the left hand side of Reformulation~\ref{reform3} is a much less obvious\footnote{Indeed, unlike the tail of the radius, it is possible that point-to-point connection probabilities continue to decay in the supercritical regime, as is the case on the $3$-regular tree. This breaks the argument that Reformulation~\ref{prop:locality_phrased_as_FSC} implies \cref{thm:main} that we gave above.}  fact and is the main content of \cite{hutchcroft2019locality}. (Indeed, in the exponential growth setting studied in \cite{hutchcroft2019locality}, a uniform upper bound on critical point-to-point connection probabilities was already established in \cite{hutchcroft2016critical}.) The argument needed to see that Reformulation~\ref{reform3} implies \cref{thm:main} is explained in more detail in \cref{sec:AKN,sec:induction_step}.

\medskip
\noindent \textbf{Sprinkled renormalization of the two-point function.} 
We now explain our unconditional proof\footnote{As discussed in \cref{rk:we_dont_actually_prove_uniform}, we do not quite proceed via Reformulation~\ref{reform3} in order to circumvent certain technical obstacles. Despite these caveats, we still think of our argument to be best understood as ``morally'' going via Reformulation~\ref{reform3}, which in any case is implied by \cref{thm:main}.} of Reformulation~\ref{reform3}, which occupies the bulk of the paper. At a very high level, we will use a ``sprinkled multi-scale induction argument'', in which we start with estimates concerning percolation on some scale and deduce similar estimates at a much larger scale after increasing $p$ by some appropriately small amount; if we can do this efficiently enough, so that the total sprinkling is small when we start at a large scale, we can carry the induction up to infinitely many scales and (hopefully) prove that the resulting slightly larger parameter is supercritical (or at least that it is not subcritical).

\medskip

Since our actual induction hypothesis is rather complicated, let us first illustrate how such an argument might work in principle. Let $d\geq 1$ be fixed and suppose that we were able to prove an implication of the form
\begin{equation} \label{eq:simple_eg_of_locality_implication}
      \min_{u,v \in B_{n}} \p_p^G ( u \leftrightarrow v ) \geq \delta(n) \quad \implies \quad \min_{u,v \in B_{\phi(n)}} \p_{p+\eps(n)}^G ( u \leftrightarrow v ) \geq \eta(n)
\end{equation}
held for all $G\in \cU^*_d$, $p\in [0,1]$, and $n\geq 1$, where $\phi:\mathbb N\to \mathbb N$ is strictly increasing and $\eps,\delta,\eta:\mathbb N\to (0,1]$ are decreasing. We claim that such an implication would suffice to prove locality \emph{provided that sufficiently strong quantitative relationships hold between the various functions that appear.} For example, the argument would work provided that $\eta(n)\geq \delta(\phi(n))$, that $\delta(n)$ is subexponentially small as a function of $n$, and that $\sum_{k=0}^\infty \eps(\phi^k(n))<\infty$, where $\phi^k$ denotes the $k$-fold convolution of $\phi$. (Consider for example $\phi(n)=(n+1)^2$, $\delta(n)=(\log n)^{-1}$, $\eta(n)=(2\log n)^{-1}$, and $\eps(n)=(\log \log n)^{-2}$.)
  To see this, note that if we define the sequence $(n_k)_{k\geq 1}$ by $n_0=1$ and $n_{k+1}=\phi(n_k)$ for each $k\geq 1$ then, under this assumption, the implication \eqref{eq:simple_eg_of_locality_implication} implies that
\begin{equation*}
      \min_{u,v \in B_{n_k}} \p_p^G ( u \leftrightarrow v ) \geq \delta(n_k) \quad \implies \quad \min_{u,v \in B_{n_{k+1}}} \p_{p+\eps(n_k)}^G ( u \leftrightarrow v ) \geq \delta(n_{k+1})
\end{equation*}
and hence by induction that
\begin{multline*} 
      \min_{u,v \in B_{n_k}} \p_p^G ( u \leftrightarrow v ) \geq \delta(n_k) \quad \text{ for some $k\geq 1$}\\ \; \implies \; \min_{u,v \in B_{n_{i}}} \p_{p+\sum_{j=k}^{i-1} \eps(n_j) }^G ( u \leftrightarrow v ) \geq \delta(n_{k+1}) \quad \text{ for every $i\geq k$}.
\end{multline*}
Since the conclusion on the right hand side implies that connection probability are subexponential in the distance at $p+\sum_{j=k}^\infty \eps(n_j)$, it would follow from \eqref{eq:simple_eg_of_locality_implication} that if $\min_{u,v \in B_{n_k}} \p_p^G ( u \leftrightarrow v ) \geq \delta(n_k)$ for some $k$ then $p_c\leq p+\sum_{j=k}^\infty \eps(n_j)$. 
The assumption that $\sum_{j=0}^\infty \eps(n_j)$ ensures that the tail sum appearing here is small, verifying Reformulation~\ref{reform3}.

\medskip

Following this approach, we are led to the problem of how to extend point-to-point connection lower bounds from one scale to a much larger scale after sprinkling by a small amount. More specifically, we want to do this \emph{as efficiently as possible}, with the hope of obtaining an inductive statement that is sufficiently strong to imply locality. 

\medskip

\noindent \textbf{Snowballing.}
In \cref{sec:snowballing}, we develop a new method based on Talagrand's theory of sharp thresholds \cite{MR1303654} and ``cluster repulsion'' inequalities inspired by the work of Aizenman-Kesten-Newman \cite{aizenman1987uniqueness} that allows us to prove a bootstrapping implication of the form \eqref{eq:simple_eg_of_locality_implication}, where the function $\phi$ depends on the volume of the ball of radius $n$. While the basic idea of using Talagrand together with Aizenman-Kesten-Newman is already present in \cite{duminil2021upper,contreras2022supercritical} (both in the polynomial growth case), we find a new way of both implementing and applying\footnote{In particular, the argument we use to efficiently extend two-point estimates to a higher scale given the outputs of this sharp threshold argument is also novel.} this argument using \emph{ghost fields} that works directly in infinite volume and uses the \emph{two-ghost inequality} of \cite{hutchcroft2019locality} rather than the classical Aizenman-Kesten-Newman inequality. We call this the \textbf{snowballing} method. While the methods of \cite{duminil2021upper,contreras2022supercritical} needed upper bounds on the growth to work, our method actually becomes \emph{more efficient} as the growth gets larger; if the growth is large enough, the bootstrapping implication we obtain (which is of the form \eqref{eq:simple_eg_of_locality_implication}) is strong enough to prove locality by the argument outlined above.
 Optimizing the snowballing argument as much as we could (see \cref{remark:weaker_growth_assumptions}), we found that this method could prove locality for graph sequences satisfying a uniform growth lower bound of the form $n^{c(\log\log n)^C}$ for some universal constant $C$ and any $c>0$.

\medskip

 For Cayley graphs, the hypothesis needed for this argument to work is of course frustratingly close to the Shalom-Tao bound $\op{Gr}(n) \geq \exp(c\log n (\log\log n)^c)$ \cite{shalom2010finitary}, where $c$ is a \emph{small} universal constant, which holds for every Cayley graph of superpolynomial growth. Thus, even a modest improvement to Shalom-Tao or the snowballing argument would allow us to prove locality for all sequences of superpolynomial growth via this method. Since we were not able to improve either argument to the required extent (and in any case must also deal with the ``diagonal'' case of locality), we will instead attack the locality conjecture in a ``pincer movement'', where we push the polynomial growth methods of \cite{contreras2022supercritical} to handle all growths that are too slow to be treated by the snowballing argument. In fact we will push these methods to handle all graphs of \emph{quasi-polynomial} growth $\op{Gr}(n)\leq \exp((\log n)^C)$, so that there is a considerable overlap in the growth regimes handled by the two methods. As mentioned above, a key technical difficulty is that (as far as we know) the same graph may oscillate between the two growth regimes infinitely often, so that the two methods must be harmonized in some way to allow for this.

\medskip

\noindent \textbf{Chaining via orange peeling.}
A well-known general approach to efficiently extend connection lower bounds from one scale to another is by a method we will call \emph{chaining}. Consider the event $\mathscr E_R$ that $S_R \leftrightarrow S_{10R}$ and the event $\mathscr U_R$ that there is at most one cluster intersecting both $S_{2R}$ and $S_{5R}$. Suppose we knew that for some large $R$ and small $\eps > 0$,
\begin{equation} \label{eq:vague_E_and_U_whp}
      \p_p(\mathscr E) \geq 1-\eps \quad \text{and} \quad \p_p(\mathscr U) \geq 1-\eps.
\end{equation}
Then, by Harris' inequality and a union bound, we could deduce that any pair of vertices $u$ and $v$ at distance $k R$ satisfy
\begin{equation} \label{eq:vague_extended_bound}
      \p_p( u  \leftrightarrow v ) \geq  \min_{u',v' \in B_{10R}} \p_p( u' \leftrightarrow v' )^2 \cdot [1 - k \eps] - k\eps.
\end{equation}
If we also had a way to use \eqref{eq:vague_extended_bound} to deduce a version of \eqref{eq:vague_E_and_U_whp} at scale $kR$ in place of $R$ (possibly after a small increase of $p$), \emph{and this argument is sufficiently efficient quantitatively}, we might be able to formulate an inductive argument yielding both connection lower bounds and uniqueness of annuli crossings at \emph{all} scales. (Of course such an argument cannot work on e.g.\ the $3$-regular tree.)

This is roughly what is done in \cite{contreras2022supercritical} (although they consider supercritical percolation, so that crossing probabilities do not decay \emph{a priori}), who establish that $\mathscr U_R$ holds with high probability when $G$ has polynomial growth and $R\to \infty$.
Their method, which (following \cite{MR1707339}) they refer to as \emph{orange peeling}\footnote{We do not completely understand the metaphor; perhaps onion peeling would be more appropriate?} and is inspired by the earlier work \cite{MR3634283}, relies on knowing a two-point lower bound within annuli of the form $B_{R + \Delta} \backslash B_{R}$ for all large $R$ and some appropriate $\Delta(R) \ll R$. This information needs to be proven using information at lower scales, so that the argument is a kind of multi-scale induction or coarse-grained renormalization.
The idea is to argue that as two clusters cross the thick annulus $B_{5R} \backslash B_{2R}$, they are very likely to become connected to each other because they have to cross many disjoint annuli of the form $B_{R' + \Delta} \backslash B_{R'}$, and in each such annulus they have a good probability to become connected to each other after sprinkling. If there are enough opportunities for any two crossing clusters to merge then all crossing clusters will have merged by the time they reach the outer sphere. The full argument of \cite{contreras2022supercritical} is highly technical and sophisticated, with the discussion in this paragraph presenting only a simplified cartoon version of some selected parts of their argument.

\medskip

\noindent \textbf{Working with quasi-polynomial growth.}
To run the something like the above ``orange peeling'' method, it is helpful to know that annuli $B_{R_2} \backslash B_{R_1}$ are in some sense well-connected. For example, in the setting of \cite{contreras2022supercritical}, the authors used the fact that for any pair of vertices $u,v$ in the \emph{exposed sphere} $S_{R}^\infty$, which is a certain subset of the usual sphere $S_{R}$, there is a path from $u$ to $v$ that stays within $B_{R+\Delta} \backslash B_{R-\Delta}$. To prove this they used the structure theory of graphs of polynomial growth (i.e.\ the fact that such graphs are finitely presented and are one-ended when they are not one-dimensional). As such, this method does not easily generalize to other graphs. (Indeed, any reasonable connectivity-of-annuli statement cannot hold in complete generality - annuli in regular trees are as poorly connected as sets at a given distance can possibly be.) We will prove that a weak connectivity property of annuli does hold if we assume that a quasi-polynomial growth upper bound $\op{Gr}(R) \leq \exp((\log R)^C)$ holds around the relevant scales. This statement, which we call the \emph{polylog-plentiful tubes condition}, says that for any two sets of vertices $A$ and $B$ crossing a thick annulus $B_{3R} \backslash B_{R}$, we can find many (i.e., at least $(\log R)^{\Omega(1)}$) paths from $A$ to $B$ that are not too long (i.e., have length at most $R (\log R)^{O(1)}$) and that are well-separated from each other (i.e., any two paths in the set have distance at least $(\log R)^{\Omega(1)}$). The proof of this polylog-plentiful tubes condition will exploit a structure vs. randomness dichotomy, where we use the structure theory of \cite{EHStructure} to handle the low-dimensional case and handle the high-dimensional case by building the required disjoint tubes using coupled families of random walks. Once this polylog-plentiful tubes condition is verified, we then show it can be used to push the methods of \cite{contreras2022supercritical} to handle graphs of quasi-polynomial growth. (This requires significant technical changes throughout their entire argument, with the proofs of some intermediate steps being completely different.)


We might hope that the low- and high-growth arguments both imply a common statement similar to \eqref{eq:simple_eg_of_locality_implication}. Unfortunately our induction statement is more involved than this. The issue is that our low-growth argument works with point-to-point connections \emph{within finite sets} such as balls and tubes, while the high-growth argument works directly in infinite volume, and our induction hypothesis must be able to handle oscillation between these two cases. Our actual induction statement, which we explain in detail in \cref{sec:induction_step}, therefore has two parts: A lower bound on the full-space (infinite volume) two-point function, with no restriction on the geometry of the graph at the relevant scale, together with a lower-bound on connection probabilities \emph{inside tubes} that holds only when the \emph{thickness} of the tube happens to belong to a quasi-polynomial growth scale. (The \emph{length} of these tubes are permitted to be quasi-polynomial in their thickness, so that this estimate tells us not just about quasi-polynomial growth scales but also many subsequent larger scales.) Thus, we are able to formulate a multi-scale inductive implication that holds \emph{in all growth regimes} and is strong enough to imply \cref{thm:main}.

\subsection{Organization and overview}


We now outline the structure of the rest of the paper.

\medskip

\noindent \textbf{Section 2:} In this section we review both the classical Aizenman-Kesten-Newman inequality \cite{aizenman1987uniqueness} and its consequences for the ``uniqueness zone'' as derived in \cite{contreras2022supercritical} as well as the \emph{two-ghost inequality} of \cite{hutchcroft2019locality}, which is a kind of infinite-volume version of Aizenman-Kesten-Newman analyzing volumes of clusters rather than their diameters. We also state a useful lemma of \cite{hutchcroft2019locality} that applies this inequality and that leads to Reformulation~\ref{reform3} above.

\medskip

\noindent \textbf{Section 3:} In this section we formulate the multi-scale induction framework used to prove locality, stating the induction step as \cref{prop:complicated_induction_statement}. We then explain how this technical proposition both implies \cref{thm:main} and leads to a new proof of $p_c<1$ for transitive graphs that are not one-dimensional as originally proven in \cite{MR4181032}.

\medskip

\noindent \textbf{Section 4}:
In this section we develop a new method of deducing two-point function lower bounds at a large scale from lower bounds at a smaller scale, which we call \textbf{snowballing}. This is based in part on an idea already present in \cite{duminil2021upper,contreras2022supercritical}, in which one uses Aizenman-Kesten-Newman bounds as an input to Talagrand's sharp threshold theorem  \cite{MR1303654}. Unlike those works, we use the two-ghost formulation of Aizenman-Kesten-Newman from \cite{hutchcroft2019locality} to work directly with volumes of clusters and prove bounds that hold in arbitrary unimodular transitive graphs; this causes the inequalities we derive to become vastly more efficient as the growth of the graph gets larger. The proof of the main snowballing proposition also involves a novel
%
 ghost-based chaining argument to convert this sharp threshold statement into a statement about extending point-to-point connection probabilities. All the arguments in this section work more efficiently as the growth gets larger, but still have content in the polynomial growth case. In Section 2.2, we explain how the snowballing method implies part of the main multi-scale induction step and also easily implies the full locality conjecture for graphs satisfying a mild superpolynomial growth lower bound. (Here the growth assumption is needed only to ensure that the total amount of sprinkling is small when we start at a large scale and induct up to infinity; no further geometric properties of graphs of high growth are used.)
   The tools we develop in Section 2.1 also play an indispensable role in our low-growth arguments in \cref{sec:low_growth_multiscale}. 

\medskip

\noindent \textbf{Section 5}: In this section we establish deterministic geometric features of transitive graphs of quasi-polynomial growth (at some scale) that will be used to analyze percolation on these graphs in \cref{sec:low_growth_multiscale}. The main question addressed is as follows: How can one harness low  -- but not necessarily polynomial -- volume growth to establish that annuli are well-connected? Here, the ``well-connectedness'' of annuli is made precise via what we call the \emph{polylog-plentiful tubes condition}.
The proof that the polylog-plentiful tubes condition holds for graphs of quasi-polynomial growth uses a ``structure vs. expansion'' dichotomy according to whether the \emph{rate of growth} on the relevant scale is low or high dimensional; in the low-dimensional case we employ the structure theory of \cite{breuillard2011structure,MR4253426,EHStructure} while in the high-dimensional case we construct large families of disjoint tubes using certain coupled random walks; the quasi-polynomial growth assumption is used to ensure that we can couple two walks started at distance $n$ to coalesce within distance $n(\log n)^{O(1)}$. (The low dimensional case of the analysis is the only place where we directly use the fact that our graphs are not one-dimensional, where it is needed to ensure that ``exposed spheres'' are well-connected in a certain sense.)


\medskip


\noindent \textbf{Section 6}: Using the polylog-plentiful tubes condition, we run a chaining argument for graphs of quasi-polynomial growth (on some scale) that is inspired by \cite{contreras2022supercritical}. Many changes to their argument are required to deal with this new geometric setting, and our arguments in this section also make use of the snowballing method developed in \cref{sec:snowballing}. We then use the outputs of this analysis to complete the proof of the induction step and hence of \cref{thm:main}.

\medskip


\noindent \textbf{Section 7}: In this section we first briefly sketch how the methods of \cref{sec:snowballing} can be used to prove locality of the \emph{density} of the infinite cluster in the supercritical phase; a significant generalization of this result is proven in full detail (via a different method) in our forthcoming paper \cite{easo2021supercritical2}. We finish by discussing some open problems in \cref{sec:pc_gap}.

\section{Cluster repulsion and the \emph{a priori} uniqueness zone}
\label{sec:AKN}

In this section we review various bounds on the probability that two large clusters either meet at a single edge or both intersect a small ball, together with some important consequences of these inequalities.
These inequalities originate inexplicitly in the work of Aizenman, Kesten, and Newman \cite{aizenman1987uniqueness} and were brought to the wider attention of the community in the work of Cerf \cite{Cerf_2015}. 
These inequalities come in two flavours: finite-volume estimates that work with the \emph{radii} of clusters, and infinite-volume estimates that work with the \emph{volume} of clusters. The first kind of inequality, which is closer to the original vision of \cite{aizenman1987uniqueness,Cerf_2015}, works best under a low-growth assumption and plays an important role in Contreras, Martineau, and Tassion's analysis of percolation on transitive graphs of polynomial growth \cite{contreras2022supercritical,MR4529920}, while the second kind of inequality, introduced in \cite{hutchcroft2019locality} and known as \emph{two-ghost inequalities}, become \emph{more} useful the faster the graph grows. We will also review the argument of \cite{hutchcroft2019locality} which allows us to deduce locality of $p_c$ from uniform bounds on the two-point function using the two-ghost inequality.

\subsection{Aizenman-Kesten-Newman and the \emph{a priori} uniqueness zone}
\label{subsec:AKN}

We now review the finite-volume Aizenman-Kesten-Newman inequality as presented in \cite{Cerf_2015,contreras2022supercritical}.

\medskip

\noindent
\textbf{Finite-volume two-arm estimates.}
Given $m,n \in (0,\infty)$ with $m \leq n$, we define $\op{Piv}[m,n]$ to be the event that there exist two distinct clusters in $\omega \cap B_n$ that each intersect both spheres $S_n$ and $S_m$. (That is, $\op{Piv}[m,n]$ is the event that there is more than one cluster crossing the annulus from $m$ to $n$.) The following lemma is a minor variation on \cite[Proposition 4.1]{contreras2022supercritical}.



\begin{prop} \label{prop:two_arm_bound_in_box}
For each $0<\eps < 1/2$, $0 < \eta < 1$, and $d\geq 1$ there exists a constant $C=C(\eps,\eta,d)$ such that if $G$ is a connected, transitive graph of vertex degree $d$ then
      \[
            \p_{p} \bra{ \op{Piv}[1,n] } \leq C \left[\frac{\log \op{Gr}(n)}{n}\right]^{\frac{1}{2}-\eps}.
      \]
      for every $p \in [\eta,1]$ and $n \geq 1$.
\end{prop}

To deduce this proposition from the proof of \cite[Proposition 4.1]{contreras2022supercritical}, we will require the following elementary fact about the growth of balls. Here $\partial B_m(o)$ denotes the set of edges that have one endpoint in $B_m(o)$ and the other endpoint in the complement of $B_m(o)$.

\begin{lem} \label{lem:pigeonhole_for_volume_growth}
For each $d\geq 1$ there exists a constant $C_d$ such that the following holds. Let $G$ be a graph of maximum vertex degree at most $d$, and let $o$ be a vertex of $G$. For each integer $n \geq 1$  there exists an integer $n \leq m \leq 2n-1$ such that
      \[
            \frac{\abs{\partial B_m(o)}}{\abs{B_m(o)}} \leq  \frac{C_d}{n} \log \abs{B_{2n}(o)}.
      \]
\end{lem}

\begin{proof}[Proof of \cref{lem:pigeonhole_for_volume_growth}]
Write $\op{Gr}(m) := \abs{B_m(o)}$ for every $m \in \mathbb N$. Since $G$ has vertex degrees bounded above by $d$, we have that $\abs{\partial B_m(o)} \leq d(\op{Gr}(m+1)  - \op{Gr}(m))$ and hence that $\abs{\partial B_m(o)}/\op{Gr}(m) \leq d( \op{Gr}(m+1)/\op{Gr}(m) - 1)$. It follows that
\[
\sum_{m=n}^{2n-1} \frac{\abs{\partial B_m(o)}}{|B_m(o)|} \leq d \sum_{m=n}^{2n-1}\bra{ \frac{\op{Gr}(m+1)}{\op{Gr}(m)} - 1}.
\]
Now, since we also have that $\frac{\op{Gr}(m+1)}{\op{Gr}(m)} \leq d$, there exists a constant $C=C_d$ such that 
\[
\frac{\op{Gr}(m+1)}{\op{Gr}(m)} - 1 \leq C_d \log \frac{\op{Gr}(m+1)}{\op{Gr}(m)}
\]
for every $m\geq 0$. As such, it follows that
\[
\sum_{m=n}^{2n-1} \frac{\abs{\partial B_m(o)}}{|B_m(o)|} \leq d C_d \sum_{m=n}^{2n-1}  \log \frac{\op{Gr}(m+1)}{\op{Gr}(m)} = dC_d \log \frac{\op{Gr}(2n)}{\op{Gr}(n)}
\]
which is easily seen to imply the claim.
%
\end{proof}

\begin{proof}[Proof of \cref{prop:two_arm_bound_in_box}]
This follows by exactly the same proof as \cite[Proposition 4.1]{contreras2022supercritical} except that we use our \cref{lem:pigeonhole_for_volume_growth} instead of their Lemma 4.2 (which is the same estimate specialized to the polynomial growth setting).
\end{proof}

\noindent \textbf{The \emph{a priori} uniqueness zone.} We now discuss how two-arm bounds at a single edge can be used to deduce bounds on the probability of having multiple clusters crossing an annulus. The following lemma, essentially due to Cerf \cite{Cerf_2015}, lets us apply \cref{prop:two_arm_bound_in_box} to bound the probability that there are two distinct crossings of an annulus.

\begin{lem} [{\cite[Lemma 6.2]{contreras2022supercritical}}] \label{lem:nearby_clusters_using_two_point}
Let $G$ be a connected transitive graph. Then
\begin{equation}
\label{eq:Cerf}
      \p_p\bigl( \op{Piv}[r,n] \bigr) \leq \p_p\bigl( \op{Piv}[1,n/2] \bigr) \cdot \frac{\abs{S_r}^2 \cdot \op{Gr}(m) }{ \min_{a,b \in S_r} \p_p \bigl( a \xleftrightarrow{\,B_{m}\;} b \bigr) }
\end{equation}
for every $r,m,n \in (1,\infty)$ with $ r \leq m \leq n/2$ and every $p\in (0,1)$.
\end{lem}
\begin{rk}
The authors of \cite{contreras2022supercritical} stated this lemma in their context of infinite graphs with polynomial growth. However, their proof works exactly the same for arbitrary connected transitive graphs. (The statement given in \cite{contreras2022supercritical} has $B_{2m}$ in place of $B_{m}$ -- which is slightly stronger -- but this appears to be a typo.) 
\end{rk}

Since any geodesics from $o$ to two vertices $a,b\in B_m$ are both open with probability at least  $p^{2m}$ and the growth satisfies the trivial upper bound $\op{Gr}(m)\leq d^{m+1}$, the quantity multiplying the probability $\p_p\bigl( \op{Piv}[1,n/2] \bigr)$ on the right hand side of \eqref{eq:Cerf} is at most exponential in $m$ when $p$ is bounded away from~$0$.  \cref{prop:two_arm_bound_in_box,lem:nearby_clusters_using_two_point} therefore have the following immediate corollary.


\begin{cor}[Trivial \emph{a priori} uniqueness zone] \label{cor:trivial_a_priori_uniqueness_zone}
      Let $G$ be a connected, unimodular transitive graph with vertex degree $d$ and let $\eps,\eta \in (0,1)$. There exist positive constants $c=c(d,\eta,\eps)$, $C=C(d,\eta,\eps)$, and $n_0=n_0(d,\eta,\eps)$ such that
      \[
            \p_p\bigl( \op{Piv}[ c\log n , n ] \bigr) \leq C \left[ \frac{\log \op{Gr}(n)}{n} \right]^{1/2-\eps}
      \]
      for every $n\geq n_0$ and $p\in [\eta,1]$.
\end{cor}

In other words, if $G$ has subexponential growth then the probability of having two distinct crossings of the annulus from $c\log n$ to $n$ is always small for an appropriately small constant $c$. (The restriction that $p$ is small could be removed by noting that it is very unlikely for there to be \emph{any} crossings of the annulus when $p\leq 1/d$.)

\subsection{The two-ghost inequality}
\label{subsec:two_ghost}

We now recall the \emph{two-ghost inequality} of \cite{hutchcroft2019locality}, a form of the Aizenman-Kesten-Newman bound that holds for any unimodular transitive graph, without any growth assumptions. (This version of the bound does \emph{not} imply uniqueness of the infinite cluster since it requires at least one of the clusters to be finite.)
%
%
%
%
%
%
Let $G=(V,E)$ be a graph. For each edge $e$ of $G$ and $n\geq 1$ we define $\mathscr{S}_{e,n}$ to be the event that $e$ is closed and that the endpoints of $e$ are in distinct clusters, each of which has volume at least $n$ and at least one of which is finite.

\begin{thm}[Two-ghost inequality]
\label{thm:two_ghost}{}
 Let $G$ be a unimodular transitive graph of degree $d$. There exists a constant $C_d$ such that
\begin{equation}
\label{eq:twobodyunimod}
\p_p(\mathscr{S}_{e,n})\leq C_d \left[\frac{1-p}{p n}\right]^{1/2}
\end{equation}
for every $e\in E(G)$, $p\in(0,1]$ and $n\geq 1$.
\end{thm}

\begin{proof}
This is an immediate consequence of \cite[Corollary 1.7]{hutchcroft2019locality}: The statement given there concerns the edge volume rather than the vertex volume, but this only makes the statement stronger.
\end{proof}

\cref{thm:two_ghost} has the following useful consequence concerning the probability that two large distinct clusters come close to one another; this  lets us convert bounds on the two-point function into bounds on the tail of the volume and underlies the fact that Reformulation~\ref{reform3} implies the unimodular case of \cref{thm:main}.

\begin{lem} \label{lem:nearby_clusters_pigeonhole}
      Let $G$ be an infinite, connected, unimodular transitive graph with vertex degree $d$. There exists a constant $C_d$ such that
      \[
            \p_p( |K_u|\geq n \text{ and } |K_v|\geq n \text{ but } u \nleftrightarrow{v} ) \leq C_d \cdot d(u,v) p^{-d(u,v)-1} n^{-1/2}.
      \]
      for every $u,v\in V(G)$, $n\geq 1$, and $p < p_c(G)$.
\end{lem}

\begin{proof}
This follows  from \cref{thm:two_ghost} by the same argument used to prove equation (4.2) of \cite{hutchcroft2019locality}. (Again, the only difference is that we are using vertex volumes rather than edge volumes.) The quantity $d(u,v) p^{-d(u,v)-1}$ arises as a simple upper bound on $\left((1-p)/p\right)^{1/2}\sum_{i=1}^{d(u,v)}p^{-i+1}$.
\end{proof}

\section{The multi-scale induction step}
\label{sec:induction_step}

In this section we state our key technical proposition, \cref{prop:complicated_induction_statement}, which encapsulates the multi-scale induction used to prove \cref{thm:main}. We introduce relevant definitions in \cref{subsec:induction_definitions}, give the statement of the proposition in \cref{subsec:induction_statement}, and explain how it implies our main theorem in \cref{subsec:induction_deduction}. In \cref{subsec:induction_deduction} we also explain how \cref{prop:complicated_induction_statement} yields a new proof of the fact that $p_c < 1$ for all infinite, connected, transitive graphs that are not one-dimensional. 

\subsection{Definitions}
\label{subsec:induction_definitions}

In this section we establish the notation necessary to state the main multi-scale induction proposition in \cref{subsec:induction_statement}.

\medskip

\noindent
\textbf{Natural coordinates for sprinkling.} We define the \textbf{sprinkling function} $\sprinkle:(0,1) \times \mathbb{R} \to(0,1)$ by
\[
\sprinkle(p;\lambda)=1-(1-p)^{e^\lambda} \qquad \text{so that} \qquad (1-\sprinkle(p;\lambda))=(1-p)^{e^\lambda}.
\]
The sprinkling functions $(\sprinkle(\,\cdot\,;\lambda))_{\lambda \in \mathbb{R}}$ form a semigroup in the sense that 
\[\sprinkle(p;\lambda+\mu)=\sprinkle(\sprinkle(p;\lambda);\mu)\]
for every $\lambda,\mu\in \mathbb{R}$ and $p\in (0,1)$.
For each $0<p,q <1$ we define
\[
\delta(p,q) = \log\left[\frac{\log (1-\max\{p,q\})}{\log (1-\min\{p,q\})}\right], \qquad \text{so that}\qquad \max\{p,q\}=\sprinkle(\min\{p,q\};\delta(p,q)).
\]
Note that if $p\geq 1/d$, as we will assume throughout most of the paper, then for each $D<\infty$ there exists a positive constant $c=c(d,D)$ such that
\begin{equation} \label{eq:natural_sprinkling_lower_bound}
(1-\sprinkle(p;\delta))=(1-p)^{e^\delta-1}(1-p) \leq \left(\frac{2d-1}{2d}\right)^\delta (1-p) \leq (1-c \delta) (1-p)
\end{equation}
for every $0 \leq \delta \leq D$, so that Bernoulli bond percolation with parameter $\sprinkle(p;\delta)$ stochastically dominates the independent union of a Bernoulli-$p$ bond percolation configuration and a Bernoulli-$(c\delta)$ bond percolation configuration. (Note however that  $1-\sprinkle(p;\delta)$ is much smaller than $(1-c\delta)(1-p)$ when $p$ is close to $1$.)

\medskip

\noindent
\textbf{The corridor function.}
Given a path $\gamma$, we write $\op{len}(\gamma)$ for its length. Given a path $\gamma$ and some $r \in (0,\infty)$, we define $B_{r}(\gamma) := \bigcup_{i = 0}^{\op{len}(\gamma)} B_r(\gamma_i)$, and call a set of this form a \textbf{tube}. We refer to $\op{len}(\gamma)$ and $r$ as the \textbf{length} and \textbf{thickness} of the tube respectively. (Note that these parameters depend on the choice of representation of the tube $B_{r}(\gamma)$, and are not  determined by the tube as a set of vertices.) Following \cite{contreras2022supercritical}, we define the \textbf{corridor function} by
\[
      \kappa_p(m,n) := \inf_{\gamma : \op{len}(\gamma) \leq m } \p_p\Bigl( \gamma_0 \xleftrightarrow{B_n(\gamma)} \gamma_{\op{len}(\gamma) } \Bigr)
\]
for each $p \in (0,1)$ and $n,m\geq 1$, so that $\kappa_p(m,n)$ measures the difficulty of connecting points within tubes of thickness $n$ and length at most $m$. We may also take $n=\infty$ in the definition of the corridor function, where the restriction for connections to lie in a tube disappears and we have simply that
\[
\kappa_p(m,\infty) = \kappa_p(m)= \inf\{ \p_p(x \leftrightarrow y) : d(x,y) \leq m\}.
\]
Note that the corridor function $\kappa_p(m,n)$ is increasing in $p$ and $n$ and decreasing in $m$.

\medskip

\noindent 
\textbf{Low growth scales.} 
Given a transitive graph $G$ and a parameter $D>0$,  we define the set of \textbf{low growth scales} to be
\[
\mathscr{L}(G,D) = \Bigl\{n \geq 1: \log \op{Gr}(m) \leq (\log m)^{D} \text{ for all } m \in [n^{1/3} , n]\Bigr\},
\]
so that
\[
\Bigl\{n \geq 1: \log \op{Gr}(n) \leq 3^{-D}(\log n)^{D}\Bigr\} \subseteq \mathscr{L}(G,D) \subseteq \Bigl\{n \geq 1: \log \op{Gr}(n) \leq (\log n)^{D}\Bigr\}.
\]
(We will sometimes call these \textbf{quasi-polynomial growth scales} since the function $\exp[(\log x)^C]$ is sometimes known as a quasi-polynomial.)
For the purposes of the proof of the main theorem, we will apply this definition only with the (somewhat arbitrary) choice of constant $D=20$.

\medskip

\noindent
\textbf{The burn-in sprinkle.} The first step of our induction will have a different form to the others, which necessitates a possibly larger amount of sprinkling. We now introduce notation describing this initial amount of sprinkling. Given a transitive graph $G$, $p\in (0,1)$, and $m\geq 1$ define 
\[
b(m) = b(m,p)=\max\Bigl\{b \in \mathbb{N} : 1 \leq b \leq \frac{1}{8}m^{1/3} \text{ and } \p_{p}(\op{Piv}[4b,m^{1/3}]) \leq (\log m)^{-1}\Bigr\},\]
setting $b(m)=0$ if the set being maximized over is empty, 
%
and define the \textbf{burn-in} to be
\begin{multline*}
\burnin(n,p)=\burnin(G,n,p) \\=\max\Biggl\{\left(\frac{ \log \log m}{\min\{\log m,\log \op{Gr}(b(m))\}}\right)^{1/4} : m \in \mathscr{L}(G,20) \cap \Bigl[ (\log n)^{1/2} , n \Bigr] \Biggr\},
\end{multline*}
setting $\burnin(n,p)=0$ if the set being maximized over is empty and setting $\burnin(n,p)=\infty$ if there exists $m$ belonging to the set for which $b(m)\leq 1$. (If $b(m)>1$ then $\log \op{Gr}(b(m))$ and $\log m$ are both positive.) Note that $\burnin(G,n,p)$ is determined by the ball of radius $n$ in $G$, so that two graphs whose balls of this radius are isomorphic have the same value of $\burnin(n,p)$ for each $p\in (0,1)$.


\subsection{Statement of the induction step}
\label{subsec:induction_statement}

We are now ready to state our main multi-scale induction proposition. We recall that, when applied to logical propositions, the symbol ``$\vee$'' means ``or'' while the symbol ``$\wedge$'' means ``and''. The condition $n_0 \geq 16$ appearing in the proposition ensures that $\log \log n_0 \geq 1$.

\begin{prop}[The main multi-scale induction step] \label{prop:complicated_induction_statement}
For each $d\in \mathbb{N}$ there exist constants $K=K(d)$ and $N=N(d)\geq 16$ such that the following holds.
Let $G$ be an infinite, connected, unimodular transitive graph with vertex degree $d$ that is not one-dimensional, let $p_0 \in (0,1)$, and let $n_0 \geq 16$.
Let $n_{-1}=(\log n_0)^{1/2}$, let 
\[
 \delta_0 = \frac{1}{(\log \log n_0)^{1/2}} + K \cdot \burnin(n_0,p_0),
\]
define sequences $(n_i)_{i\geq 1}$ and $(\delta_i)_{i\geq1}$ recursively by
\[
      n_i := \exp((\log n_{i-1})^9)=  \exp^{\circ 3}\left( \log^{\circ 3}(n_0) + i \log 9 \right) \quad \text{and} \quad \delta_i := (\log \log n_i)^{-1/2} = 3^{-i}(\log \log n_0)^{-1/2},
\]
and let $(p_i)_{i\geq 1}$ be an increasing sequence of probabilities satisfying $p_{i+1}\geq \sprinkle(p_i;\delta_i)$ for each $i\geq 0$.
For each $i\geq 0$ define the statement
\begin{align*}\text{\emph{\textsc{Full-Space}}}(i) &= \Biggl(\p_{p_i} ( u \leftrightarrow v ) \geq \exp\left[-(\log \log n_{i})^{1/2}\right] \text{ for all $u,v \in B_{n_i}$}\Biggr) 
\intertext{and for each $i\geq 1$ define the statement}
\text{\emph{\textsc{Corridor}}}(i) &= \Biggl(\kappa_{p_{i}}( e^{[\log m]^{10}}, m ) \geq \exp\left[-(\log \log n_{i})^{1/2}\right] \text{ for every $m \in \mathscr{L}(G,20) \cap [n_{i-2},n_{i-1}]$}\Biggr).
\end{align*}
If $n_0 \geq N$, $p_0 \geq 1/d$, and $\delta_0 \leq 1$ then the implications
      \begin{align}
\upsc{Full-Space}(0)  &\implies \Bigl[  \upsc{Corridor}(1) \vee (p_{1} \geq p_c) \Bigr],
\tag{C$_0$}\label{implication:corridor0}\\
            \Bigl[\text{\emph{\textsc{Full-Space}}}(i) \,\wedge\, \bigwedge _{k =1}^{i} \text{\emph{\textsc{Corridor}}}(k) \Bigr] &\implies \Bigl[  \text{\emph{\textsc{Corridor}}}(i+1) \vee (p_{i+1} \geq p_c) \Bigr], \quad \text{and} \label{implication:corridor} \tag{C}\\
             \Bigl[ \text{\emph{\textsc{Full-Space}}}(j) \wedge \text{\emph{\textsc{Corridor}}}(j+1) \Bigr] &\implies \Bigl[ \text{\emph{\textsc{Full-Space}}}(j+1) \vee (p_{j+1} \geq p_c) \Bigr] \label{implication:full-space} \tag{F}
      \end{align}
hold for every $i\geq 1$ and $j \geq 0$.
\end{prop}

\begin{rk}
Note that $\upsc{Corridor}(1)$ has a significantly different form than $\upsc{Corridor}(i)$ for $i\geq 2$ since $n_{-1}$ is much smaller than the natural extrapolation of the sequence $(n_i)_{i\geq 0}$ to $i=-1$. 
\end{rk}


\begin{rk}\label{remark:p_lower_bound_redundant}
The condition $p\geq 1/d$ appearing in the hypotheses of \cref{prop:complicated_induction_statement} is in fact redundant: An elementary path counting argument yields that 
\begin{multline*}\p_p(u \leftrightarrow v) \leq \p_p(\text{there is a simple open path of length at least $d(u,v)$ starting from $u$}) 
\\\leq \frac{d}{d-1} \cdot \frac{1}{1-p(d-1)} \cdot (p(d-1))^{d(u,v)}\end{multline*}
for every $p<1/(d-1)$ and $u,v\in V(G)$, so that if $n_0$ is sufficiently large and $\upsc{Full-Space}(0)$ holds then $p\geq 1/d$. We include this redundant assumption anyway to clarify the structure of the proof.
\end{rk}

\begin{rk}
In the statement $\upsc{Corridor}(i)$, the tubes that arise in the relevant  corridor function $\kappa_{p_{i}}( e^{[\log m]^{10}}, m )$ have \emph{thickness} given by the low-growth scale $m$, but can have length equal to the much larger value $e^{[\log m]^{10}}$. As such, the statement $\upsc{Corridor}(i)$ gives us strong control of percolation not just at low-growth scales but at a large range of scales above each low-growth scale. In particular, provided $n_i$ is sufficiently large that $e^{[\log (n_i^{1/3})]^{10}}\geq e^{(\log n_i)^{9}} = n_{i+1}$, 
the implication \eqref{implication:full-space} holds trivially whenever $n_i \in \mathscr{L}(G,20)$.
\end{rk}

\begin{rk}
The ``$\vee (p_{i+1}\geq p_c)$'' that appears on the right hand side of the implications \eqref{implication:corridor0}, \eqref{implication:corridor}, and \eqref{implication:full-space} can be removed if one assumes that $G$ is amenable, or if one works with "wired" connections as discussed in \cref{sec:continuity_of_density}. This would be useful if one wished to use (a modification of) our methods to study the geometry of the infinite cluster in graphs of low growth, extending the results of \cite{contreras2022supercritical,hutchcroft2023transience} to this setting. 
\end{rk}

Most of the paper is dedicated to proving \cref{prop:complicated_induction_statement}: 
The implication \eqref{implication:full-space} is proven in \cref{sec:snowballing} while the implications \eqref{implication:corridor0} and \eqref{implication:corridor} are proven in \cref{sec:tubes,sec:low_growth_multiscale}.

\subsection{Deduction of the main theorem from the induction step}
\label{subsec:induction_deduction}

In this section we deduce our main theorem, \cref{thm:main}, from \cref{prop:complicated_induction_statement}.
We write $p_\infty=\lim_{i\to\infty}p_i$ for the limit of the parameters $(p_i)_{i\geq 0}$ defined recursively by $p_{i+1}=\sprinkle(p_i;\delta_i)$, so that
\[
p_\infty = \sprinkle \Bigl(p_0;\sum_{i=0}^\infty \delta_i\Bigr) \leq \sprinkle\Bigl(p_0;\;2(\log \log n_0)^{-1/2}+K\cdot \burnin(n_0,p_0)\Bigr).
\]
We will apply \cref{prop:complicated_induction_statement} via the following corollary, which is a consequence of \cref{prop:complicated_induction_statement} together with the sharpness of the phase transition.

\begin{cor}
\label{cor:carrying_out_induction}
Let $d\geq 1$ and let $K=K(d)$ and $N=N(d)$ be the constants from \cref{prop:complicated_induction_statement}. Let $G$ be an infinite, connected, unimodular transitive graph with vertex degree $d$ that is not one-dimensional. 
 Then the implication
\begin{multline*}
\left(\left[\p_{p}(u \leftrightarrow v) \geq e^{-(\log \log n)^{1/2}}  \text{ for every $u,v \in B_{n}$} \right]\wedge \left[ \delta_0 \leq 1 \right]\right)
\\
\implies \left[p_c(G) \leq \sprinkle\Bigl(p;\;2(\log \log n)^{-1/2}+K\cdot \burnin(n,p)\Bigr) \right].
\end{multline*}
holds for every $n\geq N$ and $p \geq 1/d$.
\end{cor}

\begin{proof}[Proof of \cref{cor:carrying_out_induction} given \cref{prop:complicated_induction_statement}]
It suffices to prove that $p_c\leq p_\infty$ whenever $n\geq N$ and $p\geq 1/d$ are such that
 $\p_{p}(u \leftrightarrow v) \geq e^{-(\log \log n)^{1/2}}$  for every $u,v \in B_{n}$ and $\delta_0 \leq 1$. Fix one such $n$ and~$p$. Set $n_0=n$ and define $(n_i)_{i\geq 0}$ as in \cref{prop:complicated_induction_statement}. Since $\upsc{Full-Space}(0)$ holds by assumption, it follows from \cref{prop:complicated_induction_statement} that either $p_i \geq p_c$ for some $i\geq 1$ or $\upsc{Full-Space}(i)$ and $\upsc{Corridor}(i)$ hold for every $i\geq 1$. In the former case we may trivially conclude that $p_c \leq p_\infty$, while in the latter case we have that
\begin{equation}
\label{eq:p_infty_subexponential}
\p_{p_\infty}(u\leftrightarrow v) \geq e^{-(\log \log n_i)^{1/2}} \qquad \text{for every $i\geq 0$ and every $u,v\in B_{n_i}$.}
\end{equation}
On the other hand, it follows from the sharpness of the phase transition \cite{MR852458,MR874906} that for each $p<p_c$ there exists a positive constant $c_p$ such that $\p_p(u \leftrightarrow v) \leq e^{-c_p d(u,v)}$ for every $u,v\in V(G)$. This is incompatible with the (very) subexponential lower bound \eqref{eq:p_infty_subexponential}, so that  $p_\infty \geq p_c$ in this case also.
\end{proof}

We now apply \cref{cor:carrying_out_induction} to prove \cref{thm:main}. The proof we give here will rely on the results of both \cite{hutchcroft2019locality,hutchcroft2020nonuniqueness} (to deal with the nonunimodular case) and \cite{contreras2022supercritical,MR4529920} (to deal with the case that the limit has polynomial growth). We remark that the quantitative proof of the theorem given in \cite{EH_quantitative} yields a completely self-contained and ``uniform in the graph'' deduction of locality from \cref{prop:complicated_induction_statement} in the unimodular case that does not rely on the results\footnote{Of course many of the proof techniques remain closely inspired by these works!} of \cite{contreras2022supercritical,MR4529920}. (Doing this requires non-trivial bounds on the burn-in which can be avoided in the case that the limit has superpolynomial growth as we will see below.)

\begin{proof}[Proof of \cref{thm:main} given \cref{prop:complicated_induction_statement}]
Let $(G_m)_{m\geq 1}$ be a sequence of infinite, connected, transitive graphs converging locally to some infinite transitive graph $G$, and suppose that the graphs $G_m$ all have superlinear growth. We want to prove that $p_c(G_m)\to p_c(G)$. If $G$ is nonunimodular the claim follows from the results of \cite{hutchcroft2019locality,hutchcroft2020nonuniqueness} (specifically \cite[Theorem 5.6]{hutchcroft2019locality}), while if $G$ has polynomial growth the result follows from the main result of \cite{MR4529920}.  Thus, we may assume that $G$ is unimodular and has superpolynomial growth. Since the set of nonunimodular graphs is both closed and open in $\mathcal{G}$ \cite[Corollary 5.5]{hutchcroft2019locality}, we may assume that the graphs $(G_n)_{n\geq 1}$ are all unimodular. We may also assume that the graphs $G_n$ and $G$ all have the same vertex degree $d$. Let $N_1=N_1(d)$ and $K=K(d)$ be the constants from \cref{prop:complicated_induction_statement}.

\medskip

Suppose for contradiction that $p_c(G_m)$ does not converge to $p_c(G)$.
Since $p_c$ is lower semi-continuous (\cite[\S 14.2]{Gabor} and \cite[p.4]{duminil2016new}), we have that $\liminf_{m\to\infty} p_c(G_m) \geq p_c(G)$. Thus, by taking a subsequence, we may assume that $\inf_{m\geq 1} p_c(G_m) = p_* > p_c(G)$. Let $p_0=(p_c(G)+p_*)/2$ so that $p_c(G)<p_0<p_*$. For each $n\geq 1$, let $m(n)$ be minimal such that the balls of radius $n$ are isomorphic in $G_m$ and $G$ for all $m\geq m(n)$. Let $c > 0$ be the constant from \cref{cor:trivial_a_priori_uniqueness_zone}. Since $G$ has superpolynomial growth, we have by \cite{MR623534,MR735714} that
\[
\lim_{n\to\infty} \frac{\log \log n}{\log \op{Gr}(c\log n;G)} = 0,
\] 
where $\op{Gr}(m;G)$ denotes the volume of the ball of radius $m$ in $G$, and it follows from \cref{cor:trivial_a_priori_uniqueness_zone} that
\begin{equation} \label{eq:uniform_convergence_of_burnin}
\lim_{n\to\infty}\sup_{p\in [1/d,1]} \sup_{m \geq m(n)} \burnin(G_m,n,p) = 0.
\end{equation}
In particular, there exists $N_2\geq N_1$ (depending on the superpolynomial graph $G$ and the sequence $(G_m)$) such that if $n_0 \geq N_2$ then $(\log \log n_0)^{-1/2}+K \cdot \burnin(G_m,n_0,p_0) \leq 1$ and
\[
\sprinkle(p_0;\,2(\log \log n_0)^{-1/2}+K \cdot \burnin(G_m,n_0,p_0)) < p_*
\]
for every $m\geq m(n_0)$. Thus, it follows from \cref{cor:carrying_out_induction} that for every $n_0 \geq N_2$ and $m\geq m(n_0)$ there exist vertices $u$ and $v$ in the ball of radius $n_0$ in $G_m$ such that
\[
\p_{p_0}^{G_m}(u \leftrightarrow v) < \exp\left[-(\log \log n_0)^{1/2}\right].
\]
Applying \cref{lem:nearby_clusters_pigeonhole}, we deduce that there exists a constant $C_1$ such that 
\[
\p_{p_0}^{G_m}(|K|\geq k)^2 \leq C_1 n_0 p_0^{-2n_0} k^{-1/2} +  \exp\left[-(\log \log n_0)^{1/2}\right]
\]
for every $n_0 \geq N_2$, $m\geq m(n_0)$, and $k \geq 1$.
Taking $n_0= \lceil c_2 \log k \rceil$ for an appropriately small constant $c_2$ (which makes the first term $O(k^{-1/4})$, say, and hence of lower order than the second term), it follows that there exist positive constants $C_2$ and $c_3$ such that
\begin{equation}
\label{eq:almost_done_non_quantitative_locality}
\p_{p_0}^{G_m}(|K|\geq k) \leq C_2 \exp\left[-c_3 (\log \log \log k)^{1/2}\right]
\end{equation}
for every $n_0\geq N_2$ and $m\geq m(n_0)$. Since $p_0 > p_c(G)$, the probability $\p_{p_0}^G(o\leftrightarrow \infty)$ is positive and it follows from \eqref{eq:almost_done_non_quantitative_locality} there exist $k_0$ and $m_0$ such that
\begin{equation}
\label{eq:almost_done_non_quantitative_locality2}
\p_{p_0}^{G_m}(|K|\geq k_0) \leq \frac{1}{2}\p_{p_0}^G(o \leftrightarrow \infty)
\end{equation}
for every $m\geq m_0$. On the other hand, if $m$ is sufficiently large that balls of radius $k_0$ are isomorphic in $G_m$ and $G$ then
\[
\p_{p_0}^{G_m}(|K|\geq k_0) \geq 
\p_{p_0}^{G_m}(o \leftrightarrow B_{k_0}^c) = \p_{p_0}^{G}(o \leftrightarrow B_{k_0}^c) \geq
\p_{p_0}^{G}(o \leftrightarrow \infty),
\]
which contradicts the upper bound of \eqref{eq:almost_done_non_quantitative_locality2}.
\end{proof}

Let us now explain how \cref{prop:complicated_induction_statement} yields a new proof of the $p_c<1$ theorem as originally established in \cite{MR4181032}. Recall that $\mathcal G^*$ is the space of all infinite, connected, transitive graphs that are not one-dimensional.

\begin{thm}\label{thm:p_c<1}
    Every graph $G \in \mathcal G^*$ satisfies $p_c(G) < 1$.
\end{thm}

\begin{proof}[Proof of \cref{thm:p_c<1} given \cref{prop:complicated_induction_statement}]
If $G$ has exponential growth then sharpness of the phase transition easily implies that $p_c \leq (\lim_{r\to\infty} \op{Gr}(r)^{-1/r})<1$ (see also \cite{hutchcroft2016critical,lyons1995random}). Since every nonunimodular transitive graph is nonamenable and therefore has exponential growth, it suffices to consider the case that $G$ is unimodular. On the other hand, if $G$ has polynomial growth, then it is well-known that $p_c(G) < 1$ follows from the fact that $p_c(\mathbb Z^2) < 1$ and the structure theory of transitive graphs of polynomial growth as explained in detail in \cite[Section 3.4]{hutchcroft2021nontriviality}.

\medskip

We now consider the case that $G$ is unimodular and has superpolynomial growth.
  Let $d$ denote the vertex degree of $G$ and let $N_1=N_1(d)$ and $K=K(d)$ be the constants from \cref{prop:complicated_induction_statement}. It follows by the same reasoning as we gave for \eqref{eq:uniform_convergence_of_burnin} that for every $\eta > 0$ there exists a constant $M(d,\eta)$ such that $b(m,p) \leq \eta$ 
for every $m\in \mathscr{L}(G,20)$ with $m \geq M$ and every $p \geq 1/d$. 
%
%
Since we also have trivially that $\op{Gr}(n)\geq n$ for every $n\geq 1$ since $G$ is infinite, it follows that there exists a constant $N_2=N_2(d)$ such that
\begin{equation}
\label{eq:finite_burn}
(\log \log n)^{-1/2} + K \cdot \burnin(n,p) \leq 1 \qquad \text{ for every $p \geq 1/d$ and $n \geq N_2$.}
\end{equation}
Let $n_0=n_0(d)=N_1\vee N_2$. Since $\p_p(u \leftrightarrow v) \geq p^{d(u,v)}$ for every $u,v\in V$ and $p\in [0,1]$, there exists a constant 
\[
p_0 = p_0(d)=\frac{1}{d} \vee \exp\left(-\frac{(\log \log n_0)^{1/2}}{ n_0} \right)
\]
satisfying
$1/d \leq p_0<1$ such that $\p_{p_0}(u \leftrightarrow v) \geq \exp(- (\log \log n_0)^{1/2})$ for every $u,v\in B_{n_0}$. Since $n_0 = N_1 \vee N_2$, it follows from \cref{cor:carrying_out_induction} and \eqref{eq:finite_burn} that 
\[
p_c(G) \leq \sprinkle\Bigl(p_0;2(\log \log n_0)^{-1/2}+ K \burnin(n_0,p_0)\Bigr) \leq \sprinkle\bigl(p_0;2 \bigr).
\]
The claim follows since the right hand side is strictly less than one.
\end{proof}



\section{Making connections via sharp threshold theory}
\label{sec:snowballing}

In this section we describe a powerful new way to extend point-to-point connection lower bounds from one scale to another, which we call the ``snowballing method''. We develop this method in \cref{subsec:snowballing}, then apply it to prove the implication \eqref{implication:full-space} of \cref{prop:complicated_induction_statement} in \cref{subsec:high_growth_proof}. In \cref{subsec:high_growth_proof} we will also explain how the method allows us to conclude the proof of locality for unimodular graphs satisfying a mild uniform superpolynomial growth assumption.


\subsection{Snowballing}
\label{subsec:snowballing}

We now begin to develop the snowballing method. This method is primarily encapsulated through the following proposition, whose proof is the main goal of this section, but the intermediate lemmas used in its proof can be used to prove results of indepedent interest as discussed in \cref{sec:joint_continuity}.
Given (not necessarily finite) non-empty sets of vertices $A$, $B$, and $\Lambda$ in a graph $G=(V,E)$ and a parameter $p \in [0,1]$, we define
\[
      \tau_p^{\Lambda}(A,B) := \min_{\substack{a \in A \\ b \in B}} \p_p\left( a \xleftrightarrow{\Lambda} b \right),
\]
where we recall that $\{a \xleftrightarrow{\Lambda} b\}$ denotes the event that $a$ is connected to $b$ by an open path all of whose vertices belong to $\Lambda$.
We will also write $\tau_p^{\Lambda}(A) := \tau_p^{\Lambda}(A,A)$ and $\tau_p(A,B) := \tau_p^{V}(A,B)$.
We also use the notion of distance $\delta(p,q)$ between two parameters $p,q\in (0,1)$ as defined in \cref{sec:induction_step}.


\begin{prop}[Snowballing] \label{prop:snowballing}
For each $d\geq 1$ and $D<\infty$ there exist positive constants $c_1=c_1(D)$ and $h_0=h_0(d,D)$, and universal positive constants $c_2$ and $c_3$ such that the following holds. Let $G=(V,E)$ be a unimodular transitive graph with vertex degree $d$, let $A_1,\ldots,A_n$ be non-empty sets of vertices in $G$, and suppose that $0<p_1<p_2 <1$ are such that there is at most one infinite cluster $\p_{p}$-almost surely for every $p\in [p_1,p_2]$.
Let $h\geq (\min_i|A_i|)^{-1}$ and let $r$ be a positive integer with $hr \geq 1$ such that $\p_{p}( \op{Piv}[1,hr] ) < h$ for every $p \in [p_1,p_2]$. If $p_1\geq 1/d$, $\delta=\delta(p_1,p_2)\leq D$, and $h \leq h_0$ then the implication
\begin{multline}
\label{eq:main_snowballing_implication}
      \Bigl(h^{c_1 \delta^3} \leq c_3n^{-1} \text{ and }\tau_{p_1}^{\Lambda}(A_i \cup A_{i+1}) \geq 4 h^{c_1\delta^4} \text{ for every $i=1,\ldots,n-1$}\Bigr)  \\\Rightarrow \Bigl(\tau_{p_2}^{B_{2r}(\Lambda)}(A_1,A_n) \geq c_2 \tau_{p_1}^\Lambda(A_1) \tau_{p_1}^\Lambda(A_n)\Bigr)
\end{multline}
holds for every non-empty set of vertices $\Lambda$ in $G$. In particular, taking $\Lambda=V$ yields the implication
\begin{multline}
\label{eq:main_snowballing_implication_full_space}
      \Bigl(h^{c_1 \delta^3} \leq c_3n^{-1} \text{ and }\tau_{p_1}(A_i \cup A_{i+1}) \geq 4 h^{c_1\delta^4} \text{ for every $i=1,\ldots,n-1$}\Bigr)  \\\Rightarrow \Bigl(\tau_{p_2}(A_1,A_n) \geq c_2 \tau_{p_1}(A_1) \tau_{p_1}(A_n)\Bigr)
\end{multline}
whenever $p_1 \geq 1/d$, $\delta=\delta(p_1,p_2)\leq D$, and $h \leq h_0$.
\end{prop}

\begin{rk}
The fact that we can always find an integer $r$ such that $\p_{p}( \op{Piv}[1,hr] ) < h$ for every $p \in [p_1,p_2]$ can be deduced from the fact that there is at most one infinite cluster $\p_{p}$-almost surely for every $p\in [p_1,p_2]$ by an easy compactness argument. 
\end{rk}

\begin{rk}
The $\Lambda=V$ case of this lemma stated in \eqref{eq:main_snowballing_implication_full_space} already allows us to easily deduce that $p_c(G_n)\to p_c(G)$ when the transitive graphs in the sequence $(G_n)_{n\geq 1}$ all satisfy a uniform superpolynomial growth lower bound of the form $\op{Gr}(r) \geq r^{c(\log\log r)^{10}}$. This is explained in detail in \cref{subsec:high_growth_proof}. Working within finite domains as in \eqref{eq:main_snowballing_implication} will be useful when we apply \cref{prop:snowballing} at low-growth scales in \cref{sec:low_growth_multiscale}.
\end{rk}

The basic idea underlying \cref{prop:snowballing}, which is inspired by earlier works including \cite{contreras2022supercritical,duminil2021upper}, is that one can use the universal two-arm estimates derived from the work of Aizenman, Kesten, and Newman \cite{aizenman1987uniqueness} as reviewed in \cref{sec:AKN} to bound the maximum influence of an edge on certain connection events, which can then be used as an input in  Talagrand's sharp threshold theorem \cite{MR1303654}. Compared to those works, our primary additional insight is that these methods can be made vastly more efficient (especially in the high-growth case) by working with \emph{ghost field connection events} instead of more obvious connection events. Intuitively, these ghost field connection events are ``smoother'' than ordinary connection events, making it easier to bound the maximum influence of an edge. Moreover, the influence bound we get by using the two-ghost inequality of \cite{hutchcroft2019locality} gets better as the size of the relevant sets increases, so that we get extremely strong sharp threshold estimates when the graph has high growth.

\medskip

 Given a set of vertices $A$ in a graph $G$ and a parameter $h \in [0,1]$, the \emph{ghost field} of intensity $h$ on $A$ is the random subset $\mathscr{G}_A$ of $A$ in which each vertex is included independently at random with probability\footnote{In the literature one often takes this probability to be $1-e^{-h}$, which makes certain calculations more convenient. The distinction makes little difference since $1-e^{-h}=h\pm O(h^2)$ as $h\to 0$.} $h$. 
  We denote the law of $\mathscr{G}_A$ by $\mathbb G_h^A$. We record the following reformulation of the two-ghost inequality of \cite{hutchcroft2019locality}.


  \begin{lem}
  \label{lem:two_ghost_two_ghost}
  Let $G$ be a unimodular transitive graph of vertex degree $d$. There exists a constant $C=C(d)$ such that
\[\mathbb{G}^A_h \otimes \mathbb{G}^B_h \otimes \mathbb{P}_p( \{x \leftrightarrow \mathscr{G}_A\}\cap \{y \leftrightarrow \mathscr{G}_B\} \cap  \{x \nleftrightarrow y\} \cap \{ x \nleftrightarrow \infty \text{ or } y \nleftrightarrow \infty \}) \leq C \sqrt{\frac{1-p}{p} h} \]
for every $h \in [0,1]$, every pair of neighbouring vertices $x,y\in V(G)$, every two sets of vertices $A,B \subseteq V(G)$, and every $p \in (0,1)$.
  \end{lem}

  \begin{proof}
It suffices without loss of generality to consider the case $A=B=V$ since the relevant probability is increasing in $A$ and $B$. In this case, the probability is the same as if we had a single ghost field instead of two independent ghost fields, since the restrictions of a ghost field to the clusters of $x$ and $y$ are independent when these clusters are disjoint. This version of the lemma then follows easily from \cref{thm:two_ghost}. Alternatively, one can deduce the desired estimate from \cite[Theorem 1.6]{hutchcroft2019locality}. (The only difference is that in that paper the ghost fields are parameterised by $1-e^{-h}$ and are random sets of \emph{edges} rather than vertices. As with \cref{thm:two_ghost}, this is not a problem since the edge version of the statement is stronger than the vertex version.)
  \end{proof}

The following lemma states roughly that the probability that two low-intensity ghost fields are connected in a region $\Lambda$ undergoes a sharp threshold with respect to the percolation parameter $p$. This rough statement has two caveats: as we increase the percolation parameter, we also have to increase the ghost field intensity and thicken the region $\Lambda$. Although the argument using ghost field connections is new, the way we adapt it to run inside a given domain $\Lambda$ (when $\Lambda$ is not the whole vertex set) is inspired by the analysis of \cite[Section 5]{contreras2022supercritical}.

\begin{lem}[Sharp threshold for ghost connections] \label{lem:ghost_in_a_box_main}
For each $d\geq 1$ and $D<\infty$ there exists a positive constant $c=c(d,D)$ such that the following holds.
 Let $G=(V,E)$ be a unimodular transitive graph with vertex degree $d$, let $A,B$ be non-empty sets of vertices in $G$, and suppose that $0<p_1<p_2<1$ are such that either
 \begin{enumerate}
\item[(i)] there is at most one infinite cluster $\p_p$-almost surely for every $p\in [p_1,p_2]$, or
\item[(ii)] $B$ has finite complement.
\end{enumerate}
 If $p_1 \geq 1/d$  and $\delta(p_1,p_2) \leq D$ then the implication
      \begin{equation*}
            \Bigl(\mathbb G_{ h }^A \otimes \mathbb G_{ h }^B \otimes \p_{p_1} ( \mathscr{G}_A \xleftrightarrow{\Lambda\,} \mathscr{G}_B ) \geq h^{c \delta(p_1,p_2)}\Bigr)
      \Rightarrow
         \Bigl( \mathbb G_{h^c}^A \otimes \mathbb G_{h^c}^B \otimes \p_{p_2} ( \mathscr{G}_A \xleftrightarrow{B_{r}(\Lambda)} \mathscr{G}_B ) \geq 1 - h^{c\delta(p_1,p_2)}\Bigr)
      \end{equation*}
      holds for every $h\leq 1/d$ and every set $\Lambda \subseteq V$, where $r$ is the minimum positive integer such that $\p_p(\op{Piv}[1,hr])<h$ for every $p\in [p_1,p_2]$.
\end{lem}

\begin{rk}
 To prove \cref{prop:snowballing} we will apply this lemma only under the hypothesis (i). The version with hypothesis (ii) can be used as part of an alternative of the joint continuity of $\theta(p,G)$ in the supercritical region, as discussed in \cref{sec:continuity_of_density}.
\end{rk}

Before proving \cref{lem:ghost_in_a_box_main} we first recall \emph{Talagrand's inequality} \cite{MR1303654}, which (in combination with Russo's formula \cite{MR0488383}) states that there exists a universal positive constant $c$ such that if $A \subseteq \{0,1\}^E$ is an increasing event in a finite product space then
\begin{equation}
            \frac{d}{dp} \p_p(\mathcal A) \geq c \p_p(\mathcal A) (1-\p_p(\mathcal A)) \cdot \left[ p(1-p) \log \frac{2}{p(1-p)} \right]^{-1} \log \frac{1}{p(1-p) \max_{e\in E}  \p_p( \op{Piv}_e[\mathcal A] )}
\end{equation}
for every $p\in (0,1)$.





\begin{proof}[Proof of \cref{lem:ghost_in_a_box_main}]
      We may assume without loss of generality that $A$, $B$, and $\Lambda$ are finite, exhausting by finite sets and taking a limit otherwise. Since we want to apply Talagrand's inequality to the \emph{inhomogeneous}\footnote{The fact that Talagrand's inequality for homogeneneous product measures implies a version for inhomogeneous measures with parameters close to zero was already observed in \cite[Appendix A]{dembin2022sharp}. In our setting we have some parameters that may be close to zero and others that may be close to $1$.} product measure $\mathbb{G}^A_h \otimes \mathbb{G}^B_h \otimes \mathbb{P}_p$, we will first encode a random variable with this law as a function of i.i.d.\ random bits.
      Define 
      \[ m_E:=\left\lfloor \frac{\log(1-p_1)}{\log((d-1)/d)} \right\rfloor, \quad q_1:=1-(1-p_1)^{1/m_E}, \quad \text{ and } \quad q_2:=1-(1-p_2)^{1/m_E}.\]
      The assumption that $p_1 \geq 1/d$ ensures that $m_E \geq 1$, while it follows from the definitions that $(1-q_1)^{m_E}=(1-p_1)$, that $(1-q_2)^{m_E} = (1-p_2)$, and that
      \begin{equation*}
     \frac{1}{d} \leq  q_1 = 1-\exp\left( \left(\left\lfloor \frac{\log(1-p_1)}{\log((d-1)/d)} \right\rfloor\right)^{-1} \log(1-p_1)\right) \leq \frac{2d-1}{d^2} \leq \frac{2}{d}.
      \end{equation*}
     (Here we used only that $x/2\leq \lfloor x \rfloor \leq x$ for $x \geq 1$.) We also define 
     \[m_G= \left \lfloor \frac{\log h}{\log q_1} \right\rfloor,\]
     which satisfies $m_G \geq 1$ since $h\leq 1/d\leq q_1$.
     For each $q \in (0,1)$, let $\bar \p_q$ be the law of a random variable
      \[
            \mathrm{\textsc{bits}} = (\mathrm{\textsc{bits}}_A,\mathrm{\textsc{bits}}_B,\mathrm{\textsc{bits}}_\omega) \in  \{0,1\}^{A \times \{1,\ldots, m_G\} } \times \{0,1\}^{B \times \{1,\ldots,m_G\}}\times \{0,1\}^{E\times \{1,\ldots,m_E\}}=:\{0,1\}^\Omega,
      \]
      whose constituent random bits are independent Bernoulli random variables of parameter $q$. Given $\mathrm{\textsc{bits}} = (\mathrm{\textsc{bits}}_A,\mathrm{\textsc{bits}}_B,\mathrm{\textsc{bits}}_\omega) \in \{0,1\}^\Omega$, we define $(\mathscr{G}_A,\mathscr{G}_B,\omega)$ as a function of $\mathrm{\textsc{bits}}$ by
      \[
       \mathscr{G}_A(a) = \prod_{i=1}^{m_G} \mathrm{\textsc{bits}}_A(a,i), \quad  \mathscr{G}_B(b) = \prod_{i=1}^{m_G} \mathrm{\textsc{bits}}_B(b,i),  \quad \text{ and } \quad \omega(e) = 1- \prod_{i=1}^{m_E} (1-\mathrm{\textsc{bits}}_\omega(e,i)),
      \]
       so that the triple $(\mathscr{G}_A, \mathscr G_B,\omega)$ has law $\mathbb G_{q^{m_G}}^A \otimes \mathbb G_{q^{m_G}}^B \otimes \p_{1-(1-q)^{m_E}}$. The choice of parameter $m_G$ ensures that $q_1^{m_G} \geq h$ and hence that
      \[
      \bar\p_{q_1} ( \mathscr{G}_A \xleftrightarrow{\Lambda} \mathscr{G}_B ) \geq \mathbb G_{ h }^A \otimes \mathbb G_{ h }^B \otimes \p_{p_1} ( \mathscr{G}_A \xleftrightarrow{\Lambda} \mathscr{G}_B ).
      \]
      Note moreover that if $\mathcal{A} \subseteq \{0,1\}^A \times \{0,1\}^B \times \{0,1\}^E$ is an increasing event, $e\in E$ is an edge, and $1\leq k \leq m_E$ then $(e,k)$ is a closed pivotal for the event $\{\mathrm{\textsc{bits}} : (\mathscr{G}_A,\mathscr{G}_B,\omega) \in \mathcal{A}\}$ if and only if $e$ is a closed pivotal for the event $\mathcal{A}$ in the configuration $\omega$, so that
      \begin{multline}
      \label{eq:edge_bits1}
      (1-q)\bar \p_q \bigl((e,k) \text{ pivotal for } \{\mathrm{\textsc{bits}} : (\mathscr{G}_A,\mathscr{G}_B,\omega) \in \mathcal{A}\}\bigr) 
    \\= (1-q)^{m_E} \mathbb{G}^A_{q^{m_G}}\otimes \mathbb{G}^B_{q^{m_G}} \otimes \p_{1-(1-q)^{m_E}} (\op{Piv}_e[\mathcal{A}]).
      \end{multline}
      On the other hand, an element  $(x,k)$ of $A \times \{1,\ldots,m_G\}$ can only possibly be an open pivotal if $\mathrm{\textsc{bits}}_A(x,j)=1$ for every $j \in \{1,\ldots,m_G\}$. Similar considerations also apply with $A$ replaced by $B$, so that we obtain the coarse bound
      \begin{equation}\label{eq:ghost_bits1}
      q \hspace{0.05cm}\bar \p_q \bigl((x,k) \text{ pivotal for } \{\mathrm{\textsc{bits}} : (\mathscr{G}_A,\mathscr{G}_B,\omega) \in \mathcal{A}\}\bigr) \leq  q^{m_G}
      \end{equation}
      for every $x\in A\sqcup B$ and $1\leq k \leq m_G$.

\medskip

       Let $\ell := \lfloor h^{-1} \rfloor$ and for each $i =1,\ldots,\ell$ define the event 
      \[\mathcal E_i := \{ \mathscr{G}_A \xleftrightarrow{B_{ir h}(\Lambda)} \mathscr{G}_B \}.\] 
      We want to bound the maximum influence of an element of $(E\times \{1,\ldots,m_E\}) \sqcup (A\times \{1,\ldots,m_G\}) \sqcup (B \times \{1,\ldots,m_G\})$ on the event $\mathcal{E}_i$. For elements of $(A\times \{1,\ldots,m_G\}) \sqcup (B \times \{1,\ldots,m_G\})$ it will suffice to use the trivial bound of \eqref{eq:ghost_bits1}, so that it remains only to bound the pivotality probability $\max_e \mathbb{G}^A_{h'}\otimes \mathbb{G}^B_{h'} \otimes \p_{p} (\op{Piv}_e[\mathcal{E}_i])$ where $(1-p)=(1-q)^{m_E}$ and $h'=q^{m_G}$.  Following \cite{contreras2022supercritical}, we will do this not for every $i$ but instead show that an influence bound of the desired form must hold for an \emph{average} choice of $i$. (Note that if we are working directly in the case $\Lambda=V$ then this issue does not arise.)
    More precisely, we claim that for each $q\in [q_1,q_2]$ there exists a set $I(q)\subseteq \{1,\ldots,\ell\}$ with $|I|\geq 1/(3h)$ such that
\begin{equation}
\label{eq:Ei_max_influence}
\max_{i\in I}\max_{e \in E} q(1-q) \mathbb{G}^A_{h'} \otimes \mathbb{G}^B_{h'} \otimes \p_p(\op{Piv}_{e}[\mathcal{E}_i]) \leq C_d (h')^{1/2}
\end{equation}
where $C_d$ is a constant depending only on the degree $d$.
    There are two separate cases to consider: Edges both of whose endpoints belong to $B_{(i-1)r h}(\Lambda)$ (\emph{bulk edges}) and edges with at least one endpoint not in $B_{(i-1)r h}(\Lambda)$ (\emph{boundary edges}).

\medskip


    \textbf{Bulk edges.}
      First consider an edge $e$ both of whose endpoints $x$ and $y$ belong to $B_{(i-1)r h}(\Lambda)$.
      If $\omega$ is such that $e$ is pivotal for the event $\mathcal{E}_i$ then at least one of the following two events must occur:
      \begin{enumerate}
      \item[(i)] The endpoints $x$ and $y$ of $e$ belong to distinct $\omega$-clusters, one of which intersects $\mathscr{G}_A$ but not $\mathscr{G}_B$ and the other of which intersects $\mathscr{G}_B$ but not $\mathscr{G}_A$; at least one of these clusters must be finite almost surely by the hypotheses of the lemma, which allows us to bound the relevant probability using the two-ghost inequality.
      \item[(ii)] The endpoints $x$ and $y$ of $e$ are both $\omega$-connected to the boundary of $B_{irh}(\Lambda)$ but are not $\omega$-connected to each other within $B_{irh}(\Lambda)$, so that $\op{Piv}[1,rh](x)$ and $\op{Piv}[1,rh](y)$ both hold.
  \end{enumerate}
  Thus, it follows by the two-ghost inequality as stated in \cref{lem:two_ghost_two_ghost} and the definition of $r$ that
      \begin{equation} \label{eq:pivotality_prob_for_good_edge}
            q(1-q)\bar \p_{q}(\op{Piv}_e[ \mathcal E_i ]) \leq q(1-q)\left[ C_1 \frac{(1-p)^{1/2}}{p^{1/2}} (h')^{1/2} + h \right] \leq C_2 (h')^{\frac{1}{2}},
      \end{equation}
      for every $q\in [q_1,q_2]$ and $1\leq i \leq \ell$, where $C_1$ and $C_2$ are constants depending only on $d$ and we used that $p \geq 1/d$.

      \medskip

      \textbf{Boundary edges.}
      An edge $e$ not having both its endpoints in $B_{irh}(\Lambda)$ cannot possibly be pivotal for $\mathcal E_i$; we need only bound $\bar \p_{q}( \op{Piv}_e[\mathcal E_i] )$ for edges $e$ with both endpoints in $B_{irh}(\Lambda)$ and with at least one endpoint not in $B_{(i-1)rh}(\Lambda)$. Rather than bound this maximum influence uniformly in $i$, we will bound it \emph{on average}. 
      Fix $q \in [q_1,q_2]$ and, for each $i \in \{1,\ldots,\ell\}$, pick an edge $e_i=e_i(q) \in B_{irh}(\Lambda)$ with at least one endpoint not in $B_{(i-1)rh}(\Lambda)$ that maximises $\bar \p_{q}( \op{Piv}_{e_i}[\mathcal E_i] )$ over all such edges. Notice that the events $\op{Piv}_{e_i}[\mathcal E_i] \cap \{ \omega(e_i) = 1\}$ for $i \in \{1,\ldots,\ell\}$ are pairwise disjoint. Indeed, if $\op{Piv}_{e_i}[\mathcal E_i] \cap \{ \omega(e_i)=1\}$ occurs then $e_i$ must belong to \emph{every} path connecting $\mathscr{G}_A$ to $\mathscr{G}_B$ in $B_{irh}(\Lambda)$, and such a path cannot possibly include $e_j$ if $j>i$. It follows in particular that
      \[
           q \sum_{i=1}^{\ell} \bar \p_{q}( \op{Piv}_{e_i}[\mathcal E_i] ) = \sum_{i=1}^{\ell} \bar \p_{q}( \op{Piv}_{e_i}[\mathcal E_i] \cap \{ \omega(e_i) = 1\} ) \leq 1.
      \]
      By Markov's inequality, it follows that we can find a subset $I(q) \subseteq \{1,\ldots,\ell\}$ with $\abs{I} \geq \frac{1}{3h}$ such that
      \[
            \max_{i \in I} q \bar \p_{q}( \op{Piv}_{e_i}[\mathcal E_i] ) \leq 6h \leq 6 (h')^{1/2}
      \]
      as required. This concludes the proof of \eqref{eq:Ei_max_influence}.


\medskip

     Now, the assumption that $\delta(p_1,p_2) \leq D$ implies that $(1-p_2) \geq (1-p_1)^{e^D}$ and hence that $1-q_2 \geq (1-q_1)^{e^D} \geq ((2d-1)/(2d))^{e^D}>0$, so that $q_1$ and $q_2$ are bounded away from $0$ and $1$ by constants depending only on $d$ and $D$. As such, Talagrand's inequality (which is valid to use in our setting since all our events only depend on the status of finitely many bits) together with \eqref{eq:Ei_max_influence} yields that
\[\frac{d}{dq} \log \left[\frac{\bar\p_{q}(\mathcal E_i)}{1-\bar\p_{q}(\mathcal E_i)}\right] \geq c_1 \log \frac{1}{C_d q^{m_G/2}} \geq c_2 m_G \log \frac{1}{q_1} \geq c_3 \log \frac{1}{h} \]
    for every $q_1 \leq q \leq q_2$ and every $i\in I(q)$, where $c_1$, $c_2$, and $c_3$ are positive constants depending only on $d$ and $D$. Since $\ell \leq h^{-1}$ and the derivative is non-negative for every $i$, we can sum over $i$ to obtain that
    \[
    \frac{1}{\ell} \sum_{i=1}^\ell \frac{d}{dq} \log \left[\frac{\bar\p_{q}(\mathcal E_i)}{1-\bar\p_{q}(\mathcal E_i)}\right] \geq \frac{c_3}{3} \log \frac{1}{h} 
    \]
    for every $q_1\leq q \leq q_2$.    
    Integrating this differential inequality yields that
    \[
    \frac{1}{\ell} \sum_{i=1}^\ell \left(\log \left[\frac{\bar\p_{q_2}(\mathcal E_i)}{1-\bar\p_{q_2}(\mathcal E_i)}\right]-\log \left[\frac{\bar\p_{q_1}(\mathcal E_i)}{1-\bar\p_{q_1}(\mathcal E_i)}\right]\right) \geq \frac{c_3}{3} |q_2-q_1| \log \frac{1}{h} 
    \]
    and hence that there exists $i\in \{1,\ldots,\ell\}$ such that
    \begin{align*}
    \max\left\{ \log \left[\frac{1}{1-\bar\p_{q_2}(\mathcal E_i)}\right], \log \left[\frac{1}{\bar\p_{q_1}(\mathcal E_i)}\right] \right\} &\geq 
    \max\left\{ \log \left[\frac{\bar\p_{q_2}(\mathcal E_i)}{1-\bar\p_{q_2}(\mathcal E_i)}\right] , -\log \left[\frac{\bar\p_{q_1}(\mathcal E_i)}{1-\bar\p_{q_1}(\mathcal E_i)}\right] \right\} 
    \\&\geq \frac{c_2}{6} |q_2-q_1| \log \frac{1}{h}.
    \end{align*}
    The claim follows easily from this together with the inequality
      \begin{multline*}
      |q_2-q_1| = |(1-p_2)^{1/m_E}-(1-p_1)^{1/m_E}| \geq \Bigl|(1-p_2)^{\frac{\log((d-1)/d)}{\log(1-p_1)}}-(1-p_1)^{\frac{\log((d-1)/d)}{\log(1-p_1)}}\Bigr|
      \\= \frac{d-1}{d}\left[1-\exp\left(\left(\frac{\log (1-p_2)}{\log(1-p_1)}-1\right)\log \frac{d-1}{d}\right)\right] \geq c_4 \delta(p_1,p_2),
      \end{multline*}
      which holds by calculus with the constant $c_4$ depending only on $d$ and $D$.
    \qedhere
\end{proof}

Our next goal is to run a chaining-like argument but with ghost fields. In this analogy, the previous lemma can be thought of as an \emph{existence} statement whereas the next lemma can be thought of as a \emph{uniqueness} statement. (Note that our proof of the next lemma relies crucially on the previous lemma.) 
We work with the standard monotone coupling $(\omega_p)_{p\in [0,1]}$ of Bernoulli bond percolation, write $\{A \xleftrightarrow[p]{\,\Lambda\,} B\}$  for the event that $A$ is connected to $B$ by a path that is contained in $\Lambda$ and open in $\omega_p$, and write $\{A \nxleftrightarrow[p]{\Lambda} B\}$ for the event that $A$ is \emph{not} connected to $B$ by any such path.

\begin{figure}
\centering
\includegraphics[width=0.75\textwidth]{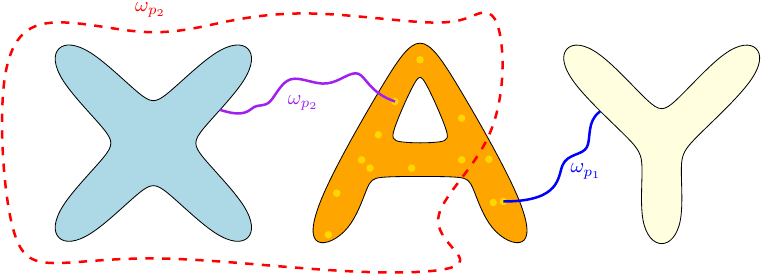}
\caption{Schematic illustration of the event whose probability is estimated in \cref{lem:ghost_in_box_transitivity} when $\Lambda=V$: If any two points in $A$ have a reasonable probability to be connected in $\omega_{p_1}$, then it is unlikely that $X$ is connected to a weak ghost field on $A$ in $\omega_{p_2}$ and $Y$ is connected to a weak ghost field on $A$ in $\omega_{p_1}$ but that $X$ and $Y$ are not connected in $\omega_{p_2}$.}
\label{fig:XAY}
\end{figure}

\begin{lem}[Gluing ghost connections] \label{lem:ghost_in_box_transitivity}
For each $d\geq 1$ and $D<\infty$ there exist positive constants $c_1=c_1(d,D)$ and $c_2=c_2(d,D)$ such that the following holds.
 Let $G=(V,E)$ be a unimodular transitive graph with vertex degree $d$, let $A$, $X$, and $Y$ be non-empty sets of vertices in $G$, and suppose that $0<p_1<p_2<1$ are such that there is at most one infinite cluster $\p_p$-almost surely for every $p\in [p_1,p_2]$. If $p_1 \geq 1/d$ and $\delta=\delta(p_1,p_2)\leq D$ then the implication
%
%
\[
\Biggl[\tau_{p_1}^\Lambda(A) \geq h^{c_1\delta}\Biggr] \Rightarrow  \Biggl[    \mathbb G_{h}^{A} \otimes \p \left( X \xleftrightarrow[p_2]{ B_r(\Lambda) } \mathscr{G}_A \text{ and } Y \xleftrightarrow[p_1\,]{\;\;\Lambda\;\;\,} \mathscr{G}_A \text{ but } X \nxleftrightarrow[p_2]{ B_r(\Lambda)}Y  \right) \leq 3 h^{c_2\delta^3}\Biggr]
\]
holds for every $h \leq 1/d$ and every set $\Lambda \subseteq V$, where $r$ is the minimum positive integer such that $\p_p(\op{Piv}[1,hr])<h$ for every $p\in [p_1,p_2]$.
\end{lem}

An illustration of the event whose probability this lemma estimates is given in \Cref{fig:XAY}.

\begin{rk}
Several of the calculations in the following proof can be simplified significantly if one allows the constants to depend on how small $1-p_1$ is. Getting the constants to be independent of the choice of $p_1$ is important when using our methods to deduce that $p_c<1$ for transitive graphs of superlinear volume growth.
\end{rk}

\begin{proof}[Proof of \cref{lem:ghost_in_box_transitivity}]
It suffices to prove the claim with $3h^{c_2 \delta^4}$ replaced by $\max\{h^{c_1\delta},3h^{c_2 \delta^4}\}$, as the latter  can be bounded by the former  after an appropriate decrease of the relevant constant.
%
We write $p_{4/3}$ and $p_{5/3}$ for the parameters defined by $p_{4/3}=\sprinkle(p_1;\delta/3)$
 and $p_{5/3}=\sprinkle(p_1;2\delta/3)$, so that $p_1\leq p_{4/3} \leq p_{5/3} \leq p_2$. To lighten notation we write $\omega_1=\omega_{p_1}$, $\omega_{4/3}=\omega_{p_{4/3}}$, $\omega_{5/3}=\omega_{p_{5/3}}$, and $\omega_{2}=\omega_{p_{2}}$. We will write $\asymp$, $\preceq$, and $\succeq$ for equalities and inequalities holding to within positive multiplicative constants depending only on $d$ and $D$. We also use the asymptotic big-$O$ and big-$\Omega$ notation with all implicit constants depending only on $d$ and $D$.

Let $c_0=c_0(d,D)$ be the constant from \cref{lem:ghost_in_a_box_main}. We will prove the claim with $c_1=(c_0 \wedge 1)/9$. 
      Assume to this end that $\tau_{p_1}^{\Lambda}(A) \geq h^{c_1\delta}$. 
Let $K_X(2,r)$ be the set of vertices that are connected to $X$ by an $\omega_2$-open path in $B_r(\Lambda)$ and let $K_Y(1,0)$ be the set of vertices that are connected to $Y$ by an $\omega_1$-open path in $\Lambda$.
Suppose that there exist two disjoint, connected sets of vertices $C_X \supseteq X$ and $C_Y \supseteq Y$  such that $\p(K_X(2,r) = C_X, \; K_Y(1,0) = C_Y)>0$ and
      \begin{multline} \label{eq:definition_of_xy}
            \min\{\mathbb G_h^A ( \mathscr{G}_A \cap C_X \not= \emptyset),\mathbb G_h^A ( \mathscr{G}_A \cap C_Y \not= \emptyset)\} \\ \geq \mathbb G_{h}^{A} \otimes \p \left( X \xleftrightarrow[p_2]{ B_r(\Lambda)\, } \mathscr{G}_A \xleftrightarrow[p_1]{ \Lambda\, } Y  \mid K_X(2,r) = C_X, \; K_Y(1,0) = C_Y \right) \geq h^{ c_1\delta};
      \end{multline}
      If no such pair of sets exists then the conclusion holds trivially (since we are proving a modified version of the claim with $\max\{h^{c_1\delta},3h^{c_2 \delta^4}\}$ on the right hand side of the implication). 
      For such $C_X$ and $C_Y$ we have that
      \[\begin{split}
            \mathbb G_h^{C_X} \otimes \mathbb G_h^{C_Y} \otimes \mathbb P \left( \mathscr{G}_{C_X} \xleftrightarrow[p_1]{\,\Lambda\;} \mathscr{G}_{C_Y}  \right) &\geq \mathbb G_h^{C_X} \otimes \mathbb G_h^{C_Y} (\mathscr{G}_{C_X} \cap A \not= \emptyset \text{ and } \mathscr{G}_{C_Y} \cap A \not= \emptyset ) \min_{u,v \in A} \p(u \xleftrightarrow[p_1]{\,\Lambda\;} v)\\
            &= \mathbb G_h^A ( \mathscr{G}_A \cap C_X \not= \emptyset)\mathbb G_h^A ( \mathscr{G}_A \cap C_Y \not= \emptyset) \min_{u,v \in A} \p(u \xleftrightarrow[p_1]{\,\Lambda\;} v) \\
            &\geq h^{c_1\delta } \cdot h^{c_1 \delta} \cdot h^{c_1\delta} = h^{3c_1\delta} \geq h^{\frac{c_0}{3} \delta}.
      \end{split}\]
      Thus, since $\delta(p_1,p_{4/3}) = \frac{1}{3}\delta(p_1,p_2)$, it follows from \cref{lem:ghost_in_a_box_main} (applied with $p_1$ and $p_{4/3}$ rather than $p_1$ and $p_2$) that
      \begin{equation} \label{eq:applied_ghost_field_trick_to_ghost_clusters}
           \p\left( C_X \nxleftrightarrow[p_{4/3}]{B_r(\Lambda)} C_Y \right) \leq \mathbb G_{h^{c_0}}^{C_X} \otimes \mathbb G_{h^{c_0}}^{C_Y} \otimes \p \Bigl( \mathscr{G}_{C_X} \nxleftrightarrow[p_{4/3}]{B_r(\Lambda)\,} \mathscr{G}_{C_Y} \Bigr) \leq h^{\frac{c_0}{3} \delta}.
      \end{equation}

      Let $M$ denote the maximum cardinality of a set of edge-disjoint paths from $C_X$ to $C_Y$ that are contained in $B_r(\Lambda)$ and are open in $\omega_{5/3}$ with the possible exception of their first and last edges. (We stress that $C_X$ and $C_Y$ are not random, but are fixed sets satisfying \eqref{eq:definition_of_xy}.) By Menger's theorem, $M$ is equal to the minimum size of a set of edges separating $C_X$ from $C_Y$ in the subgraph of $B_r(\Lambda)$ spanned by those edges that are either $\omega_{5/3}$-open or have at least one endpoint in $C_X \cup C_Y$. 
      Observe that $M$ is independent of the event $\{K_X(2,r)=C_X, K_Y(1,0)=C_Y\}$ since $M$ depends only on edges with neither endpoint in $C_X \cup C_Y$ while the latter event depends only on edges with at least one endpoint in $C_X \cup C_Y$.
      Conditional on $\omega_{5/3}$, each edge that is open in $\omega_{5/3}$ is closed in $\omega_{4/3}$ with probability 
      \[
      \beta:=\frac{p_{5/3}-p_{4/3}}{p_{5/3}}.
      \]
      It follows that
      \[
            \p\left( C_X \nxleftrightarrow[p_{4/3}]{B_r(\Lambda)} C_Y \mid M \leq N \right) \geq \beta^N
      \]
      for every $N\geq 0$ 
      and hence by \eqref{eq:applied_ghost_field_trick_to_ghost_clusters} and the aforementioned independence that
      \begin{equation} \label{eq:few_disjoint_paths_is_unlikely_between_ghost_clusters}
          \p( M \leq N \mid K_X(2,r) = C_X, \; K_Y(1,0) = C_Y ) =   \p\left( M \leq N \right) \leq \beta^{-N} h^{\frac{c_0}{3}\delta}
      \end{equation}
      for every $N\geq 0$. 
      Let $K_Y(5/3,r)$ be the set of vertices that are connected to $Y$ by an $\omega_{5/3}$-open path in $B_r(\Lambda)$.
      Conditioned on the event $\{K_X(2,r)=C_X$ and $K_Y(1,0)=C_Y\}$ and the value of $M$, the size of the boundary $\abs{ \partial K_X(2,r) \cap \partial K_Y(5/3,r)}$  stochastically dominates a sum of $M$ Bernoulli random variables of parameter 
\[
\alpha_1 := \frac{p_{5/3}-p_1}{1-p_1}.
\]
Indeed, if we take some maximal-cardinality set of paths as in the definition of $M$, then the edge adjacent to $C_Y$ of each path in the set is open in $\omega_{5/3}$ with this probability, and the size of the relevant boundary is at least the number of these edges that are open in $\omega_{5/3}$.
      Letting $Z$ be a sum of $N$ Bernoulli random variables each of parameter $1-\alpha_1$, we have that there exist constants $c_3$ and $c_4$ depending only on $d$ and $D$ such that
      \begin{multline} \label{eq:many_disjoint_paths_between_ghost_clusters_yields_many_touches}
            \p \Bigl( \abs{ \partial K_X(2,r) \cap \partial K_Y(5/3,r)} \leq \frac{\alpha_1}{2} N \;\Big|\; K_X(2,r) = C_X, \; K_Y(1,0) = C_Y, \; M \geq N \Bigr) \\
            \leq \p\left(Z \geq \frac{2-\alpha_1}{2}N\right) 
            =\p\left(Z \geq \left(1+\frac{\alpha_1}{2(1-\alpha_1)}\right)\mathbb{E} Z\right)
            \\
            \leq \exp\left[-\left(\frac{2-\alpha_1}{2}\log \frac{2-\alpha_1}{2-2\alpha_1}-\frac{\alpha_1}{2}\right) N\right]
      \end{multline}
      for every $N\geq 1$, where the final inequality follows from the Chernoff bound $\p(Z\geq (1+x) \mathbb{E}Z) \leq (e^{-x}(1+x)^{1+x})^{-\mathbb{E} Z}$, which holds for all sums of independent Bernoulli random variables.

%
%
%

\medskip

      We will now apply \eqref{eq:few_disjoint_paths_is_unlikely_between_ghost_clusters} and \eqref{eq:many_disjoint_paths_between_ghost_clusters_yields_many_touches} with
      \[N := \left\lfloor \frac{c_0 \delta }{12 \log (1/\beta)} \log \frac{1}{h} \right\rfloor, \qquad \text{so that} \qquad \beta^{-N}h^{\frac{c_0}{3} \delta} \leq h^{\frac{c_0}{6} \delta}\]
      to prove the inequality 
       \begin{multline} \label{eq:conditioned_on_good_ghost_clusters_there_are_many_touches2}
            \p \left( \abs{ \partial K_X\bigl(2,r\bigr) \cap \partial K_Y\bigl(5/3,r\bigr)} \leq \frac{c_0\alpha_1 \delta}{48\log(1/\beta)} \log \frac{1}{h} \,\Big|\, K_X(2,r) = C_X, \; K_Y(1,0) = C_Y \right) 
        \\\leq 
        2 h^{\Omega(\delta^4)}.
      \end{multline}
      To do this, we will make repeated use of the elementary estimates
       \begin{align}
      p_{5/3}\beta=(1-p_{5/3})^{e^{-\delta/3}} - (1-p_{5/3}) &= (1-p_{5/3})((1-p_{5/3})^{e^{-\delta/3}-1}-1) 
      \nonumber\\&\geq
      (1-p_{5/3})((1-p_{5/3})^{-\Omega(\delta)}-1) 
      \succeq \delta(1-p_{5/3}) \log \frac{1}{1-p_{5/3}}
       \label{eq:beta_estimate_elementary}
      \end{align}
      and
      \begin{align}
      \alpha_1 = 1-(1-p_{5/3})^{1-e^{-2\delta/3}} = 1-(1-p_{5/3})^{\Theta(\delta)}
      \label{eq:alpha_1_estimate_elementary}
      \end{align}
      where we recall that the implicit constant appearing here depends only on $d$ and $D$.
       First suppose that $N\geq 1$, so that the rounding in the definition of $N$ reduces its size by a factor of at most $1/2$. In this case, \eqref{eq:few_disjoint_paths_is_unlikely_between_ghost_clusters} and \eqref{eq:many_disjoint_paths_between_ghost_clusters_yields_many_touches} 
       yield that there exist positive constants $C_1$, $c_5$ and $c_6$ depending only on $d$ and $D$ such that
      \begin{multline} \label{eq:conditioned_on_good_ghost_clusters_there_are_many_touches}
            \p \left( \abs{ \partial K_X\bigl(2,r\bigr) \cap \partial K_Y\bigl(5/3,r\bigr)} \leq \frac{c_0\alpha_1 \delta}{48\log(1/\beta)} \log \frac{1}{h} \,\Big|\, K_X(2,r) = C_X, \; K_Y(1,0) = C_Y \right) 
        \\\leq h^{\frac{c_0}{6}\delta} +
        \exp\left[-\frac{c_0 \delta}{24 \log(1/\beta)}\left(\frac{2-\alpha_1}{2}\log \frac{2-\alpha_1}{2-2\alpha_1}-\frac{\alpha_1}{2}\right) \log \frac{1}{h} \right].
      \end{multline}
      We claim that 
      \begin{equation}
      \label{eq:calculus_claim}
      \frac{c_0 \delta}{24 \log(1/\beta)}\left(\frac{2-\alpha_1}{2}\log \frac{2-\alpha_1}{2-2\alpha_1}-\frac{\alpha_1}{2}\right) \succeq \delta^4.
      \end{equation}
        We prove this by a case analysis according to whether $1-p_{5/3} \leq e^{-1/\delta}$. 
%
      If $1-p_{5/3} \geq e^{-1/\delta}$, it follows from \eqref{eq:beta_estimate_elementary} that $\log 1/\beta = O(\delta)$ (where we stress again that the implicit constants depend only on $d$ and $D$). On the other hand, since $p_1 \geq 1/d$, we have that $\alpha_1 \succeq \delta$, and together with the elementary inequality
      \[
\frac{2-\alpha_1}{2}\log \frac{2-\alpha_1}{2-2\alpha_1}-\frac{\alpha_1}{2} \geq \frac{\alpha_1^2}{8}
      \] 
      this yields that if $1-p_{5/3} \geq e^{-1/\delta}$ then
      \[
\frac{c_0 \delta}{24 \log(1/\beta)}\left(\frac{2-\alpha_1}{2}\log \frac{2-\alpha_1}{2-2\alpha_1}-\frac{\alpha_1}{2}\right) \succeq \frac{\delta}{\delta}\cdot \delta^2 \succeq \delta^3
      \]
      as claimed. On the other hand, if $1-p_{5/3}\leq e^{-1/\delta}$ then $1-\alpha_1 \geq 1-e^{-\Omega(1)} \succeq 1$, so that
      \[
\frac{2-\alpha_1}{2}\log \frac{2-\alpha_1}{2-2\alpha_1}-\frac{\alpha_1}{2} \succeq\log \frac{1}{1-\alpha_1} \succeq \delta \log \frac{1}{1-p_{5/3}}. 
      \] 
We have under the same assumption that
\[
\log (1/\beta) \preceq \log \frac{1}{1-p_{5/3}},
\]
and it follows that if $1-p_{5/3} \leq e^{-1/\delta}$ then
\[
\frac{c_0 \delta}{24 \log(1/\beta)}\left(\frac{2-\alpha_1}{2}\log \frac{2-\alpha_1}{2-2\alpha_1}-\frac{\alpha_1}{2}\right) \succeq \delta^2 \succeq \delta^3.
\]
This completes the proof of \eqref{eq:calculus_claim}, which together with \eqref{eq:conditioned_on_good_ghost_clusters_there_are_many_touches} yields the claimed inequality \eqref{eq:conditioned_on_good_ghost_clusters_there_are_many_touches2} in the case that $N\geq 1$. Now suppose that $N=0$, so that
\begin{multline}
\p \left( \abs{ \partial K_X\bigl(2,r\bigr) \cap \partial K_Y\bigl(5/3,r\bigr)} \leq \frac{c_0\alpha_1 \delta}{48\log(1/\beta)} \log \frac{1}{h} \,\Big|\, K_X(2,r) = C_X, \; K_Y(1,0) = C_Y \right)
\\
 = \p \left( \abs{ \partial K_X\bigl(2,r\bigr) \cap \partial K_Y\bigl(5/3,r\bigr)} =0 \,\Big|\, K_X(2,r) = C_X, \; K_Y(1,0) = C_Y \right) 
 \leq h^{\frac{c_0}{3}\delta}+(1-\alpha_1),
 \label{eq:many_touches3}
\end{multline} 
where, as in \eqref{eq:many_disjoint_paths_between_ghost_clusters_yields_many_touches}, the first term on the right hand side of the second line bounds the probability that $M=0$ and the second bounds the probability of the appropriate event conditional on $M \geq 1$. 
We will once again prove \eqref{eq:conditioned_on_good_ghost_clusters_there_are_many_touches2} via case analysis according to whether or not $1-p_{5/3} \geq e^{-1/\delta}$. If $1-p_{5/3}\geq e^{-1/\delta}$ then $\log (1/\beta) = O(\delta)$, and since $N=0$ it follows that $h=e^{-O(\delta^{-2})}$. As such, in this case there exists a positive constant $c_5$ such that $h^{c_5 \delta^4} \geq 1/2$, and the desired inequality \eqref{eq:conditioned_on_good_ghost_clusters_there_are_many_touches2} follows trivially since probabilities are bounded by $1$. Otherwise, $1-p_{5/3} \leq e^{-1/\delta}$ and $1-\alpha_1 \leq (1-p_{5/3})^{\Omega(\delta)}$. Since $N=0$, we also have that $h \geq (1-p_{5/3})^{O(\delta^{-1})}$ and hence that $h^{\delta^2}=(1-\alpha_1)^{\Omega(1)}$. As such, \eqref{eq:many_touches3} is stronger than the claimed inequality \eqref{eq:conditioned_on_good_ghost_clusters_there_are_many_touches2} when $N=0$ regardless of the value of $1-p_{5/3}$. This completes the proof of the inequality \eqref{eq:conditioned_on_good_ghost_clusters_there_are_many_touches2}.
     

\medskip

      Since the sets $C_X$ and $C_Y$ were arbitrary sets such that $\p(K_X(2,r) = C_X, \; K_Y(1,0) = C_Y)>0$ and that satisfy \eqref{eq:definition_of_xy}, it follows from  \eqref{eq:conditioned_on_good_ghost_clusters_there_are_many_touches2} that
      \begin{multline} \label{eq:ghost_transitivity_reduced_to_touching_without_joining}
            \mathbb G_h^{A} \otimes \p \left( X \xleftrightarrow[p_2]{B_r(\Lambda)} \mathscr{G}_A \xleftrightarrow[p_1]{\,\Lambda\;} Y \text{ but } X \nxleftrightarrow[p_2]{B_r(\Lambda)} Y \right) 
            \leq h^{c_1\delta}+ 2h^{ \Omega(\delta^4)}
            \\ + \p\left( \abs{ \partial K_X(2,r) \cap \partial K_Y(5/3,r) } \geq \frac{c_0\alpha_1 \delta}{48\log(1/\beta)}\log \frac{1}{h} \text{ but } X \nxleftrightarrow[p_2]{B_r(\Lambda)} Y \right),
      \end{multline}
      where the first term $h^{c_1\delta}$ accounts for the possibility that $K_X(2,r)$ and $K_Y(1,0)$ are equal to some sets $C_X$ and $C_Y$ that do not satisfy \eqref{eq:definition_of_xy}. It remains to bound the final term appearing on the right hand side of \eqref{eq:ghost_transitivity_reduced_to_touching_without_joining}. 
      Notice that we can replace the set $K_X(2,r)$ appearing in \eqref{eq:ghost_transitivity_reduced_to_touching_without_joining} by the set $\tilde K_X$ of vertices that are connected to $X$ by an $\omega_2$-open path using only edges that have both endpoints in $B_r(\Lambda)$ and neither endpoint in $K_Y(5/3,r)$, since the two sets $K_X(2,r)$ and $\tilde K_X$ are equal when $X \nxleftrightarrow[p_2]{B_r(\Lambda)} Y$. Now let $C_X \supseteq X$ and $C_Y \supseteq Y$ be connected sets of vertices that are disjoint from each other such that $\p(\tilde K_X = C_X, K_Y(5/3,r)=C_Y)>0$ and
      \[
            \abs{\partial C_X \cap \partial C_Y} \geq \frac{c_0\alpha_1 \delta}{48\log(1/\beta)} \log \frac{1}{h};
      \]
      If no such sets exist then the last term on the right hand side of \eqref{eq:ghost_transitivity_reduced_to_touching_without_joining} is zero.
     Each edge that is closed in $\omega_{5/3}$ is open in $\omega_2$ with probability
      \[
      \alpha_2 = \frac{p_2-p_{5/3}}{1-p_{5/3}},
      \]
      and we have by independence that
      \begin{multline*}
            \p\left( X \nxleftrightarrow[p_2]{B_r(\Lambda)} Y \mid \tilde K_X=C_X,\, K_Y(5/3,r) = C_Y \right) \leq (1-\alpha_2)^{ \abs{\partial C_X \cap \partial C_Y}}\\
            \leq \exp\left[ - \left(\frac{c_0\alpha_1 \delta}{48\log(1/\beta)} \log\frac{1}{1-\alpha_2}\right) \log \frac{1}{h} \right] \leq h^{\Omega(\delta^3)},
      \end{multline*}
      where the final inequality follows by a similar calculation used to prove \eqref{eq:calculus_claim} above.
      Thus, summing over all choices of $C_X$ and $C_Y$, we obtain that
      \begin{equation} \label{eq:ghost_transitivity_unlikely_to_not_join_if_you_touched}
            \p \left( X \nxleftrightarrow[p_2]{B_r(\Lambda)} Y \;\Big|\; \bigl|\partial \tilde K_X \cap \partial K_Y(5/3,r) \bigr| \geq \frac{c_0\alpha_1 \delta}{48\log(1/\beta)} \log \frac{1}{h} \right) \leq h^{\Omega(\delta^3)},
      \end{equation}
      and the claim follows from \eqref{eq:ghost_transitivity_reduced_to_touching_without_joining} and \eqref{eq:ghost_transitivity_unlikely_to_not_join_if_you_touched}.
\end{proof}

We now apply \cref{lem:ghost_in_a_box_main,lem:ghost_in_box_transitivity} to prove \cref{prop:snowballing}.

\begin{proof}[Proof of \cref{prop:snowballing}]
      To lighten notation, for each $\rho \in [0,1]$ define $\mathbb Q_\rho := \bigotimes_{i=1}^n \mathbb G_\rho^{A_i}$, and for each $i$ write $\mathscr{G}_i := \mathscr{G}_{A_i}$. 
      Let $c_1=c_1(d,D)$  be the constant from \cref{lem:ghost_in_a_box_main} and let $c_2=c_2(d,D)$ and $c_3=c_3(d,D)$ be the constants from \cref{lem:ghost_in_box_transitivity} (that are called $c_1$ and $c_2$ in the statement of that lemma). Define $c=c(d,D) =  (2D)^{-5}(c_1 \wedge 1) (c_2 \wedge 1) (c_3 \wedge 1)$. Let $h\geq (\min_i \abs{A_i})^{-1}$, noting that $\mathbb G_{h}^{A_i} ( \mathscr{G}_i \not= \emptyset) \geq 1-(1-(\min_i \abs{A_i})^{-1})^{\abs{A_i}} \geq 1-1/e \geq 1/2$ for every $i$, and let $r$ be the minimum positive integer such that $\p_p(\op{Piv}[1,hr])<h$ for every $p\in [p_1,p_2]$.

\medskip

      Fix $u \in A_1$ and $v \in A_n$ and suppose that 
      \[\tau_{p_1}^{\Lambda}(A_i \cup A_{i+1}) \geq 4 h^{c \delta^4} \qquad \text{for all $1\leq i \leq n-1$.}\]
      We want to bound from below the probability under $\p_{p_2}$ that $u$ and $v$ are connected inside $B_{2r}(\Lambda)$.      
      Let $1\leq i \leq n-1$ be arbitrary. Note that
      \begin{equation} \label{eq:consecutive_ghosts_can_connect}
            \mathbb Q_h \otimes \p\left( \mathscr{G}_i \xleftrightarrow[p_1]{\,\Lambda\;} \mathscr{G}_{i+1} \right) \geq \mathbb Q_h(\mathscr{G}_i \not= \emptyset) \mathbb Q_h(\mathscr{G}_{i+1}\not= \emptyset) \tau_{p_1}^{\Lambda}(A_i,A_{i+1}) \geq 4 h^{c \delta^4} \cdot \left(\frac{1}{2}\right)^2 \geq h^{c_1\delta/2},
      \end{equation}
      and similarly that
      \begin{equation} \label{eq:first_and_last_ghosts_can_connect_to_u_and_v}
            \mathbb Q_h \otimes \p \left( 
u \xleftrightarrow[p_1]{\,\Lambda\;} \mathscr{G}_{1}
             \right) \geq \frac{1}{2}\tau_{p_1}^\Lambda(A_1)\quad \text{and} \quad \mathbb Q_h \otimes \p \left( 
\mathscr{G}_n \xleftrightarrow[p_1]{\,\Lambda\;} v
              \right) \geq \frac{1}{2}\tau_{p_1}^\Lambda(A_n).
      \end{equation}
      Let $p_{3/2}=  \sprinkle(p_1;\delta/2)$, so that $\delta(p_1,p_{3/2})=\frac{1}{2}\delta$. Using \eqref{eq:consecutive_ghosts_can_connect}, \cref{lem:ghost_in_a_box_main} implies that
      \begin{equation} \label{eq:preparing_for_big_ghost_union_bound_1}
            \mathbb Q_{h^c} \otimes \p (\mathscr{G}_i \xleftrightarrow[p_{3/2}]{B_r(\Lambda)} \mathscr{G}_{i+1} ) \geq 1- h^{c_1 \delta/2} \geq 1- h^{c\delta}.
      \end{equation}
    Since we also have by choice of $c$ that
      \[
            \tau_{p_{3/2}}^{B_r(\Lambda)}(A_i) \geq \tau_{p_1}^{\Lambda}(A_i) \geq \tau_{p_1}^{\Lambda}(A_i\cup A_{i+1}) \geq 4 h^{c \delta^4} \geq (h^{c_1})^{c_2\delta/2},
      \]
      we may apply \cref{lem:ghost_in_box_transitivity} (with $[p_1,p_2,h,X,Y,A,\Lambda] := [p_{3/2},p_2,h^{c_1},\{u\},\mathscr{G}_{i+1},A_i,B_r(\Lambda)]$) to deduce that if $h^{c_1} \leq 1/d$ then
      \begin{equation} \label{eq:preparing_for_big_ghost_union_bound_2}
            \mathbb Q_{h^c} \otimes \p\left( u \xleftrightarrow[p_2]{B_{2r}(\Lambda)} \mathscr{G}_i \xleftrightarrow[p_{3/2}]{ B_r(\Lambda) } \mathscr{G}_{i+1} \text{ but } u 
            \nxleftrightarrow[p_2]{B_{2r}(\Lambda) } \mathscr{G}_{i+1}
             \right) \leq 4(h^{c_1})^{ c_3 (\delta/2)^3 }= 4 h^{c_4 \delta^3}
      \end{equation}
      where $c_4=c_4(d,D)=(c_1\cdot c_3)/8$. (Note that using $h^{c_1}$ instead of $h$ changes the hypotheses on $r$, but this is not a problem since the hypotheses on $r$ associated to $h^{c_1}$ are weaker than those associated to $h$.)
      A similar application of \cref{lem:ghost_in_box_transitivity} (with $[p_1,p_2,h,X,Y,A,\Lambda] := [p_{3/2},p_2,h^{c_1},\{u\},\{v\},A_n,B_r(\Lambda)]$) yields that if $h^{c_1} \leq 1/d$ then
      \begin{equation} \label{eq:preparing_for_big_ghost_union_bound_3}
            \mathbb Q_{h^{c_1}} \otimes \p
            \left( u \xleftrightarrow[p_2]{B_{2r(\Lambda)}} \mathscr G_n \xleftrightarrow[p_{3/2}]{B_{r(\Lambda)}} v \text{ but } u \nxleftrightarrow[p_2]{B_{2r(\Lambda)}} v \right) \leq 4h^{c_4 \delta^3}.
      \end{equation}
    Now, we have by a union bound that
      \begin{multline}
            \p\left( u \xleftrightarrow[p_2]{B_{2r}(\Lambda)} v \right) \geq \mathbb Q_{h^{c_1}} \otimes \p \left( u \xleftrightarrow[p_{3/2}]{B_r(\Lambda)} \mathscr{G}_1 \xleftrightarrow[p_{3/2}]{B_r(\Lambda)} \ldots \xleftrightarrow[p_{3/2}]{B_r(\Lambda)} \mathscr{G}_n \xleftrightarrow[p_{3/2}]{B_r(\Lambda)} v \right) \\
            - \sum_{i=1}^{n-1} \mathbb Q_{h^{c_1}} \otimes \p\left( u \xleftrightarrow[p_2]{B_{2r}(\Lambda)} \mathscr{G}_i \xleftrightarrow[p_{3/2}]{B_r(\Lambda)} \mathscr{G}_{i+1} \text{ but } u \nxleftrightarrow[p_2]{B_{2r}(\Lambda)} \mathscr{G}_{i+1} \right)
            \\
            -\mathbb Q_{h^{c_1}} \otimes \p\left( u \xleftrightarrow[p_2]{B_{2r}(\Lambda)} \mathscr{G}_n \xleftrightarrow[p_{3/2}]{B_r(\Lambda)} v \text{ but } u \nxleftrightarrow[p_2]{B_{2r}(\Lambda)} v \right).
\end{multline}
Using \eqref{eq:first_and_last_ghosts_can_connect_to_u_and_v}, \eqref{eq:preparing_for_big_ghost_union_bound_1}, and the Harris-FKG inequality to bound the first term and \eqref{eq:preparing_for_big_ghost_union_bound_2} and \eqref{eq:preparing_for_big_ghost_union_bound_3} to control the error terms, we obtain that there exists a universal positive constant $c_5$ and positive constants $c_6$ and $C$ depending only on $d$ and $D$ such that if $h^c \leq 1/d$ then
      \[\begin{split}
            \p\left( u \xleftrightarrow[p_2]{B_{2r}(\Lambda)} v \right) &\geq \frac{\tau_{p_1}^\Lambda(A_1)}{2} \left( 1 - h^{c\delta} \right)^{n-1} \frac{\tau_{p_1}^\Lambda(A_n)}{2} - n \cdot 4 h^{ c_4 \delta^3 }\\
            &\geq c_5\left[ 1 - C nh^{c_6\delta^3} \right] \tau_{p}^\Lambda(A_1) \tau_p^\Lambda(A_n),
      \end{split}\]
      where we used that $\tau_{p_1}^\Lambda(A_1) \tau_{p_1}^\Lambda(A_n) \geq h^{2c \delta^4} \geq h^{c_1c_3 \delta^3 / 2^4}$ to absorb the error term into the prefactor.
      The proposition follows easily since the vertices $u \in A_1$ and $v \in A_n$ were arbitrary.
\end{proof}

\subsection{Graphs of high growth and the implication (\ref{implication:full-space})}

\label{subsec:high_growth_proof}

We now apply \cref{prop:snowballing} to prove the implication \eqref{implication:full-space} of the main induction step \cref{prop:complicated_induction_statement}.
We state the implication without including the burn-in term in $\delta_0$; this will not cause problems since the statement is stronger without this term than with it.

\begin{prop}[The implication (\ref{implication:full-space})] \label{prop:implication_F}
For each $d\in \mathbb{N}$ there exists a constant $N=N(d)\geq 16$ such that the following holds.
Let $G$ be an infinite, connected, unimodular transitive graph with vertex degree $d$, let $p_0 \in (0,1)$, and let $n_0 \geq 16$.
Let
$\delta_0 = (\log \log n_0)^{-1/2}$,
define sequences $(n_i)_{i\geq 1}$ and $(\delta_i)_{i\geq1}$ recursively by
\[
      n_i := \exp^{\circ 3}\left( \log^{\circ 3}(n_0) + i \log 9 \right) \quad \text{and} \quad \delta_i := (\log \log n_i)^{-1/2}
\]
and let $(p_i)_{i\geq 1}$ be an increasing sequence of probabilities satisfying $p_{i+1}\geq \sprinkle(p_i;\delta_i)$ for each $i\geq 0$. Let $n_{-1} := (\log n_0)^{1/2}$.
For each $i\geq 0$ define the statement
\begin{align*}\text{\emph{\textsc{Full-Space}}}(i) &= \Biggl(\p_{p_i} ( u \leftrightarrow v ) \geq \exp\left[-(\log \log n_{i})^{1/2}\right] \text{ for all $u,v \in B_{n_i}$}\Biggr) 
\intertext{and for each $i\geq 1$ define the statement}
\text{\emph{\textsc{Corridor}}}(i) &= \Biggl(\kappa_{p_{i}}( e^{[\log m]^{10}}, m ) \geq \exp\left[-(\log \log n_{i})^{1/2}\right] \text{ for every $m \in \mathscr{L}(G,20) \cap [n_{i-2},n_{i-1}]$}\Biggr).
\end{align*}
If $n_0 \geq N$ and $p_0 \geq 1/d$ then the implication
      \begin{align}
             \Bigl[ \upsc{Full-Space}(i) \wedge \upsc{Corridor}(i+1) \Bigr] &\implies \Bigl[ \text{\emph{\textsc{Full-Space}}}(i+1) \vee (p_{i+1} \geq p_c) \Bigr] \label{implication:full-space_restate} \tag{F}
      \end{align}
holds for every $i\geq 0$.
\end{prop}


\begin{proof}[Proof of \cref{prop:implication_F}]
Fix $i\geq 0$ and suppose that $\upsc{Full-Space}(i)$ and $\upsc{Corridor}(i+1)$ both hold. If $n_i \in \mathscr{L}(G,20)$, then $e^{(\log n_i)^{10}} \geq 2n_{i+1}=2e^{(\log n_i)^9}$ whenever $N$ is larger than some universal constant, so that if $u,v\in B_{n_{i+1}}$ then $d(u,v) \leq 2 n_{i+1}\leq e^{(\log n_i)^{10}}$ and
\[
\p_{p_{i+1}}(u \leftrightarrow v) \geq \kappa_{p_{i+1}}\Bigl(e^{[\log n_i]^{10}},n_i\Bigr) \geq e^{-(\log \log n_{i+1})^{1/2}} \qquad \text{ for every $u,v \in B_{n_{i+1}}$.}
\] 
That is, $\upsc{Corridor}(i+1)$ trivially implies $\upsc{Full-Space}(i+1)$ whenever $n_i \in \mathscr{L}(G,20)$.
%
%
Now suppose that $n_i \not\in \mathscr{L}(G,20)$ and suppose that $\upsc{Full-Space}(i)$ holds, so that
\[
\op{Gr}(n_i) \geq \exp((\log (n_i^{1/3}))^{20})=:h \qquad \text{ and } \qquad \tau_{p_i}( B_{n}(u_i) )  \geq \exp\left[- (\log \log n_i)^{1/2}\right].
\]
 Fix two arbitrary vertices $u,v\in B_{n_{i+1}}$ and let $u = u_1 , u_2 ,\ldots , u_k = v$ be the vertices in a geodesic from $u$ to $v$. It follows from the Harris-FKG inequality that for all $j$,
      \begin{align*}
            \tau_{p_i}( B_{n}(u_j) \cup B_{n}(u_{j+1}) ) \geq 
            \tau_{p_i}( B_{n+1}(u_j) )  \geq p_i^2 \tau_{p_i}( B_{n+1}(u_j) ) \geq d^{-2}\exp\left[- (\log \log n_j)^{1/2}\right].
      \end{align*}
      Let $c_1$, $c_2$, $c_3$ and $h_0=h_0(d)$ be the constants from \cref{prop:snowballing} applied with $D=1$ (so that $c_1$, $c_2$, and $c_3$ are universal) and let $N_1=N_1(d)$ be sufficiently large that $\exp(-(\log n)^{20}) \leq h_0$ for every $n\geq N_1$. There exists a constant $N_2 =N_2(d)\geq N_1$ such that if $n_i \geq n_0 \geq N_2$ then
      \[
      h^{c_1\delta_i^3}=\exp\left[- \frac{c_1}{3^{20}}  \frac{(\log n_i)^{20}}{(\log \log n_i)^{3/2}}\right] \leq c_3 n_{i+1}^{-1}
      \]
      and for all $j$,
      \[
      \tau_{p_i}( B_{n}(u_j) \cup B_{n}(u_{j+1}) ) \geq d^{-2}\exp\left[- (\log \log n_i)^{1/2}\right] \geq 4\exp\left[- \frac{c_1}{3^{20}}  \frac{(\log n_i)^{20}}{(\log \log n_i)^{2}}\right] = 4 h^{c_1\delta_i^4}.
      \]
      Thus, if $n_0\geq N_2$ then \cref{prop:snowballing} (applied with $D=1$ and $A_i = B_n(u_i)$) implies that for all $j$,
      \[
      \p_{p_{i+1}}(u \leftrightarrow v) \geq \tau_{p_{i+1}}(B_{n_i}(u),B_{n_i}(v)) \geq c_2 \tau_{p_i}(B_{n_i}(u))\tau_{p_i}(B_{n_i}(v)) \geq c_2 \exp\left[-2(\log \log n_i)^{1/2} \right],
      \]
      where the final inequality follows from the assumption that $\upsc{Full-Space}(i)$ holds. Since $u$ and $v$ were arbitrary vertices in $B_{n_{i+1}}$, it follows that there exists a constant $N_3=N_3(d)\geq N_2$ such that if $n_i \geq n_0\geq N_3$ then $\upsc{Full-Space}(i+1)$ holds as claimed. \qedhere

\end{proof}

As mentioned in the introduction, \cref{prop:implication_F} already allows us to conclude the proofs of \cref{thm:main,thm:p_c<1} under a mild uniform superpolynomial growth assumption.

\begin{cor}
\label{cor:high_growth_pc}
Let $G$ be an infinite, connected, unimodular transitive graph. If $\log \op{Gr}(r) > (\log r)^{20}$ for all sufficiently large $r$ then $p_c(G)<1$.
\end{cor}

\begin{cor}
\label{cor:high_growth_locality}
Let $(G_n)_{n\geq 1}$ be a sequence of infinite, connected, unimodular transitive graphs converging to some transitive graph $G$, and suppose that there exists $R$ such that
$\log \op{Gr}(r;G_n) > (\log r)^{20}$
for every $r\geq R$ and $n\geq 1$. Then $p_c(G_n)\to p_c(G)$
\end{cor}

\begin{proof}[Proof of \cref{cor:high_growth_pc,cor:high_growth_locality}] Observe that if $G$ satisfies $\log \op{Gr}(r;G) > (\log r)^{20}$ for every $r\geq n_0$ then the statement $\upsc{Corridor}(i)$ holds vacuously for every $i$ since $\mathscr{L}(G,20)\cap [n_0,\infty)$ is empty. Thus, these two corollaries follow from \cref{prop:implication_F} by the same argument used to deduce \cref{thm:main,thm:p_c<1} from \cref{prop:complicated_induction_statement}. (In fact the proof is slightly simpler since one no longer needs to control the burn-in.)
\end{proof}

\begin{rk}[Weaker growth conditions and the gap conjecture]
\label{remark:weaker_growth_assumptions}
The proof of \cref{cor:high_growth_pc,cor:high_growth_locality} extends straightforwardly to (sequences of) graphs satisfying much weaker growth conditions, such as
\begin{equation}
\label{eq:loglog_growth}
\log \op{Gr}(r) \geq c\log r (\log \log r)^{10}.
\end{equation}
It is plausible that this class (and indeed the class treated by \cref{cor:high_growth_pc,cor:high_growth_locality}) includes \emph{every} transitive graph of superpolynomial growth, so that the methods of this section would suffice to prove locality and non-triviality of $p_c$ for unimodular transitive graphs of superpolynomial growth. (One would still need the remainder of the paper to handle sequences of graphs of polynomial growth converging to a graph of superpolynomial growth, such as the Cayley graphs of the free step-$s$ nilpotent groups, which converge to trees as $s\to \infty$.) Indeed, one formulation of Grigorchuk's \emph{gap conjecture} (see \cite{grigorchuk2014gap}) states that there exist universal positive constants $c$ and $\gamma$ such that if $G$ is a Cayley graph of superpolynomial growth then 
$\op{Gr}(r) \geq e^{c r^\gamma}$ for every $r\geq 1$, and it seems reasonable to extend this conjecture to transitive graphs. Thus, the lower bound \eqref{eq:loglog_growth} required for our arguments to work is much smaller than what might plausibly hold universally for all transitive graphs of superpolynomial growth. On the other hand, the best known bound for the gap conjecture, due to Shalom and Tao \cite{shalom2010finitary}, states that there exists a universal constant $c$ such that
\begin{equation}
\label{eq:Shalom_Tao}
\log \op{Gr}(r)\geq c \log r (\log \log r)^c
\end{equation}
for every Cayley graph of superpolynomial growth and every $r\geq 1$. (The authors claim their proof should yield this estimate with the constant $c=0.01$, but do not carry out the necessary bookkeeping in the paper.)
Even if the strong form of the gap conjecture is false, there does not seem to be any reason to believe that the Shalom-Tao bound \eqref{eq:Shalom_Tao} is optimal, so that \eqref{eq:loglog_growth} may very well hold for all transitive graphs of superpolynomial growth.
\end{rk}

\section{Quasi-polynomial growth I: Building disjoint tubes}

\label{sec:tubes}


In this section we prove the geometric facts that we will later use to analyze percolation in the low-growth (a.k.a.\ quasi-polynomial growth) regime.
Let $d\geq 1$, let $G \in \cG_d^*$, and let $D \geq 1$ be a fixed parameter. We recall that the set of \textbf{low growth scales} $\mathscr{L}(G,D)$ is defined to be
\[
\mathscr{L}(G,D) = \Bigl\{n \geq 1: \log \op{Gr}(m) \leq (\log m)^{D} \text{ for all } m \in [n^{1/3} , n]\Bigr\}.
\]
It would suffice for all our applications to take e.g.\ $D=20$; we keep $D$ as a parameter for now to emphasize that the analysis carried out in this section and \cref{sec:low_growth_multiscale} works for arbitrarily large $D$. 

\medskip

Recall that a \textbf{tube} is defined to be a set of the form $B_r(\gamma)=\bigcup_i B_r(\gamma_i)$ where $\gamma$ is a path and $r \in (0,\infty)$ is a parameter we call the \textbf{thickness} of the tube; we call $\op{len}(\gamma)$ the \textbf{length} of the tube. We will also sometimes write $B(\gamma,r)=B_r(\gamma)$ to avoid writing large expressions in the subscript. (Strictly speaking the length and thickness of a tube depends on the pair $(\gamma,r)$ used to represent it, but we will not belabour this point further.) We would like to show that whenever $n$ is in the low growth regime, for all suitable sets $(A,B)$ of vertices at scale $n$, we can find many disjoint tubes from $A$ to $B$ that are reasonably thick and not unreasonably long. 

\begin{figure}
\center
\includegraphics[width=0.9\textwidth]{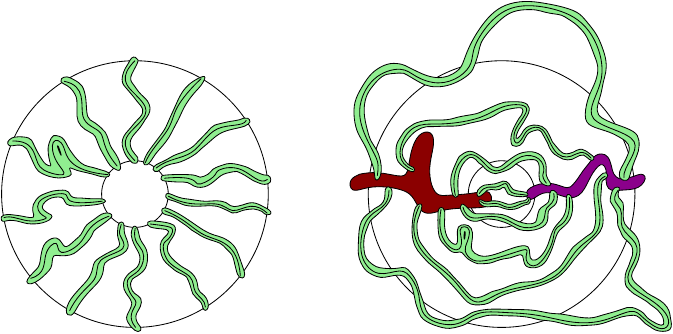} 
\caption{Schematic illustration of radial and annular tubes. Left: The $(k,r,\ell)$-plentiful \emph{radial} tubes condition means that we can find $k$ disjoint tubes crossing the annulus that all have thickness $r$ and length at most $\ell$. Right: the $(k,r,\ell)$-plentiful \emph{annular} tubes condition means that for any two crossings of the annulus, we can find $k$ disjoint tubes connecting the two crossings that all have thickness $r$ and length at most $\ell$; these tubes are \emph{not} required to stay inside the annulus.}
\label{fig:plentiful_tubes}
\end{figure}

\begin{defn}
Let $G$ be a connected transitive graph. Given $k,r,\ell\geq1$ (which need not be integers)
  and $n\geq 1$, we say that $G$ has \textbf{$(k,r,\ell)$-plentiful radial tubes at scale $n$} if there exists a family of paths $\Gamma$ in $G$ with $\abs{\Gamma} \geq k$ satisfying all of the following properties: 
\begin{itemize} 
      \item Every path $\gamma\in \Gamma$ starts in the sphere $S_{n}$ and ends in the sphere $S_{4n}$;
      \item For every two distinct paths $\gamma,\gamma' \in \Gamma$, the  tubes  $B(\gamma,r)$ and $B(\gamma',r)$ are disjoint.
      \item Every path $\gamma\in \Gamma$ has length at most $\ell$;
\end{itemize}
(In particular, the parameter $k$ controls the number of tubes, the parameter $r$ controls the thickness of the tubes, and the parameter $\ell$ controls the length of the tubes.)
Given $m\geq n \geq 1$, we say that a set of vertices $A \subseteq V$ is an \textbf{$(n,m)$ crossing} if it contains a path from $S_n$ to $S_{m}$.
We say that $G$ has \textbf{$(k,r,\ell)$-plentiful annular tubes at scale $n$} if for every pair of sets of vertices $(A,B)$ that are both $(n,3n)$  crossings\footnote{The constant $3$ that appears here could safely be replaced by any constant strictly larger than $2$. We need the constant to be strictly larger than $2$ to not cause problems with our application of \cref{prop:if_you_can_get_past_2n_you_can_get_away}.}, there exists a family of paths $\Gamma$ in $G$ with $\abs{\Gamma} \geq k$ satisfying all of the following properties:
\begin{itemize} 
      \item Every path $\gamma\in \Gamma$ starts in $A$ and ends in $B$;
         \item For every two distinct paths $\gamma,\gamma' \in \Gamma$, the tubes  $B(\gamma,r)$ and $B(\gamma',r)$ are disjoint.
      \item Every path $\gamma\in \Gamma$ has length at most $\ell$;
\end{itemize}
We say that $G$ has \textbf{$(k,r,\ell)$-plentiful tubes at scale $n$} if $G$ has both $(k,r,\ell)$-plentiful annular tubes and $(k,r,\ell)$-plentiful radial tubes at scale $n$. See \Cref{fig:plentiful_tubes} for an illustration. Given parameters $c,\lambda>0$, we say that $G$ has \textbf{$(c,\lambda)$-polylog-plentiful tubes} at scale $n$ if it has $(k,r,\ell)$-plentiful tubes with $k=(\log n)^{c\lambda}$, $r=n(\log n)^{-\lambda/c}$, and $\ell= n(\log n)^{\lambda/c}$. (NB: These tubes are very thick!)
\end{defn}

\begin{rk}
      The property of having $(k,r,\ell)$-plentiful tubes at scale $n$  gets stronger as $k$ and $r$ increase and gets weaker as $\ell$ increases; the property of having $(c,\lambda)$-polylog-plentiful tubes at scale $n$ gets stronger as $c$ increases but has no obvious monotonicity in the parameter $\lambda$. There is of course a trade-off between the number of tubes and their thicknesses, since e.g.\ if the $(n,3n)$ crossings $A$ and $B$ are segments of geodesics from the origin between distances $n$ and $3n$ then we cannot have more than $2n/r$ disjoint tubes of thickness $r$ connecting $A$ and $B$. In our applications in \cref{sec:low_growth_multiscale} we will want to use families of tubes that have thickness $n(\log n)^{-O(1)}$ and length $n(\log n)^{O(1)}$, which is why we phrase \cref{prop:existence_of_interval_with_controlled_tubes} in terms of polylog-plentiful tubes. As will be clear later in the section, our constructions allow for many other possible trade-offs between $k$ and $r$ which may be useful in future applications.
\end{rk}

Ideally, we would like to say that there is some constant $c \in (0,1)$ such that for every $\lambda$, $G$ has $(c,\lambda)$-polylog-plentiful tubes at every sufficiently large scale $n\in \mathscr{L}(G,D)$. This is, however, slightly stronger than
  what we have been able to prove. (Fortunately it is also stronger than we need for our applications!)
We instead establish the slightly weaker statement that for every scale $n\in \mathscr{L}(G,D)$, there exists a large interval of scales not much smaller than $n$ on which we have plentiful tubes. We now give the precise statement, the proof of which takes up the rest of this section.

\begin{prop}[Quasi-polynomial growth yields a large range of consecutive scales with polylog-plentiful tubes]
\label{prop:existence_of_interval_with_controlled_tubes}
      For each $D \in [1,\infty)$, $\lambda \in [1,\infty)$, and $d \geq 1$ there exist positive constants $c(d,D) \in (0,1)$ and $n_0=n_0(d,D,\lambda)$ such that the following holds. Let $G$ be an infinite, connected, unimodular transitive graph with vertex degree $d$ that is not one-dimensional. 
       For each integer $n\geq n_0$ with $n \in \mathscr{L}(G,D)$ there exist integers $n^{1/3} \leq m_1 \leq m_2 \leq n$ with $m_2 \geq m_1^{1+c}$
         such that $G$ has $(c,\lambda)$-polylog-plentiful tubes at every scale $m_1 \leq m \leq m_2$.
\end{prop}

The constant $c=c(d,D)$ can be thought of as an ``exchange rate'', governing the cost to trade between the number of tubes and their thicknesses, while $\lambda$ is the parameter we vary to make this trade. \emph{It is very important that the constant $c$ does not depend on $\lambda$!}


\begin{rk}
No non-trivial statements about plentiful annular tubes can be made without some kind of growth upper bound (or other geometric assumption), since  the $3$-regular tree does not have plentiful $(k,r,\ell)$-plentiful annular tubes for any $k,r,\ell>1$. Similarly, the assumption that the graph is not one-dimensional is needed since the line graph $\mathbb{Z}$ does not have $(k,r,\ell)$-plentiful radial or annular tubes for any $k,r,\ell>1$. (On the other hand, the assumption of unimodularity is redundant since it is implied by the existence of large scales with subexponential growth \cite[Section 5.1]{hutchcroft2019locality}.)
\end{rk}

To prove this proposition, we split into two cases according to the \emph{rate of change} of growth in the given interval. More precisely, we will split according to whether or not there exists $n$ in a suitable initial segment of the interval such that $G$ satisfies the small-tripling condition $\op{Gr}(3n) \leq 3^{5} \op{Gr}(n)$ at scale $n$, where the constant $5$ could be replaced by any other constant strictly larger than $4$. (This is related to the fact that four is the critical dimension for two independent random walks to intersect infinitely often almost surely.) Our proofs in the two cases are completely different from one another. 

\medskip

In the first case, when there does exist such an $n$, we will build our disjoint tubes by applying the structure theory of transitive graphs of polynomial volume growth \cite{breuillard2011structure,MR4253426,EHStructure}. Informally, this theory guarantees that, for a large interval of scales, our graph looks approximately like the Cayley graph of a finitely presented group whose relations are generated by cycles of diameter much smaller than the given scale. We will use these techniques to prove the following proposition.


\begin{prop}[Slow tripling yields plentiful tubes] \label{prop:structure_theory_conclusion}
    For each $d\geq 1$ and $\kappa <\infty$ there exist positive constants $c=c(d,\kappa)$, $C=C(d,\kappa)$, and $n_0=n_0(d,\kappa)$ such that
%
      if $G$ is an infinite, connected, transitive graph of vertex degree $d$ that is not one-dimensional and $n \geq n_0$ is such that
      $\op{Gr}(3 n) \leq 3^{\kappa} \op{Gr}(n)$, then there exists a set $A \subset [n,\infty)$ with $|A|\leq C$ such that for each $k \geq 1$, 
        $G$ 
        has $(ck,ck^{-1} m,C k^C m)$-plentiful tubes at every scale $m \geq Ckn$ such that $m \notin \bigcup_{a\in A}[a,2ka]$.
\end{prop}

In the second case, when there does \emph{not} exist such a scale $n$ with small tripling, we will prove that the desired plentiful tubes condition holds using random walks. 
More specifically, we will apply an estimate of \cite{MR3395463}, which can be thought of as a ``quantitative weak elliptic Harnack inequality'' for graphs of subexponential growth. Under our low-growth assumption, this inequality implies that two random walks on $G$ started at distance $n' \in [n^{1/3},n^{0.9}]$, say, can be coupled to coincide with good probability by the time they reach distance $n' (\log \op{Gr}(n'))^{O(1)}$. On the other hand, the assumption that the growth is rapidly increasing (in an at-least-five-dimensional fashion) lets us prove that two independent pairs of these coupled random walks are unlikely to have their tubes intersect. 
(This will be proven using fairly standard random walk techniques, most notably the isoperimetric inequality of Coulhon and Saloff-Coste \cite{MR1232845} and the heat kernel bounds resulting from this inequality together with the work of Morris and Peres \cite{morris2005evolving}.)
 Unfortunately these walks will have length of order $(n')^2(\log \op{Gr}(n'))^{O(1)}$, which is larger than we want by a factor of $n'$. This can be fixed by a simple coarse-graining argument, using the diffusivity of random walks on low-growth graphs, to replace the random walk by a shorter path with  essentially the same tube around it. 
We will use these techniques to prove the following proposition.

\begin{prop}[Fast tripling and quasi-polynomial growth yield plentiful tubes] \label{prop:rw_conclusion}
     Let $G$ be an infinite connected unimodular transitive graph with vertex degree $d$, let $D,\lambda \geq 1$ and let $\eps>0$. There exist positive  constants $c=c(d,D,\eps)$ and $n_0=n_0(d,D,\lambda,\eps)$ such that 
       if $n\geq n_0$ satisfies
      \[
            \op{Gr}(m) \leq e^{(\log m)^D} \quad \text{and} \quad \op{Gr}(3 m) \geq 3^{5} \op{Gr}(3m) \qquad \text{ for every $n^{1-\eps} \leq m \leq n^{1+\eps}$}
      \]
      then $G$ has $(c,\lambda)$-polylog-plentiful tubes at scale $n$.
\end{prop}


We prove \cref{prop:structure_theory_conclusion} in \cref{subsection:structure_theory} and \cref{prop:rw_conclusion} in \cref{subsection:rws}. Before doing this, we note that \cref{prop:existence_of_interval_with_controlled_tubes} follows easily from these two propositions.

\begin{proof}[Proof of \cref{prop:existence_of_interval_with_controlled_tubes} given \cref{prop:structure_theory_conclusion,prop:rw_conclusion}]

Let $c_1(d,5)$, $C(d,5)$, and $n_1(d,5)$ be the constants from \cref{prop:structure_theory_conclusion} with $\kappa := 5$. Let $c_2(d,D,0.1)$ and $n_2(d,D,\lambda,0.1)$ be the constants from \cref{prop:rw_conclusion} with $\eps := 0.1$. We may assume that $c_1 \vee c_2 \leq 1/2$ and $C \geq 2$. Suppose that some $n\in \mathscr{L}(G,D)$ is large with respect to $d,D,\lambda$, satisfying in particular $ n\geq (n_1 \vee n_2)^3$. If there exists $n' \in [n^{1/3},n^{10/11}]$ such that $\op{Gr}(3n') \leq 3^{5} \op{Gr}(n')$, then we may apply \cref{prop:structure_theory_conclusion} with $k := (\log n)^{\lambda}$ to obtain an interval of scales $[m_1,m_2] \subseteq [n^{1/3},n]$ with $m_2\geq m_1^{1.1}$ on which $G$ has $(1/(2C),\lambda)$-polylog-plentiful tubes on every scale. Otherwise, since each $n' \in [n^{0.4},n^{0.8}]$ satisfies $[(n')^{0.9},(n')^{1.1}] \subseteq [n^{1/3},n^{10/11}]$, we can apply 
\cref{prop:rw_conclusion} to each such scale to obtain that $G$ has $(c_1,\lambda)$-plentiful tubes at every scale in the interval $[n^{0.4},n^{0.8}]$. This is easily seen to imply the claim in either case. 
\end{proof}

\subsection{Using the structure theory of approximate groups} \label{subsection:structure_theory}

The goal of this subsection is to prove \cref{prop:structure_theory_conclusion}. Let us first give some relevant context.  Given a graph $G$, a vertex $v \in V(G)$, and a radius $r \geq 0$, we define the \textbf{exposed sphere} $S_{r}^{\infty}(v)$ to be the set of vertices $u \in S_r(v)$ such that there exists an infinite self-avoiding path started at $u$ that never returns to $B_r(v)$ after its first step. When $G$ is transitive, we set $S_r^{\infty} := S_r^{\infty}(o)$. In \cite{contreras2022supercritical}, the authors applied results of Timar \cite{timar2007cutsets} to obtain geometric control of exposed spheres in transitive graphs of polynomial growth using the fact that (by Gromov's theorem \cite{MR623534} and Trofimov's theorem \cite{MR735714}) these graphs are quasi-isometric to Cayley graphs of nilpotent groups, which are finitely presented.

\medskip

In this section, we will run a similar argument to build our disjoint tubes under the hypotheses of \cref{prop:structure_theory_conclusion}. While these hypotheses suffice to guarantee polynomial growth by the results of \cite{breuillard2011structure,MR4253426} (discussed in detail below), a key technical difference between our analysis and that of \cite{contreras2022supercritical} is that we must run all our arguments in a more finitary, quantitative way;
we need to build our disjoint tubes at a scale not much larger than the scale where we are assumed to witness relative polynomial growth.  This requires us to engage more deeply with the structure theory of approximate groups than was necessary in \cite{contreras2022supercritical}. Indeed, rather than using Gromov's theorem and Trofimov's theorem, we will instead apply finitary versions of these theorems due to Breuillard, Green, and Tao \cite{breuillard2011structure} and Tessera and Tointon \cite{MR4253426}. Moreover, we will also apply a \emph{new} structure theoretic result proven in our paper \cite{EHStructure}, which can be thought of as a ``uniform'' version of the statement that groups of polynomial growth are finitely presented.

\begin{rk}
It will be convenient in this section to let $\Gamma$ denote a group. This will not conflict with the $\Gamma$ used to denote a set of disjoint tubes, which we will not use in this section.
\end{rk}



\medskip

\noindent
\textbf{Structure theory.} We now state the main structure-theoretic results we will use after reviewing some relevant definitions. We begin with Tessera and Tointon's finitary structure theorem for vertex-transitive graphs of low growth \cite{MR4253426}; this theorem builds on Breuillard, Green, and Tao's  structure theorem for approximate groups \cite{breuillard2011structure} as well as Carolino's extension of this theorem to locally compact approximate groups \cite{carolino2015structure}.
Recall that a function $\phi:V_1\to V_2$ between the vertex sets of two graphs $G_1=(V_1,E_1)$ and $G_2=(V_2,E_2)$ is said to be an $(\alpha,\beta)$\textbf{-quasi-isometry} (a.k.a.\ \textbf{rough isometry}) if 
\[\alpha^{-1} d(x,y) -\beta \leq d(\phi(x),\phi(y)) \leq \alpha d(x,y) +\beta\]
for every $x,y\in V_1$ and every vertex $z\in V_2$ is within distance at most $\beta$ of $\phi(V_1)$; note that the second property holds automatically if $\phi$ is surjective.
Given a transitive graph $G$ and a subgroup $H \subseteq \Aut(G)$, we write $G/H$ for the associated quotient graph. If $H$ is a normal subgroup of $\Aut(G)$ then the action of $\Aut(G)$ on $G$ descends to a transitive action of $\Aut(G)$ on $G/H$ (see \cite[Section 3]{MR4253426}), and we write $\Aut(G)_{G/H}$ for the image of $\Aut(G)$ in $\Aut(G/\Gamma)$ induced by this action. Note that we view $\Aut(G)$ as a \emph{topological} group where convergence is given by pointwise convergence (see \cite[\S 4]{MR4253426}).


 \begin{thm}[Finitary structure theory of transitive graphs of polynomial growth]
 \label{thm:tessera-tointon}
For each $K\geq 1$ there exist constants $n_0=n_0(K)$ and $C=C(K)$ such that the following holds. Let $G$ be a connected (locally finite) (vertex-)transitive graph with a distinguished vertex $o$, and suppose that there exists $n \geq n_0$ such that $\op{Gr}(3n) \leq K\op{Gr}(n)$. 
 Then there exists a compact normal subgroup $H \triangleleft \op{Aut}(G)$ such that:

\begin{enumerate}
    \item Every fibre of the projection $\pi:G \rightarrow G/H $ has diameter at most $Cn$.
    
    \item $\Aut(G)_{G/H }$ can be canonically identified with $\Aut(G)/H $.
    
    \item $\Aut(G)/H$ has a nilpotent normal subgroup $N$ of rank, step and index at most $C$.
    
    \item The set $S = \{g \in \Aut(G)_{G/H } : d_{G/H } (g(H o), H o) \leq 1\}$ is a finite symmetric generating set for $\Aut(G)/H$.
    
    \item Every vertex stabiliser of the action of $\Aut(G)/H$ on $G/H $ has cardinality at most $C$.
    
    \item If for each $v\in G$ we let $g_v\in \Aut(G)/H$ be such that $g_v(\pi(e))=\pi(v)$ then $v\mapsto g_v$ is a $(1,Cn)$-quasi-isometry from $G$ to $\op{Cay}(\Aut(G)/H,S)$.
    \item If $\op{Gr}'$ denotes the growth function of the Cayley graph $\op{Cay}(\Aut(G)/H,S)$ then
    \[
    \frac{\op{Gr}(m_2)}{C\op{Gr}(m_1+Cn)} \leq  \frac{\op{Gr}'(m_2)}{\op{Gr}'(m_1)} \leq \frac{C\op{Gr}(m_2+Cn)}{\op{Gr}(m_1)}
    \]
    for every $m_1,m_2 \in \mathbb{N}$.
    \item The growth bound $\op{Gr}(m_2) \leq C (m_2/m_1)^C \op{Gr}(m_1)$ holds for every $m_2\geq m_1 \geq m$.
\end{enumerate}
\end{thm}

\begin{proof}[Proof of \cref{thm:tessera-tointon}]
The first five items of the theorem are essentially equivalent to \cite[Theorem 2.3]{MR4253426} (although that theorem is slightly more general as it allows one to replace $\Aut(G)$ with any other transitive group of automorphisms of $G$).
In their original statement of the theorem, Tessera and Tointon do not explicitly identify the rough isometry from $G$ to $\op{Cay}(\Aut(G)/H,S)$, but the fact we can take it to be of the form above is implicit in their proof. 
 Item 7 is implied by \cite[Proposition 9.1]{MR4253426}, while Item 8 follows from Item 7 together with \cite[Theorem 1.1]{MR3439705}.
\end{proof}

\begin{rk}
\label{remark:S_bounded}
The set $S$ has size equal to the union of the stabilizers of the vertices $\{u \in G/H : u \sim v\}$, so that $|S|\leq C(\deg(o)+1)$.
\end{rk}

\begin{rk}
For $K=3^\kappa$ with $\kappa$ an integer, the growth bound $\op{Gr}(m_2)\leq C(m_2/m_1)^C\op{Gr}(m_1)$ can be improved to the sharp bound $\op{Gr}(m_2)\leq C(m_2/m_1)^\kappa\op{Gr}(m_1)$ under the stronger assumption that the graph satisifes an \emph{absolute} growth bound of the form $\op{Gr}(n)\leq \eps_K n^{\kappa+1}$ at a sufficiently large scale and for a sufficiently small constant $\eps_K>0$ \cite[Corollary 1.5]{MR4253426}. This strong bound is not implied by the small-tripling condition (which is the relevant condition for our applications) as shown in \cite[Example 1.11]{MR3658282}. (Indeed, this example suggests that the small-tripling condition $\op{Gr}(3m)\leq 3^\kappa \op{Gr}(m)$ should not imply any bound on the limiting growth dimension stronger than $O(\kappa^2)$. Optimal bounds on growth implied by small tripling will be established in forthcoming work of Tessera and Tointon.)
\end{rk}


    
    
    
    
    

\noindent \textbf{Uniform finite presentation.}
We next state our theorem on uniform finite presentation proven in \cite{EHStructure}.
Given a set of elements $A$ in a group $\Gamma$, we define $\llangle A \rrangle$ to be the normal subgroup of $\Gamma$ generated by $A$ and define $\bar A=A\cup\{\mathrm{id}\}\cup A^{-1}$. Consider a group $\Gamma$ with a finite generating set $S$, let $F_S$ be the free group on $S$ and let $\pi : F_S \to G$ be the associated group homomorphism with kernel $R$. Since $\Gamma \cong F_S / R$, we can think of the sequence of quotients $(\Gamma_r)_{r\geq 1}$ defined by $\Gamma_r:=F_S / \llangle \bar S^r \cap R \rrangle$  as being finitely presented approximations to $\Gamma$, since
$\Gamma_r$ admits a finite presentation $\Gamma_r=\langle S \mid R_r \rangle=F_S/\llangle R_r\rrangle$ with $R_r = \bar S^r \cap R \subseteq \bar S^r$.
 These approximations have the property that the Cayley graphs $\op{Cay}(\Gamma_r,S)$ and $\op{Cay}(\Gamma,S)$ have isomorphic $((r/2)-1)$-balls. We record this fact in the following lemma, which is taken from \cite[Lemma 5.6]{EHStructure}. (Although it is stated there only in the case that $r$ is a power of 2, the same proof works for arbitrary $r$.)

\begin{lem}\label{lem:small_presentation_means_balls_look_similar}
Let $\Gamma$ be a group with a finite generating set $S$. For all $i \geq 1$, the quotient map $\Gamma_{r} \to \Gamma$ induces a map of the associated Cayley graphs that restricts to an isomorphism between the balls of radius $\lfloor r/2\rfloor-1$.
\end{lem}

The main result of \cite{EHStructure} can be stated as follows:

\begin{thm}[{\cite{EHStructure}, Theorem 1.1}]
\label{thm:induced_subgroup_doesnt_keep_changing}
For each $K,d<\infty$ there exist constants $n_0=n_0(K)$ and $C=C(K,d)$ such that if $\Gamma$ is a group and $S$ is a finite generating set for $\Gamma$ with $|S|\leq d$ whose growth function $\operatorname{Gr}$ satisfies $\operatorname{Gr}(3n)\leq K \operatorname{Gr}(n)$ for some integer $n\geq n_0$ then
\[
\#\Bigl\{k\in \mathbb{N} : k \geq \log_2 n \text{ and $\llangle R_{2^{k+1}} \rrangle \neq  \llangle  R_{2^k} \rrangle $} \Bigr\}\leq C.
\]
\end{thm}

For our purposes, the main output of \cref{thm:tessera-tointon,thm:induced_subgroup_doesnt_keep_changing} is that if the small tripling condition $\operatorname{Gr}(3n)\leq K \operatorname{Gr}(n)$ holds at some sufficiently large $n$, then at ``most'' scales $r\geq n$ the graph ``looks like'' the Cayley graph of a finitely presented group with relations generated by words of length much smaller than $r$.



\medskip

\noindent \textbf{Using the structure theory to build disjoint tubes.} It remains to apply \cref{thm:tessera-tointon,thm:induced_subgroup_doesnt_keep_changing} to prove \cref{prop:structure_theory_conclusion}. When working with a group $\Gamma$ and a finite generating set $S$ of $\Gamma$, we will continue to use the notation $(\Gamma_r)_{r\geq 1}$ and $(R_r)_{r\geq 1}$ as defined above. 

\medskip

Let us introduce some more definitions. Let $G=(V,E)$ be a graph, and consider the set $\{0,1\}^E$ with addition modulo $2$ as a vector space over $\mathbb{Z}_2$.
 If $G = \op{Cay}(\Gamma,S)$ for some group $\Gamma=\langle S | R \rangle$ with finite generating set $S$ and (not necessarily finitely generated) relation group $R \triangleleft F_S$, then we can identify every oriented cycle started at the origin with a word in $R$. Under this identification, we see that if $R$ is generated as a normal subgroup by words of length at most $r$ in $F_S$, then every cycle in $G$ can be written as a mod-$2$ sum of cycles of length at most $r$. Note that this leads to a notion of finite presentation for graphs that are not Cayley graphs, namely that their cycle space is generated as a $\mathbb{Z}_2$-vector space by the cycles of some finite length $r<\infty$; see \cite{timar2007cutsets} for further details.

 \medskip

We will rely crucially on the following lemma, which is essentially due to Tim\'ar \cite{timar2007cutsets}.
Recall that an infinite graph $G = (V,E)$ is \textbf{one-ended} if for every finite set of vertices $W \subseteq V$, the graph $G \backslash W$ has exactly one infinite component; groups of polynomial growth are one-ended if and only if they are not one-dimensional. (Note that when $G$ is the Cayley graph of an infinite finitely-generated group, the property of being one-ended is independent of the choice of finite generating set, so that one can sensibly refer to a \emph{group} as being one-ended without specifying a generating set.) 

\begin{lem}
\label{lem:short_generating_loops_yield_tubes}
Let $G$ be an infinite, connected, one-ended transitive graph. If $r \in \mathbb N$ has the property that every cycle in $G$ is equal to a sum of cycles of length at most $r$, then for every $k,n \in \mathbb N$ with $r \leq k \leq n$ and every $u,v \in S_n^{\infty}$ there exists a path from $u$ to $v$ that is contained in $\bigcup_{x \in S_n^{\infty}} B_{2k}(x)$ and has length at most $3k \abs{B_{3n}}/\abs{B_k}$.
\end{lem}

\begin{proof}
  This statement is implicit in the proofs of \cite[Lemma 2.1 and 2.7]{contreras2022supercritical}, and is an easy consequence of \cite[Theorem 5.1]{timar2007cutsets}.
\end{proof}



We will also need two more elementary geometric facts. The first states that if a path travels from a sphere $S_r$ to $S_{2r+1}$, then it must pass through the exposed sphere $S_{r}^{\infty}$.



\begin{prop}[{\cite[Proposition 5]{MR3291630}}] \label{prop:if_you_can_get_past_2n_you_can_get_away}
      Let $G$ be an infinite connected transitive graph and let $r \in \mathbb N$. Every path that starts in $S_r$ and ends in $S_{2r+1}$ contains a vertex in $S_{r}^{\infty}$. 
\end{prop}

The next lemma lets us pass disjoint tubes through quasi-isometries with all relevant quantities changing in a controlled way.

\begin{lem}\label{lem:quasi_isometry_tubes}
Let $G=(V,E)$ and $G'=(V',E')$ be two graphs and let $\phi:V\to V'$ be an $(\alpha,\beta)$-quasi isometry for some $\alpha,\beta \geq 1$. 
Let $u,v\in V$ and suppose that $x,y\in V'$ satisfy $d(x,\phi(u))\leq \beta$ and $d(y,\phi(v)) \leq \beta$. For each path $\gamma'$  from $x$ to $y$  in $G'$,  there exists a path $\gamma$ from $u$ to $v$ in $G$ such that $\op{len}(\gamma) \leq 10\alpha (\op{len}(\gamma')+\beta)$ and  
\[\phi(B_r(\gamma)) \subseteq B_{\alpha r + (5\alpha^2+2)\beta}(\gamma')\]
for every $r\geq 0$.
\end{lem}

\begin{proof}[Proof of \cref{lem:quasi_isometry_tubes}]
Let $\psi:V'\to V$ be a function such that $\psi(x)=u$, $\psi(y)=v$, and $d(\phi(\psi(z)),z)\leq \beta$ for every $z\in V$; such a function exists by the definition of an $(\alpha,\beta)$-quasi-isometry and our assumptions on $u$, $v$, $x$, and $y$. Let $\ell=\lceil \op{len}(\gamma')/\lceil \beta \rceil \rceil $, and let the sequence $(n_i)_{i=0}^\ell$ be defined by $n_i = \lceil \beta \rceil i $ for $i<\ell$ and $n_\ell = \op{len}(\gamma')$. For each $0\leq i \leq \ell$ let $x_i=\gamma_{n_i}'$ and let $u_i=\psi(x_i)$, so that $u_0=u$ and $u_\ell=v$.  For each $0\leq i <\ell$, the points $x_i$ and $x_{i+1}$ have distance at most $\lceil \beta\rceil \leq 2\beta$ in $G'$, so that $\phi(u_i)$ and $\phi(u_{i+1})$ have distance at most $4 \beta$ in $G'$. Since $\phi$ is an $(\alpha,\beta)$-quasi-isometry, it follows that $u_i$ and $u_{i+1}$ have distance at most $5\alpha \beta$ in $G$. Let $\gamma$ be formed by concatenating geodesics between $u_i$ and $u_{i+1}$ for each $0\leq i < \ell$. The path $\gamma$ clearly has length at most $5\alpha \beta \lceil \op{len}(\gamma')/\lceil \beta \rceil \rceil \leq 10\alpha (\op{len}(\gamma')+\beta)$. Moreover, given $r\geq 1$, each point in $B_r(\gamma)$ has distance at most $5\alpha\beta+r$ from one of the points $u_i$, whose image under $\phi$ has distance at most $\beta$ from one of the points of~$\gamma'$. Using the definition of an $(\alpha,\beta)$-quasi-isometry again, this implies that every point in $\phi(B_r(\gamma))$ has distance at most $\alpha(5\alpha\beta+r)+2\beta$ from a point of $\gamma'$, which is equivalent to the desired set inclusion.
\end{proof}






We now have everything we need to prove \cref{prop:structure_theory_conclusion}.

\begin{proof}[Proof of \cref{prop:structure_theory_conclusion}]
Fix $\kappa<\infty$ and let $n_0=n_0(\kappa)$ and $C_1=C_1(\kappa)$ be the constants from \cref{thm:tessera-tointon} applied with $K=3^\kappa$.
 Suppose that $G$ is an infinite, connected, transitive graph of vertex degree $d$ that is not one-dimensional and that $n \geq n_0$ is such that $\op{Gr}(3n) \leq 3^\kappa \op{Gr}(n)$. Let $H$ be the normal subgroup of $\Aut(G)$ that is guaranteed to exist by \cref{thm:tessera-tointon}, let $S$ be the generating set for $\Aut(G)/H$ from Item 4 of \cref{thm:tessera-tointon}, and write $\Gamma:=\Aut(G)/H$. Letting $\op{Gr}'$ denote the growth function of the Cayley graph $\op{Cay}(\Gamma,S)$, we have by Items 7 and 8 of \cref{thm:tessera-tointon} that
\[
\frac{\op{Gr}'(3n)}{\op{Gr}'(n)} \leq \frac{C_1 \op{Gr}(3n+C_1n)}{\op{Gr}(n)} \leq C_1^2 (3+C_1)^{C_1}.
\]
Moreover, as discussed in \cref{remark:S_bounded}, the generating set $S$ has at most $C_1(d+1)$ elements. Thus, if we take $n_1=n_1(\kappa)\geq n_0$ and $C_2=C_2(\kappa,d)$ to be the constants from \cref{thm:induced_subgroup_doesnt_keep_changing} applied with $K=C_1^2 (3+C_1)^{C_1}$ and $d=C_1(d+1)$, we have that if $n\geq n_1$ then
\begin{equation}
\label{eq:presentation_thm_restatement}
\#\Bigl\{i\in \mathbb{N} : i \geq \log_2 n \text{ and $\llangle R_{2^{i+1}} \rrangle \neq  \llangle  R_{2^i} \rrangle $} \Bigr\}\leq C_2
\end{equation}
where $R_k$ is the set of relations of $\Gamma$ that have word length at most $k$ and $\llangle R_k \rrangle$ is the smallest normal subgroup of the free group $F_S$ generated by $R_k$. 
Let $G'=(V',E')$ be the Cayley graph $\op{Cay}(\Gamma,S)$, let $\phi:V\to V'$  be the $(1,C_1n)$-quasi-isometry that is guaranteed to exist by \cref{thm:tessera-tointon} (which we may assume maps $o$ to the identity of $\Gamma$, which we denote by $\mathrm{id}$), and let $\psi:V'\to V$ be such that $d(\phi(\psi(x)),x)\leq C_1n$ for every $x\in V'$, such a function being guaranteed to exist since $\phi$ is a $(1,C_1n)$-quasi-isometry. This function $\psi$ is easily seen to be a $(1,C_3n)$-quasi-isometry for an appropriate choice of constant $C_3=C_3(\kappa)$.
For each $k\geq 1$ let $G'_k$ be the Cayley graph $\op{Cay}(\Gamma_k,S)$ of the group $\Gamma_k$ defined by $\Gamma_k=\langle S \mid R_k \rangle$.

\medskip

Fix $n\geq n_1$ and consider the set
$A=\{2^i: i \geq \log_2 n \text{ and } \llangle R_{2^{i+5}}\rrangle \neq \llangle R_{2^i} \rrangle \}$.
This set contains at most five times as many elements as the set considered in \eqref{eq:presentation_thm_restatement}, so that $|A|\leq 5C_2$. Fix $k \geq n$ and suppose that $m\geq 2kn$ does not belong to $[a,2ka]$ for any $a\in A$, so that if we define $a(m)=\sup\{a\in A:a\leq m\}$ then $a(m)\leq m/2k$. The definition of $A$ ensures that $\llangle R_{2a(m)} \rrangle = \llangle R_{16m}\rrangle$ and hence by \cref{lem:small_presentation_means_balls_look_similar} that the Cayley graphs $G'$ and $G'_{a(m)}$ have isomorphic $(8m-1)$-balls. For each $r < 4m$ we write $(S_{r}^\infty)'$ for the exposed spheres in $G'$ and $G'_{a(m)}$, which can be identified by \cref{prop:if_you_can_get_past_2n_you_can_get_away}. We also identify the balls $B_{r}'(x)$ in $G'$ and $G'_{a(m)}$ for points $x$ of distance at most $4m$ from the identity and all $r\leq m$, where the prime on $B_{r}'(x)$ reminds us that we are working with $G'$ rather than $G$.

\medskip

Now, since $\Gamma_{a(m)}$ has its group of relations generated by its relations of length at most $a(m)\leq m/2k$, it follows from  \cref{lem:short_generating_loops_yield_tubes} that for each $m/2k \leq m_1 \leq m_2 \leq 3m$ and each $u,v\in (S_{m_2}^\infty)'$ there exists a path from $u$ to $v$ in $G_{a(m)}'$ that is contained in $\bigcup_{x \in (S^\infty_{m_2})'} B_{2m_1}'(x)$ and has length at most $3m_1 \op{Gr}(3m_2)/\op{Gr}(m_1)$, where all sets and growth functions are identical in the two groups $\Gamma_{a(m)}$ and $\Gamma$ by the restrictions placed on $k$ and $m$. Since these paths are entirely contained within the ball for which $G'$ and $G'_{a(m)}$ are identical, they exist in $G'$ also. Moreover, since $m_2 \geq m_1 \geq m/2k\geq n$, we have by Item 9 of \cref{thm:tessera-tointon} as above that 
\[
\frac{\op{Gr}'(3m_2)}{\op{Gr}'(m_1)} \leq C_1^2(3+C_1)^{C_1} \left(\frac{m_2}{m_1}\right)^{C_1}
\]
so that the length of this path is at most $3 C_1^2(3+C_1)^{C_1} (m_2/m_1)^{C_1} m_1=: C_4 (m_2/m_1)^{C_1} m_1$. (Everything discussed in this paragraph is still under the assumption that $m \geq 2kn$ does not belong to $[a,2ka]$ for any $a\in A$.)

\medskip

We now use the existence of these paths in $G'$ to guarantee the desired plentiful tube conditions in the original graph $G$. More concretely, we will prove that there exist positive constants $c=c(\kappa,d)$, $C=C(\kappa,d)$, and $n_2=n_2(\kappa) \geq n_1$ such that if $n\geq n_2$ then $G$ has $(ck,ck,Ck^{C}m)$-plentiful tubes on each scale $m\geq Ckn$ such that $m$ does not belong to $[a,2ka]$ for any $a\in A$.

\medskip

We begin by constructing \emph{annular} tubes. It suffices to construct tubes in the case that the two $(m,3m)$ crossings are both the vertex sets of paths $\eta_1,\eta_2$ from $S_m$ to $S_{3m}$ since any crossing contains the vertex set of such a path.
Fix $m\geq 2kn$ and two paths $\eta_1$ and $\eta_2$ from $S_m$ to $S_{3m}$ in $G$, and suppose that $m$ does not belong to $[a,2ka]$ for any $a\in A$. Apply \cref{lem:quasi_isometry_tubes} with each of these paths and the $(1,C_3n)$-quasi-isometry $\psi :V'\to V$ (taking $u$ and $v$ to be the endpoints of $\eta_i$, $x$ to be $\phi(u)$ and $y$ to be $\phi(v)$) to obtain two paths $\eta'_1$ and $\eta_2'$ in $G'$. Using the $(1,C_1n)$-quasi-isometry property of $\phi$, each of these paths starts at distance at most $m+C_1 n$ from $\phi(o)=\mathrm{id}$ and ends at distance at least $3m - C_1 n$ from $\mathrm{id}$. If $m \geq 9 C_1 n$ 
 then both paths start at distance at most $\frac{10}{9}m$ from $\mathrm{id}$ and end at distance at least $\frac{26}{9}m$ from $\mathrm{id}$. As such, it follows from \cref{prop:if_you_can_get_past_2n_you_can_get_away} that if $m\geq 9C_1 n$ then the paths $\eta_1'$ and $\eta_2'$ both intersect the exposed sphere $(S_{m_2}^\infty)'$ in $G$ for each integer $i\in I:=\mathbb{Z}\cap [\frac{10}{9}m, \frac{12}{9}m]$. (The only property of these numbers we will need is that $12>10$ and $12<26/2$.) Fix $m_1=\lceil m/2k\rceil \leq m$ and for each  integer $i\in I$ let $x_i$ be a point of $\eta_1'$ belonging to $(S_i^\infty)'$ and let $y_i$ be a point of $\eta_2'$ belonging to $(S_i^\infty)'$.  If $m \geq \max\{9C_1 n, 2kn\}$ then, since $m$ was assumed not to belong to $[a,2ka]$ for any $a\in A$, it follows from \cref{lem:short_generating_loops_yield_tubes}
as discussed above that for each such $i$ there exists a path $\gamma_i'$ from $x_i$ to $y_i$ that is contained in $\bigcup_{(S^\infty_{i})'} B_{2m_1}'(x)$ and has length at most $C_4(i/m_1)^{C_1} m_1$.
Since $i\leq 3m$, the length of this path can be bounded by $C_5 k^{C_1-1} m$ for an appropriate constant $C_5=C_5(\kappa)$. Moreover, if $r\geq 1$ and $i$ and $j$ are two integers $i,j \in I$ satisfying $|i-j|> 4 m_1 +2r$ then the tubes of radius $r$ around $\gamma_i'$ and $\gamma_j'$ in $G'$ are disjoint since they are contained in disjoint annuli $\bigcup_{(S^\infty_{i})'} B_{2m_1+r}'(x)$ and $\bigcup_{(S^\infty_{j})'} B_{2m_1+r}'(x)$.

\medskip

Since \cref{lem:quasi_isometry_tubes} guaranteed that the paths $\eta_1'$ and $\eta_2'$ have images under $\psi$ contained in the $4C_1n$-neighbourhoods of $\eta_1$ and $\eta_2$ respectively, we can for each $i\in I$ find points $u_i$ and $v_i$ in $\eta_1$ and $\eta_2$ respectively such that $d(x_i,\phi(u_i))$ and $d(y_i,\phi(v_i))$ are at most $6C_1n$.
Applying \cref{lem:quasi_isometry_tubes} to each of the paths $\gamma_i'$ (with the quasi-isometry $\phi$, the points $u_i$ and $v_i$ in $G$, the points $x_i$ and $y_i$  in $G'$, and the quasi-isometry constants $\alpha=1$ and $\beta = 6C_1 n$), it follows that there exist constants $C_6=C_6(\kappa)$ and $C_7=C_7(\kappa)$ such that for each $i\in I$ there exists a path $\gamma_i$ from $u_i$ to $v_i$ of length at most $C_6 k^{C_1-1}$ such that if $i,j\in I$ satisfy $|i-j|\geq C_7 m_1$ then the tubes of radius $m_1$ around $\gamma_i$ and $\gamma_j$ are disjoint. The claim about annular tubes follows easily by taking a $C_7m_1$-separated subset of $I$ (i.e., a subset of $I$ in which all distinct pairs of integers have distance at least $C_7m_1$), since such a set may be taken to have size at least $I/C_7m_1 \geq c_1 k$ for some positive constant $c_1=c_1(\kappa)$ whenever $n$ is larger than some constant $n_2=n_2(\kappa)\geq n_1$. (The freedom to increase $n_1$ to $n_2$ lets us make sure that every real number we round down is at least $1$, so that rounding cannot reduce any relevant quantities by more than a factor of $1/2$.)

\medskip

We now briefly argue that the same construction also yields \emph{radial tubes} crossing a shifted annulus. Run the above construction again with $\eta_1$ and $\eta_2$ taken to be the two portions crossing from $S_m$ to $S_{12m}$ of a doubly-infinite geodesic passing through $o$, so that every point in $\eta_1$ has distance at least $2m$ from every point in $\eta_2$, but with various constants changed appropriately since we are now working with $(m,12m)$ crossings instead of $(m,3m)$ crossings. In particular, the interval $I$ can now be taken to be $\mathbb{Z}\cap[8m,10m]$, with various other constants changing to reflect this change. Consider the point $u$ in $\eta_1$ that has distance $9m$ from $o$. When we perform the above construction to build paths between $\eta_1$ and $\eta_2$, each path $\gamma_i$ starts at distance at most $m$ from $u$ and ends at distance at least $11 m$ from $u$. Thus, since $G$ is transitive, the family of paths we have constructed verifies the $(ck,ck,Ck^{C}m)$-plentiful radial tubes condition holds at the scale $m$ as desired, for some constants $c=c(\kappa,d)$ and $C=C(\kappa,d)$. As before, this works under the assumption that $n\geq n_2$ for some $n_2=n_2(\kappa)$ and that $m\geq Ckn$ does not belong to $[a,2ka]$ for any $a\in A$.
\end{proof}

\subsection{Using random walk trajectories} \label{subsection:rws}

In this section we prove \cref{prop:rw_conclusion}, which verifies the plentiful tubes condition for graphs that have quasi-polynomial absolute growth but a fast rate of relative growth over an appropriate range of scales. As discussed above, we will construct the required collections of disjoint tubes by modifying certain conditioned random walk trajectories. To avoid parity issues, we work with \emph{lazy} random walks throughout the section.  We will spend most of the section proving general bounds on the behaviour of random walk on some scale in terms of the growth of the graph at that scale, specializing to the setting of \cref{prop:rw_conclusion} only at the very end of the proof.
%
Given a graph $G$ and a vertex $u$ of $G$, let $\mathbf{P}_u$ denote the law of the \emph{lazy} random walk started from $u$, which at each step either stays in place with probability $1/2$ or else crosses a uniform random edge emanating from its current position, and let the \textbf{heat kernel} $p_t(u,v)$ be defined by $p_t(u,v)=\mathbf{P}_u(X_t=v)$. 




\medskip

We begin by recalling two important facts about random walks on graphs of  quasi-polynomial growth that will be used in the proof: The Varopoulos-Carne inequality \cite{varopoulos1985long,carne1985transmutation}, which implies near-diffusive estimates on the rate of escape, and the total variation inequality of \cite[Chapter 7.5]{Yadin}, which implies that two walks started from different vertices can be coupled to coalesce by the time they reach a distance that is near-linear in their starting distance.

\medskip

\noindent
\textbf{Diffusive estimates from Varopoulos-Carne.} We now state the Varopoulos-Carne inequality, which gives Gaussian-like bounds on the $n$-step transition probabilities between two specific vertices. See e.g.\ \cite[Chapter 13.2]{MR3616205} for a modern treatment. This inequality does not require transitivity, and holds for the random walk on any graph. 

\begin{thm}[Varopoulos-Carne]
\label{thm:Varopoulos-Carne}
Let $G=(V,E)$ be a (locally finite) graph. Then 
\[
p_t(u,v) \leq 2\sqrt{\frac{\deg(v)}{\deg(u)}} \exp\left[-\frac{d(u,v)^2}{2t}\right].
\]
for every $t\geq 1$ and every $u,v\in V$.
\end{thm}

The Varopoulos-Carne inequality easily implies that the random walk on a graph of quasi-polynomial growth is very unlikely to be at a distance much larger than $t^{1/2}(\log t)^{O(1)}$ from its starting point.

\begin{cor}
\label{cor:diffusivity_from_VC}
Let $G=(V,E)$ be a (locally finite) transitive graph and let $o$ be a vertex of $G$. Then
\[
\mathbf{P}_o\left(\max_{0\leq k\leq t} d(o,X_k) \geq n\right) \leq 2(t+1) \op{Gr}(n) \exp\left[-\frac{n^2}{2t}\right]
\]
for every $t,n \geq 1$.
\end{cor}

\begin{proof}[Proof of \cref{cor:diffusivity_from_VC}]
If $\max_{0\leq k\leq t} d(o,X_k) \geq n$ then there exists $0\leq k \leq t$ such that $X_k$ has distance exactly $n$ from $o$. Since the number of points at distance $n$ from $o$ it at most $\op{Gr}(n)$, the claim follows from \cref{thm:Varopoulos-Carne} by taking a union bound over the possible values of $k$ and $X_k$.
\end{proof}

\begin{rk}
For transitive graphs of polynomial growth, Varopoulos-Carne implies a displacement upper bound of the form $\sqrt{t\log t}$ while the true displacement is of order $\sqrt{t}$ with high probability. (This sharp upper bound on the displacement can be proven using a (highly nontrivial) improvement of the Varopoulos-Carne inequality due to Hebisch and Saloff-Coste \cite{hebisch1993gaussian}.) As such, our reliance on Varopoulos-Carne leads to all of the estimates in this section having  polylog terms that are known to be unnecessary for transitive graphs of polynomial growth and are presumably non-optimal for transitive graphs of quasi-polynomial growth also. 
\end{rk}

\medskip

\noindent
\textbf{Coupling from low growth.} We now explain how low growth can be used to couple two walks to coalesce within time not much larger than quadratic in their starting distance; we will eventually  concatenate these pairs of coupled random walks to build annular tubes. We first recall some relevant definitions. Given two probability measures $\mu$ and $\nu$ on a countable set $\Omega$, the \textbf{total variation distance} $\|\mu-\nu\|_\mathrm{TV}$ between $\mu$ and $\nu$ is defined by
\[
\|\mu-\nu\|_\mathrm{TV} = \sup_{A \subseteq \Omega}|\mu(A)-\nu(A)| = \frac{1}{2} \sum_{\omega\in \Omega} |\mu(\omega)-\nu(\omega)|.
\]
The total variation distance is indeed a distance in the sense that it defines a metric on the space of probability measures on $\Omega$.
The total variation distance is related to \emph{coupling} (and to the theory of optimal transport) by the variational formula  
\[
\|\mu-\nu\|_\mathrm{TV} = \inf\Bigl\{\p(X\neq Y) : X,Y \text{ random variables with $X\sim \mu$ and $Y\sim \nu$}\Bigr\}.
\]
In particular, given two vertices $x$ and $y$ in a graph, we can couple the lazy random walks started at $x$ and $y$ to coincide at time $m$ with probability $1-\|\mathbf{P}_x(X_m=\cdot)-\mathbf{P}_y(X_m=\cdot)\|_{\mathrm{TV}}$. Note that if the two walks coincide at time $m$ then we can trivially couple them to remain equal at all subsequent times, so that $\|\mathbf{P}_x(X_m=\cdot)-\mathbf{P}_y(X_m=\cdot)\|_{\mathrm{TV}}$ is a decreasing function of $m$ when $x$ and $y$ are fixed. These couplings will be used when we construct annular tubes using pairs of coupled random walks.

\medskip

Given a (locally finite) transitive graph $G=(V,E)$, the \textbf{Shannon entropy} $H_t$ of the $t$th step of the lazy random walk is defined to be
\[
H_t = -\mathbf{E}_o \Bigl[\log p_t(o,X_t)\Bigr] = -\sum_{x\in V} p_t(o,x)\log p_t(o,x).
\]
Since the Shannon entropy of any random variable taking values in a set of size $n$ is at most $\log n$, the quantity $H_t$ satisfies the trivial inequality
$H_t  \leq \log \op{Gr}(t)$. The following extremely useful inequality\footnote{known in some circles as ``the cool inequality''.} relates the total variation distance to the increments of the Shannon entropy; versions of this inequality have been rediscovered independently in the works \cite{erschler2010homomorphisms,MR3395463,ozawa2018functional} as discussed in detail in \cite[Chapter 7.5]{Yadin}.

\begin{thm} 
\label{thm:cool_inequality}
If $G$ is a (locally finite) transitive graph and $o$ is a vertex of $G$ then
\[\frac{1}{\deg(o)}\sum_{x \sim o} \left\| \mathbf{P}_o(X_t = \cdot) - \mathbf{P}_{x}(X_{t-1} = \cdot) \right\|_{\op{TV}}^2  \leq H_t - H_{t-1}\]
for every $t\geq 1$.
\end{thm}

To apply this inequality, we will need to bound the entropy in terms of the growth. While we always have the trivial bound $H_t \leq \log \op{Gr}(t)$, it is also possible to bound the entropy in terms of $[\log \op{Gr}(t^{1/2})]^2$, which is a significantly better bound when the growth is much larger on scale $t$ than scale $t^{1/2}$. (While this bound is worse than the trivial bound when the growth is subexponential and sufficiently regular, it better fits into our philosophy of understanding the behaviour of the random walk at some scale from the growth of the graph at that scale alone.)

\begin{lem}
\label{lem:entropy_from_growth}
For each $d\geq 1$ there exists a constant $C=C(d)$ such that if $G$ is a (locally finite) transitive graph of degree $d$ then
\[H_t \leq C \Bigl(\log \op{Gr}\bigl(t^{1/2}\bigr)\Bigr)^2\]
for every $t\geq 1$.
\end{lem}

\begin{proof}[Proof of \cref{lem:entropy_from_growth}]
We may assume that the diameter of $G$ is at least $t^{1/2}$, the claim being trivial otherwise since $H_t \leq \log |V|$.
Recall that if $X$ and $Y$ are two random variables defined on the same probability space, the \textbf{conditional entropy} $H(X\mid Y)$ is defined to be the expected entropy of the conditional law of $X$ given $Y$. Bayes' rule for the conditional entropy states that
\[
H(X)=H(Y)+H(X|Y)-H(Y|X)\leq H(Y)+H(X|Y).
\]
Applying this inequality with $X=X_t$ and $Y=\mathbbm{1}(X_t \in B_r)$, and using that the entropy of a random variable supported on a set of size $N$ is at most $\log N$, we obtain that
\[
H_t \leq \log 2 + \log \op{Gr}(r)  + [\log \op{Gr}(t)] \mathbf{P}_o(X_t \notin B_r)
\leq \log 2+  \log \op{Gr}(r) + 2 (t+1) \op{Gr}(r) \exp\left[-\frac{r^2}{2t}\right] \log \op{Gr}(t) 
\]
for every $r,t\geq 1$. Using the fact that the growth is submultiplicative and that $\op{Gr}(n) \leq d^{n+1}$, we obtain that if $n:=\lfloor t^{1/2} \rfloor$ divides $r$ then
\[
H_t \leq \log 2+  \frac{r}{n} \log \op{Gr}(n) + 2 d (t+1)^2  \exp\left[\frac{r}{n} \log \op{Gr}(n)-\frac{r^2}{2t}\right],
\]
and the claim follows by taking $r$ to be a multiple of $n$ closest to $C  \log [t\op{Gr}(n)]$ for an appropriately large constant $C=C(d)$. (Note that $\log [t\op{Gr}(n)]$ and $\log \op{Gr}(n)$ are of the same order since the diameter of $G$ is at least $t^{1/2}$ and hence $\op{Gr}(n)\geq n$.)
\end{proof}

This inequality easily implies the following simple bound on the total variation distance in terms of the growth, yielding in particular that the total variation distance is small whenever $t$ is much larger than $d(x,y)^2 \log \op{Gr}(t^{1/2})^2$.

\begin{cor}
\label{cor:coupling_from_low_growth}
For each $d\geq 1$ there exists a constant $C=C(d)$ such that if $G=(V,E)$ is a (locally finite) transitive graph of degree $d$ then
\[\|\mathbf{P}_x(X_t=\cdot)-\mathbf{P}_y(X_t=\cdot)\|_{\mathrm{TV}} \leq \frac{C\log \op{Gr}(t^{1/2})}{t^{1/2}}\, d(x,y)\]
for every $t \geq 1$ and  $x,y\in V$.
\end{cor}

\begin{proof}[Proof of \cref{cor:coupling_from_low_growth}]
It follows by a standard computation that the total variation distance between a Binomial$(n,1/2)$ distribution and a Binomial$(n+1,1/2)$ distribution is of order $n^{-1/2}$. Indeed, if we let $\mu$ be the Binomial$(n,1/2)$ distribution and let $\nu$ be the Binomial$(n+1,1/2)$ distribution then $\mu$ is absolutely continuous with respect to $\nu$ with density
\[
\frac{\mu(k)}{\nu(k)} = \frac{2 (n-k+1)}{n+1} \qquad \text{ for every $0\leq k \leq n+1$},
\]
and we have by an easy computation (using e.g.\ Jensen's inequality and the linearity of the variance) that
\[
\|\mu-\nu\|_\mathrm{TV} = \frac{1}{2}\sum_{k=0}^{n+1} \left|\frac{\mu(k)}{\nu(k)}-1\right|\nu(k) = \frac{1}{2}\sum_{k=0}^{n+1} \left|\frac{n-2k+1}{n+1}\right|\nu(k) =O(n^{-1/2})
\]
as claimed. Since the conditional laws of the lazy random walks $X_t$ and $X_{t+1}$ are the same given that the number of non-lazy steps are the same, it follows that 
\[
\left\| \mathbf{P}_x(X_t = \cdot) - \mathbf{P}_{x}(X_{t-1} = \cdot) \right\|_{\op{TV}} \leq C_1 t^{-1/2}
\]
for every $t\geq 1$ and $x\in V$, where $C_1$ is a universal constant. Putting this together with \cref{thm:cool_inequality} yields that if $x$ and $y$ are neighbouring vertices on a transitive graph of degree $d$ then
\[
\left\| \mathbf{P}_x(X_t = \cdot) - \mathbf{P}_{y}(X_{t} = \cdot) \right\|_{\op{TV}} \leq \sqrt{4d (H_t-H_{t-1})}+C_1 t^{-1/2}
\]
for every $t\geq 1$. Since the left hand side is increasing in $t$ and, by \cref{lem:entropy_from_growth}, there exists $t/2\leq k\leq t$ with $H_k-H_{k-1} \leq \frac{2}{t} H_t \leq \frac{C_2}{t}\log \op{Gr}(t^{1/2})^2$ for some constant $C_2=C_2(d)$, it follows 
that
\[
\left\| \mathbf{P}_x(X_t = \cdot) - \mathbf{P}_{y}(X_{t} = \cdot) \right\|_{\op{TV}} \leq C_1 t^{-1/2}+ \sqrt{\frac{4C_2 d }{t}\log \op{Gr}(t^{1/2})^2}
\]
for every pair of adjacent vertices $x$ and $y$ and every $t\geq 1$. The analogous bound for arbitrary pairs of vertices follows from this and the triangle inequality for the total variation distance.
\end{proof}

\medskip

\noindent
\textbf{Hitting probabilities of balls.} We now want to argue that tubes around independent random walks started at distant vertices are likely to be disjoint under the assumption that our graph ``looks at least five dimensional'' on all relevant scales. In fact we will prove more general versions of these estimates in which ``five'' is replaced by an arbitrary constant $\kappa>4$. 
We begin by noting the following simple analytic consequence 
of the results of Coulhon and Saloff-Coste \cite{MR1232845} and Morris and Peres \cite{morris2005evolving}, 
 which lets us convert growth bounds into bounds on the heat kernel $p_t(u,v)$. 

\begin{lem} \label{lem:growth_gives_kernel_decay}
      For each integer $d \geq1$ there exists a positive constant $c=c(d) \in(0,1]$ such that if $G = (V,E)$ is an infinite unimodular transitive graph with vertex degree $d$ then
      \[
            p_t(u,v) \leq \frac{1}{\op{Gr} \left( c t^{1/2} \left[ \log \op{Gr}\left( t^{1/2} \right) \right]^{-1/2} \right)}
      \]
      for every integer $t\geq 4$ and every pair of vertices $u,v$ in $G$.
\end{lem}


\begin{proof}[Proof of \cref{lem:growth_gives_kernel_decay}]
It suffices to prove an inequality of the form
      \begin{equation}
      \label{eq:heat_kernel_even_only}
            p_{2t}(u,v) \leq \frac{1}{c\op{Gr} \left( c t^{1/2} \left[ \log \op{Gr}\left( t^{1/2} \right) \right]^{-1/2} \right)}
      \end{equation}
      for every integer $t\geq 4$ and every $u,v\in V$, where $c=c(d)$ is a positive constant depending only on the degree.
      Indeed, odd values of $t$ can then be handled using the inequality
      \[
      p_{2t+1}(u,v) \leq \frac{1}{d} \sum_{v' \sim v} p_{2t}(u,v') \leq \max_{v' \sim v} p_{2t}(u,v'),
      \]
      while the constant outside of the growth function can be absorbed into the constant inside the growth function using the inequality $\op{Gr}(3nm) \geq n \op{Gr}(m)$, which holds for all positive integers $n,m$ in any infinite transitive graph as an elementary consequence of the triangle inequality.

\medskip
    
    We now prove an estimate of the form \eqref{eq:heat_kernel_even_only}.
      Let the inverse growth function $\op{Gr}^{-1}$ be defined by $\op{Gr}^{-1}(x) := \inf \{ n : \op{Gr}(n) \geq x  \}$ and recall that the \textbf{isoperimetric profile} of $G$ is the function $\Phi:[1,\infty) \to [0,d]$ defined by
      \begin{equation}
      \label{eq:Coulhon_Saloff-Coste}
            \Phi(x) := \inf \left\{ \frac{ \abs{ \partial W } }{\abs{W}} : W \subseteq V(G) \text{ and } 0 < \abs{W} \leq x  \right\}.
      \end{equation}
       For transitive unimodular graphs, the isoperimetric profile and the growth are related by the inequality
      \begin{equation} \label{eq:growth_to_isop}
      \Phi(x) \geq \frac{1}{2 \op{Gr}^{-1}(2x)},
      \end{equation}
      which was proven for Cayley graphs by Coulhon and Saloff-Coste \cite{MR1232845} and extended to unimodular transitive graphs by Saloff-Coste \cite{MR1377559} and Lyons, Morris, and Schramm \cite{MR2448128}; we use the statement given in \cite[Theorem 10.46]{MR3616205}.
To make use of this inequality, we will apply the results of Morris and Peres
\cite{morris2005evolving}, which imply that there exists a constant $c_1(d) \in (0,1)$ such that
      \[
            p_{2t}(u,v) \leq \frac{1}{c_1 \sup \left\{ y : \int_1^y \frac{1}{x \Phi(4x)^2} \mathrm{d}x  \leq c_1 t  \right\}}
      \]
      for every integer $t\geq 1$ and every $u,v\in V$. Using \eqref{eq:growth_to_isop} to estimate the integral that appears here, we have for each $y \in [1,\infty)$ that
      \[ 
            \int_1^y \frac{1}{x \Phi(4x)^2} \mathrm{d}x \leq \int_1^y \frac{4 \op{Gr}^{-1}(8x)^2}{x} \mathrm{d}x \leq 4 \op{Gr}^{-1}(8y)^2 \int_1^y \frac{1}{x} \mathrm{d}x = 4 \op{Gr}^{-1}(8y)^{2} \log(y).
     \]
     Thus, to prove an estimate of the form \eqref{eq:heat_kernel_even_only} it suffices to verify that 
     \begin{equation}
     \label{eq:growth_verification}
     \text{if $y\geq 1$ satisfies } y \leq \frac{1}{8} \op{Gr}\left( \left[ \frac{c_1 t}{4 \log \op{Gr}(t^{\frac{1}{2}})}\right]^{\frac{1}{2}} \right) \qquad \text{ then } \qquad 4 \op{Gr}^{-1}(8y)^{2} \log(y) \leq c_1 t.
     \end{equation}
     This follows straightforwardly by noting that $8y \leq \op{Gr}( \sqrt{c_1 t / (4 \log y)})$ whenever $y$ satisfies the upper bound on the left hand side of \eqref{eq:growth_verification}. \qedhere

\end{proof}

\cref{lem:growth_gives_kernel_decay} has the following elementary corollary, which we will apply only in situations where $\log \op{Gr}\bigl(t^{1/2}\bigr)$ is much smaller than $n^{-1} t^{1/2}$.

\begin{cor}[Leaving a ball]
\label{cor:leaving_ball}
For each integer $d\geq 1$ and real number $\kappa \geq 1$ there exists a constant $C=C(d,\kappa)$ such that if $G= (V,E)$ is an infinite, connected, unimodular transitive graph with vertex degree $d$ and $n,t\geq 1$ are integers such that $\op{Gr}(3m)\geq 3^\kappa \op{Gr}(m)$ for every $n \leq m \leq \frac{1}{2}t^{1/2}$ then
\[
\mathbf{P}_u(X_t \in B_n(v)) \leq C\left[\log \max\left\{ \frac{t}{n^2},\op{Gr}\bigl(n\bigr)\right\} \right]^{\kappa} \left(\frac{n^2}{t}\right)^{\kappa/2}
\]
for every pair of vertices $u$ and $v$ in $G$.
\end{cor}

(Note that this corollary holds vacuously when $t\leq n^2$.) 

\begin{rk}
This estimate is quite similar to that appearing in e.g.\ \cite{MR4133700}; the important distinction is that we only assume the tripling condition $\op{Gr}(3m)\geq 3^\kappa \op{Gr}(m)$ for $m=O(t^{1/2})$ rather than for all sufficiently large scales.
\end{rk}

\begin{proof}[Proof of \cref{cor:leaving_ball}]
Since $\max_u \mathbf{P}_u(X_t\in B_n(v))$ is a decreasing function of $t$, it suffices to prove the claim under the slightly stronger assumption that $\op{Gr}(3m)\geq 3^\kappa \op{Gr}(m)$ for every $n \leq m \leq t^{1/2}$.
Fix $n,t\geq 1$ and let $c$ be the constant from \cref{lem:growth_gives_kernel_decay}. We have by submultiplicativity of the growth function that
\[
\op{Gr}(t^{1/2}) \leq \op{Gr} \left( c t^{1/2} \left[ \log \op{Gr}\left( t^{1/2} \right) \right]^{-1/2} \right)^{c^{-1}\left[\log \op{Gr}\left( t^{1/2} \right) \right]^{1/2}}
\]
and hence that
\begin{equation}
\op{Gr} \left( c t^{1/2} \left[ \log \op{Gr}\left( t^{1/2} \right) \right]^{-1/2} \right) \geq \op{Gr} \left( t^{1/2}\right)^{c\left[\log \op{Gr}\left( t^{1/2} \right) \right]^{-1/2}} = \exp\left[ c \sqrt{\log \op{Gr}\left( t^{1/2} \right)}\right].
\label{eq:submult_growth2}
\end{equation}
Applying \cref{lem:growth_gives_kernel_decay}, it follows that if $r:=ct^{1/2} \left[ \log \op{Gr}\left( t^{1/2} \right) \right]^{-1/2} \leq n$ then 
\[
p_t(u,v) \leq \frac{1}{\op{Gr(r)}} \leq \exp\left[-c^2 \frac{t^{1/2}}{n}\right]
\]
for every $u,v\in V$. It follows by a union bound that if $r\leq n$ then
\[
\mathbf{P}_u(X_t\in B_n(v)) \leq \exp\left[-c^2 \frac{t^{1/2}}{n}\right]\op{Gr}(n),
\]
which is stronger than the desired inequality.
 Now suppose that $r\geq n$. 
The assumption $\op{Gr}(3m)\geq 3^\kappa \op{Gr}(m)$ for every $n \leq m \leq t^{1/2}$ guarantees that 
\[
\op{Gr}(r) \geq \op{Gr}(n) \prod_{i=1}^{\lfloor \log_3 (r/n) \rfloor} \frac{\op{Gr}(3^i n)}{\op{Gr}(3^{i-1}n)} \geq \left(\frac{r}{3n}\right)^\kappa \op{Gr}(n).
\]
We deduce from \cref{lem:growth_gives_kernel_decay} that there exists a constant $C$ such that
\begin{align*}
p_t(u,v) 
&\leq \frac{1}{\op{Gr} \left( r \right)} 
\leq \min\left\{ \left(\frac{3n \sqrt{ \log \op{Gr}\left( t^{1/2} \right)}}{c t^{1/2} }\right)^\kappa\frac{1}{\op{Gr}(n)}, \exp\left[ -c \sqrt{\log \op{Gr}\left( t^{1/2} \right)}\right]\right\}\\
&\leq C\left(\frac{n^2}{t}\log \max\left\{ \frac{t}{n^2},\op{Gr}\bigl(n\bigr)\right\} \right)^{\kappa/2} \frac{1}{\op{Gr(n)}}
\end{align*}
where the second inequality follows since $\min\{A x^{\kappa},e^{-x}\} \leq  A (1\vee \log (1/A))^{\kappa}$ for every $A,x>0$ (as can be checked by case analysis according to whether $x \geq \log (1/A)$). 
\end{proof}

We next analyze the probability of hitting a ball whose radius is much smaller than its distance from the starting point. For transitive graphs of polynomial growth, a similar estimate without the logarithmic term can be proven by a similar calculation as in \cite[Lemma 4.4]{MR4055195}.

\begin{lem}[Hitting a distant ball]
\label{lem:hitting_distant_ball}
For each integer $d\geq 1$ and real number $\kappa > 2$ there exists a constant $C=C(d,\kappa)$ such that if $G$ is an infinite, connected, unimodular transitive graph with vertex degree $d$ and $n,t\geq 1$ are integers such that $t \geq n^2$ and $\op{Gr}(3m)\geq 3^\kappa \op{Gr}(m)$ for every $n \leq m \leq t^{1/2}$ then
\[
\mathbf{P}_u(\text{\emph{hit $B_n(v)$ before time $t$}}) \leq C \left[\log \max\left\{ d(u,v),\op{Gr}\bigl(2n\bigr)\right\}\right]^{(3\kappa+2)/2}\left(\frac{n}{d(u,v)}\right)^{\kappa-2}
\]
for every pair of vertices $u$ and $v$ with $d(u,v)\geq 2n$.
\end{lem}

\begin{proof}[Proof of \cref{lem:hitting_distant_ball}]
For each $1\leq s \leq t$, let $A_s$ be the event that the random walk hits $B_n(v)$ between times $s$ and $t$. (We will optimize over the choice of $s$ at the end of the proof.) It follows from \cref{cor:leaving_ball} that there exist constants $C_1$ and $C_2$ depending only on $d$ and $\kappa$ such that
\begin{multline}
\mathbf{E}_u\Bigl[\#\{s \leq k \leq 2t: X_k \in B_{2n}(v)\}\Bigr] \leq C_1\sum_{k=s}^{2t} \left[\log \max\left\{ \frac{k}{n^2},\op{Gr}\bigl(2n\bigr)\right\} \right]^{\kappa} \left(\frac{n^2}{k}\right)^{\kappa/2} \\
\leq C_2 \left[\log \max\left\{ \frac{s}{n^2},\op{Gr}\bigl(2n\bigr)\right\} \right]^{\kappa} \frac{n^\kappa}{s^{(\kappa-2)/2}},
\label{eq:expected_time_in_ball}
\end{multline}
where the second inequality follows by calculus. On the other hand, it follows from \cref{cor:diffusivity_from_VC} and a straightforward calculation that
\[
\mathbf{P}_w\Bigl(X_{\tau+k}\in B_{2n} \text{ for every } k \leq \frac{n^2}{8\log \op{Gr}(n)}  \Bigr) \geq \frac{1}{2}
\]
for every $w\in B_n(v)$, and since $t+ \frac{n^2}{8\log \op{Gr}(n)} \leq 2t$ it follows by the strong Markov property that
\begin{equation}
\label{eq:conditional_expected_time_in_ball}
\mathbf{E}_u\Bigl[\#\{s \leq k \leq 2t: X_k \in B_{2n}(v)\}\mid A_s\Bigr] \geq \frac{n^2}{16\log \op{Gr}(n)}.
\end{equation}
Putting together the estimates \eqref{eq:expected_time_in_ball} and \eqref{eq:conditional_expected_time_in_ball} yields that
\[
\mathbf{P}_u(A_s) \leq \frac{\mathbf{E}_u\Bigl[\#\{s \leq k \leq 2t: X_k \in B_{2n}(v)\}\Bigr]}{\mathbf{E}_u\Bigl[\#\{s \leq k \leq 2t: X_k \in B_{2n}(v)\}\mid A_s\Bigr]} \leq C_3 \left[\log \max\left\{ \frac{s}{n^2},\op{Gr}\bigl(2n\bigr)\right\} \right]^{\kappa+2} \frac{n^{\kappa-2}}{s^{(\kappa-2)/2}},
\]
while, since every point in $B_n(v)$ has distance at least $d(u,v)/2$ from $u$, it follows from the Varopoulos-Carne inequality and a union bound as in the proof of \cref{cor:diffusivity_from_VC} that
\[
\mathbf{P}_u(\text{hit $B_n(v)$ before time $s$}) \leq 2 s \op{Gr}(n) \exp\left[-\frac{d(u,v)^2}{8s}\right].
\]
Putting together these estimates yields that
\[
\mathbf{P}_u(\text{hit $B_n(v)$ before time $t$})\leq C_3 \left[\log \max\left\{ \frac{s}{n^2},\op{Gr}\bigl(2n\bigr)\right\} \right]^{\kappa+2} \frac{n^{\kappa-2}}{s^{(\kappa-2)/2}}+2 s \op{Gr}(n) \exp\left[-\frac{d(u,v)^2}{8s}\right],
\]
and the claimed inequality follows by taking $s=c' d(u,v)^2 (\log \op{Gr}(n))^{-1}$ for an appropriately small constant $c'$.
\end{proof}

\textbf{Disjoint tubes from coarse-grained random walks.}
As noted above, a naive construction of disjoint tubes using random walks is not appropriate for our plentiful tubes condition, since the two walks will couple at a time roughly quadratic in their starting distance rather than roughly linear. To circumvent this issue, we will instead consider tubes around certain coarse-grained versions of the random walk defined through what we call \textbf{ironing}, where we replace portions of the random walk with geodesics between their endpoints. This process will also be useful when we analyze intersections between random walk tubes, as the ironing process allows us to circumvent overcounting issues that would arise in a naive first-moment argument.

\medskip

We now define the ironing procedure formally; see \cref{fig:ironing} for an illustration.
Let $G$ be a graph. For every pair of distinct vertices $u,v \in V(G)$, fix a geodesic $\zeta(u,v)$ in $G$ from $u$ to $v$. (The choice of $\zeta$ is irrelevant to our arguments; we need only that it is done deterministically for every pair of vertices before we start running any random walks.) Fix $r > 0$ and let $\gamma$ be a finite path in $G$. We define a sequence $(\tau_i)_{i \geq 0}$ recursively as follows: Let $\tau_0=0$. For each $i\geq 0$, if $d(\gamma_{\tau_i},\gamma_{k})<r$ for every $\tau_i \leq k \leq \op{len}(\gamma)$ we set $\tau_{i+1}=\op{len}(\gamma)$ and stop. Otherwise, we set $\tau_{i+1}$ to be the minimal time $k$ after $\tau_i$ that $d(\gamma_{\tau_i},\gamma_{k}) \geq r$.
 We define the \textbf{crease number} $\op{cr}(\gamma)=\op{cr}_r(\gamma)$ to be the number of non-zero terms in this sequence, so that $\tau_{\op{cr}(\gamma)}=\op{len}(\gamma)$, call the points $\{\gamma_{\tau_i}:0\leq i \leq \op{cr}_r(\gamma)\}$ \textbf{crease points}, and define the \textbf{ironed path} $\op{iron}(\gamma)=\op{iron}_{r}(\gamma)$ by concatenating geodesics between crease points
\[
      \op{iron}_r(\gamma) := \zeta(\gamma_{\tau_0},\gamma_{\tau_1}) \circ \zeta(\gamma_{\tau_1},\gamma_{\tau_2}) \circ \cdots \circ \zeta(\tau_{\op{cr}(\gamma)-1},\tau_{\op{cr}(\gamma)}).
\]
Thus, the ironed path $\op{iron}(\gamma)$ is a finite path in $G$ which has the same start and end points as $\gamma$, has length at most $r\cdot\op{cr}(\gamma)$, 
and satisfies the containment of tubes
\[
B_r(\op{iron}_r(\gamma)) \subseteq B_{2r}(\gamma) \qquad \text{ and } \qquad B_{r}(\gamma) \subseteq B_{2r}(\{\gamma_{\tau_i} : 0\leq i \leq \op{cr}_r(\gamma)\}) \subseteq B_{2r}(\op{iron}_r(\gamma)).
\]

\begin{figure}
\center
\includegraphics[width=0.6\textwidth]{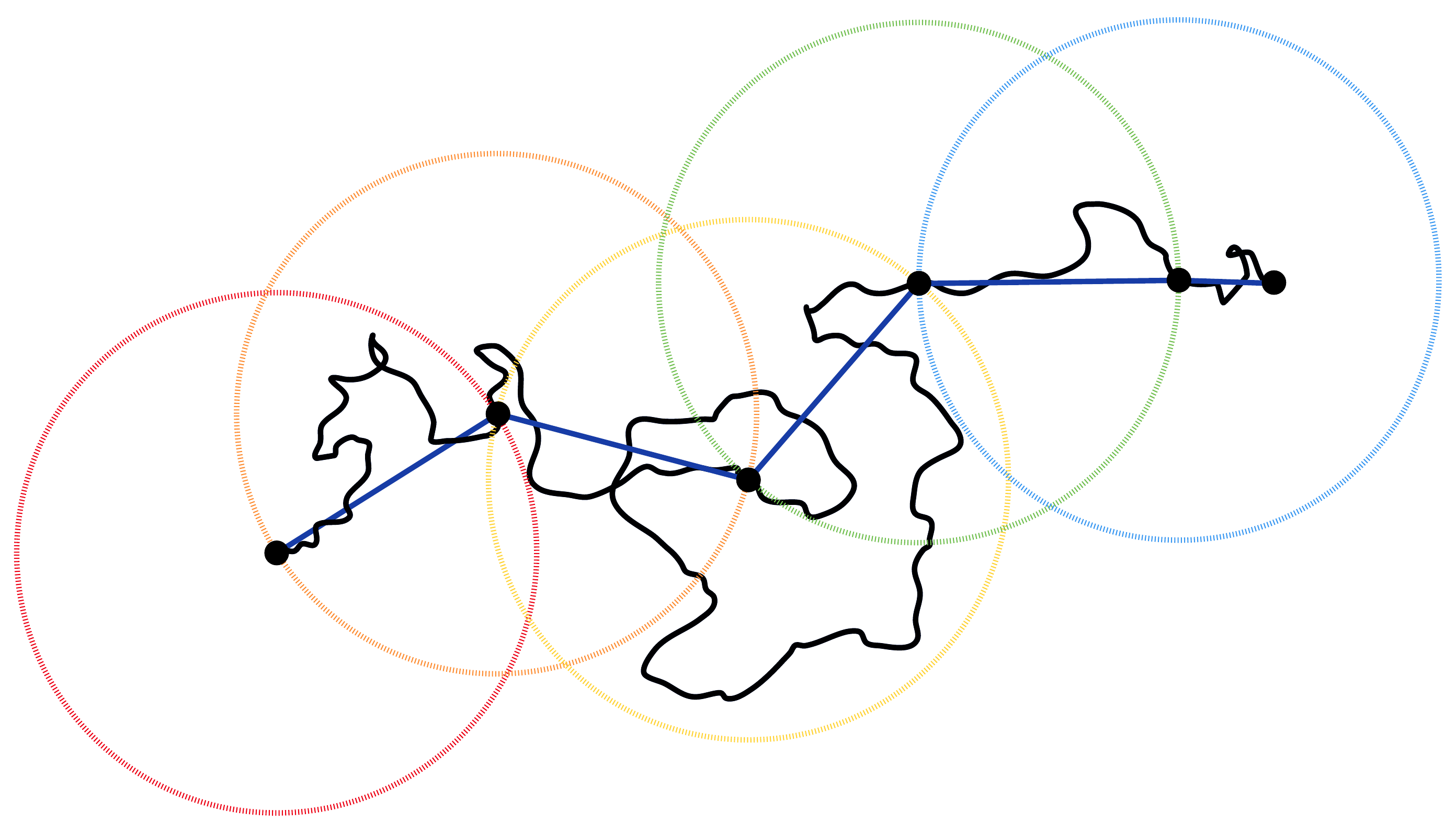} 
\caption{Schematic illustration of the ironing procedure applied to a path: The original path is in black, from left to right. The black dots are crease points. A new crease point is formed every time the path leaves a fixed-radius ball centered at the previous crease point. The straight blue line segments are geodesics between consecutive crease points (together with the final point of the walk). The ironed path is formed by concatenating these geodesics.}
\label{fig:ironing}
\end{figure}


For graphs of quasi-polynomial growth, we can use the Varopoulos-Carne inequality to show that the length of an ironed random walk $\op{iron}_r(X^t)$ is of order at most $r^{-1}t (\log t)^{O(1)}$ with high probability when $r=O(\sqrt{t})$. We write $X^t$ for the path formed by the first $t$ steps of the random walk.

\begin{lem}
Let $G$ be a (locally finite) transitive graph and let $u$ be a vertex of $G$. Then
\label{lem:creases}
\[
\mathbf{P}_u\Bigl(\op{cr}_r(X^t) > \frac{t}{m} \Bigr) \leq 2tm \op{Gr}(r) \exp\left[-\frac{r^2}{2m}\right]
\]
for every $r,t,\lambda \geq 1$.
\end{lem}

\begin{proof}
In order for the inequality $\op{cr}_r((X_i)_{i=0}^t) > t/m$ to hold, there must exist $0\leq i \leq t$ such that $d(X_i,X_{i+m}) \geq r$. As such, the claim follows from \cref{cor:diffusivity_from_VC} and a union bound.
\end{proof}

We now analyze intersections between independent random walk tubes.
Given two vertices $u$ and $v$, we write $\mathbf{P}_u\otimes \mathbf{P}_v$ for the law of  a pair of independent lazy random walks $X$ and $Y$ started at $u$ and $v$ respectively.

\begin{lem}[Intersections of random walk tubes]
\label{lem:disjoint_RW_tubes}
For each integer $d\geq 1$ and real number $\kappa > 4$ there exists a constant $C=C(d,\kappa)$ such that if $G=(V,E)$ is an infinite, connected, unimodular transitive graph with vertex degree $d$ and $n,t\geq 1$ are integers such that $t \geq r^2$ and $\op{Gr}(3m)\geq 3^\kappa \op{Gr}(m)$ for every $r \leq m \leq t^{1/2}$ then
\begin{multline*}
\mathbf{P}_u \otimes \mathbf{P}_v\Bigl(\text{there exist $0\leq i,j \leq t$ such that $d(X_i,Y_j)\leq r$}\Bigr) 
\\\leq 
C \frac{t}{d(u,v)^2}\left[\log \max\left\{ d(u,v),\op{Gr}\bigl(4r\bigr)\right\}\right]^{(3\kappa+4)/2}\left(\frac{r}{d(u,v)}\right)^{\kappa-4},
\end{multline*}
for every $u,v\in V$ with $d(u,v)\geq 4r$.
\end{lem}

\begin{rk}
Although it will suffice for all our applications, we note that this bound is very wasteful when $t$ is much larger than $d(u,v)^2$. A more careful analysis would use that $Y_{\tau_\kappa}$ is typically very far from $u$ when $k$ is large.
\end{rk}

\begin{proof}[Proof of \cref{lem:disjoint_RW_tubes}]
Let $A$ be the event that there exists $w\in V$ with $d(u,w) \geq \frac{1}{2}d(u,v)$ and $1\leq i,j\leq t$ such that $d(X_i,w),d(Y_j,w)\leq r$.
It suffices by symmetry to prove that there exists a constant $C=C(d,\kappa)$ such that
\begin{equation*}
\mathbf{P}_u \otimes \mathbf{P}_v(A) \leq C \frac{t}{r^2}\left[\log \max\left\{ d(u,v),\op{Gr}\bigl(4r\bigr)\right\}\right]^{(3\kappa+4)/2}\left(\frac{r}{d(u,v)}\right)^{\kappa-2}.
\end{equation*}
Let $\tau_0,\ldots,\tau_{\op{cr}_r(Y^t)}$ be the stopping times used to define the ironed walk $\op{iron}_r(Y^t)$, and observe that if $A$ holds then there must exist $0\leq k \leq \op{cr}_r(Y^t)$ and $0\leq i \leq t$ such that
\[
d(Y_{\tau_k},v) \geq \frac{1}{2}d(u,v)-r\geq \frac{1}{4}d(u,v) \qquad \text{ and } \qquad d(X_i,Y_{\tau_k}) \leq 2r.
\]
Thus, it follows from  \cref{lem:hitting_distant_ball} and  a union bound that there exists a constant $C_1=C_1(d,\kappa)$ such that
\begin{align*}
&\mathbf{P}_u \otimes \mathbf{P}_v(A) \\&\hspace{0.75cm}\leq \mathbf{P}_v\Bigl(\op{cr}_r(Y^t) > \frac{t}{m}\Bigr) +  \frac{t}{m}\max \left\{\mathbf{P}_u\Bigl(\text{hit $B_{2r}(w)$ before time $t$}\Bigr) : w\in V, d(u,w)\geq \frac{1}{4}d(u,v)\right\} \\
&\hspace{1.5cm}\leq 2tm \op{Gr}(r) \exp\left[-\frac{r^2}{2m}\right] +C_1 \frac{t}{m}\left[\log \max\left\{ d(u,v),\op{Gr}\bigl(4r\bigr)\right\}\right]^{(3\kappa+2)/2}\left(\frac{r}{d(u,v)}\right)^{\kappa-2}
\end{align*}
and the claim follows by taking $m=c'r^2 (\log \max\{d(u,v),\op{Gr}(r)\})^{-1}$ for an appropriately small constant $c'=c'(d,\kappa)$.
\end{proof}

We now have everything we need to prove \cref{prop:rw_conclusion}.

\begin{proof}[Proof of \cref{prop:rw_conclusion}]
We will prove the claim concerning annular tubes (which is harder); the changes to the proof needed to establish the claim concerning radial tubes are straightforward and will be explained briefly at the end of the proof.

\medskip

We will prove a general condition for $G$ to have $(k,r,\ell)$-plentiful annular tubes on scale $n$ in terms of the growth of $G$, which we specialize to give the claim about scales of quasi-polynomial growth at the end of the proof.
Fix $n\geq 1$, $k\leq n/2$, $r\leq n$, $t\geq n$ and two $(n,3n)$ crossings $A$ and $B$ as in the definition of plentiful annular tubes, so that $A$ and $B$ each contain points at every distance from $o$ between $n$ and $3n$. 
We will carry out our analysis under the hypotheses that
\begin{equation}
\label{eq:t&r_growth_assumption}
\op{Gr}(3m) \geq 3^\kappa \op{Gr}(m) \qquad \text{ for every $r\leq m \leq t$}
\end{equation}
and
\begin{equation}
\label{eq:t_coupling_assumption}
 \qquad \|\mathbf{P}_{x}(X_t = \cdot)-\mathbf{P}_{y}(X_t=\cdot)\|_\mathrm{TV} \leq \frac{1}{4} \qquad \text{ for every $x,y \in B_{3n}$.}
\end{equation}
We will return to the question of when suitable $t$ and $r$ satisfying these hypotheses can be chosen at the end of the proof.

\medskip

Since $A$ and $B$ each contain at least one point at each distance from $o$ between $n$ and $3n$, and since $n/2k \geq1$, we may choose for each $1\leq i \leq k$ points $a_i\in A$ and $b_i \in B$ such that
\[
n+(4i-4) \frac{n}{2k} \leq d(o,a_i) \leq n + (4i-3)\frac{n}{2k} \quad \text{ and } \quad n+(4i-2)\frac{n}{2k}\leq d(o,b_i) \leq n + (4i-1)\frac{n}{2k}
\]
for every $1\leq i \leq k$,
so that the set of points $\{a_i\}\cup\{b_i\}$ is $(n/k)$-separated (i.e., any two distinct points in the set have distance at least $n/k$).
%
%
%
%
For each $i$, let $\mathbf{Q}_i$ be the joint law of a random walk $(X_{i,m})_{m\geq 0}$ started at $a_i$ and a random walk $(Y_{i,m})_{m\geq 0}$ started at $b_i$, coupled so that $X_t = Y_t$ with probability at least $3/4$ (such a coupling exists by the hypothesis \eqref{eq:t_coupling_assumption} imposed on the value of $t$), and let $\mathbf{Q}=\bigotimes \mathbf{Q}_i$ be the law of the collection $\{X_{i,m},Y_{i,m}:1\leq i \leq k, m\geq 0\}$ in which the pairs $((X_{i,m})_{m\geq 0},(Y_{i,m})_{m\geq 0})$ are sampled independently for each $1\leq i \leq k$.
For each $1\leq i \leq k$, consider the events
\[
\mathcal{A}_i = \{X_{i,t} = Y_{i,t}\}
\qquad\text{and}\qquad
\mathcal{B}_i = \left\{\op{len}\bigl(\op{iron}_r(X^t)\bigr), \op{len}\bigl(\op{iron}_r(Y^t)\bigr) \leq \frac{16t}{r}\log t\op{Gr}(r) \right\}.
\]
The event $\mathcal{A}_i$ has probability at least $3/4$ for every $1\leq i \leq k$ by construction. Meanwhile, \cref{lem:creases} implies that
\begin{multline*}
\mathbf{Q}(\mathcal{B}_i^c) \leq 2\mathbf{P}_u\left(\op{cr}_r(X^t) > \frac{16t}{r^2}\log \max\{t,\op{Gr}(r)\}\right) \\
\leq 4t r^2 [\log \max\{t,\op{Gr}(r)\} ]\op{Gr}(r)\exp\left[-8 \log \max\{t,\op{Gr}(r)\}\right],
\end{multline*}
which is less than $1/4$ if $t$ is larger than some universal constant $t_0$. Now, for each $1\leq i,j \leq k$ let
\[
\mathcal{I}_{i,j} = \{B_{2r}(X^t_i) \cup B_{2r}(Y^t_i) \text{ has non-empty intersection with } B_{2r}(X^t_j) \cup B_{2r}(Y^t_j) \}.
\]
We have by a union bound that if $i$ and $j$ are distinct then
\begin{align*}
\mathbf{Q}(\mathcal{I}_{i,j}) 
&\leq 4 \max\left\{\mathbf{P}_u\otimes \mathbf{P}_v(B_{2r}(X^t) \cap B_{2r}(Y^t) \neq \emptyset) : d(u,v) \geq \frac{n}{k}\right\}\\
\\
&\leq 
C \frac{tk^2}{n^2}\left[\log \max\left\{ n,\op{Gr}\bigl(8r\bigr)\right\}\right]^{(3\kappa+4)/2}\left(\frac{rk}{n}\right)^{\kappa-4} =: \alpha = \alpha(n,t,k,r)
\end{align*}
Thus, if for each $1\leq i \leq k$ we define the event
\[
\mathcal{C}_i = \{\text{there exist at most $4\alpha k$ values of $1\leq j \leq k$ with $j\neq i$ such that $\mathcal{I}_{i,j}$ holds}\}
\]
then $\mathbf{Q}(\mathcal{C}_i) \geq \frac{3}{4}$ by Markov's inequality. It follows by a union bound that $\mathbf{Q}(\mathcal{A}_i \cap \mathcal{B}_i \cap \mathcal{C}_i) \geq 1/4$, and hence that if we define 
\[
\mathcal{E}=\#\bigl\{1\leq i \leq k : \mathcal{A}_i \cap \mathcal{B}_i \cap \mathcal{C}_i \text{ holds}\bigr\}\geq \frac{k}{4} \qquad \text{ then } \qquad \mathbf{Q}(\mathcal{E}) >0.
\]
Consider the random set of indices
\[
I = \{1\leq i \leq k: \mathcal{A}_i \cap \mathcal{B}_i \cap \mathcal{C}_i \text{ holds but } \mathcal{I}_{i,j} \text{ does not hold for any $j<i$ for which $\mathcal{A}_j \cap \mathcal{B}_j \cap \mathcal{C}_j$ holds}\}.
\]
On the event $\mathcal{E}$, the set $I$ has size at least $\lceil (k/4)/(1+4\alpha k) \rceil \geq \frac{1}{20}\min\{k,\alpha^{-1}\}$. Moreover, on this event, the set of paths formed by concatenating $\op{iron}_r(X^t_i)$ and the reversal of $\op{iron}_r(Y^t_i)$ for each $i\in I$ have the property that each such path has length at most $32 t/r \log t\op{Gr}(r)$, and the tubes of thickness $r$ around distinct such paths are disjoint. Since this event has positive probability, there must exist a set of paths with this property.
Since the sets $A$ and $B$ were arbitrary, it follows that there exists a positive constant $c=c(d,\kappa)$ such that $G$ has
\[
\left(c\min\left\{k,\,\frac{n^2}{tk^2} \left[\log \max\{n,\op{Gr}(8r)\}\right]^{-(3\kappa+4)/2}\left(\frac{n}{rk}\right)^{\kappa-4}\right\},\;r,\; \frac{32t}{r}\log \max\{t,\op{Gr}(r)\} \right)\text{-plentiful}
\]
annular tubes on scale $n$ for each $1\leq r,k \leq n/2$ and $t\geq n^2$ satisfying \eqref{eq:t&r_growth_assumption} and \eqref{eq:t_coupling_assumption}.

\medskip

We now specialize to the setting of the proposition. Let $D,\lambda \geq 1$ and $0<\eps<1$ and suppose that $n \geq 1$ satisfies
$\op{Gr}(m)\leq e^{(\log m)^D}$ and $\op{Gr}(3m) \geq 3^5 \op{Gr}(m)$ for every $n^{1-\eps}\leq m \leq n^{1+\eps}$.
Let $k=(\log n)^\lambda$, $r=n(\log n)^{-20D\lambda}$, and $t=n^2 (\log n)^{2D}$. There exists $n_0=n_0(d,D,\lambda,\eps)$ such that if $n\geq n_0$ then $n^{1-\eps}\leq r,t^{1/2}\leq n^{1+\eps}$, so that \eqref{eq:t&r_growth_assumption} holds. Moreover, since $\log \op{Gr}(t^{1/2}) \leq (\log t^{1/2})^D \leq (1+\eps)^2 (\log n)^{D}$, it follows from \cref{cor:coupling_from_low_growth} that \eqref{eq:t_coupling_assumption} holds whenever $n\geq n_1$ for some constant $n_1=n_1(d,D,\lambda,\eps)\geq n_0$. It follows that there exists a constant $C$ such that if $n\geq n_1$ then $G$ has 
\[
\left(c\min\left\{(\log n)^\lambda,\, (\log n)^{(20\lambda-\frac{21}{2})D-\lambda}\right\},\;n(\log n)^{-20D\lambda},\; C(\log n)^{(3+20\lambda)D} \right)\text{-plentiful}
\]
annular tubes on scale $n$. Since the plentiful annular tubes condition is monotone increasing in the number and thickness parameters and monotone decreasing in the length parameter, it follows that there exists a constant $n_2=n_2(d,D,\lambda,\eps)\geq n_1$ such that if $n\geq n_2$ then $G$ has $(1/40D,\lambda)$-polylog-plentiful annular tubes on scale $n$ as claimed (where we used the assumption that $n$ is large to absorb the constant prefactors into the exponents of the logarithms).

\medskip

For the claim concerning \emph{radial} tubes, one uses random walks started at $k$ equidistant points along a geodesic from $o$ to $S_n$. \cref{cor:leaving_ball} implies that each of these random walks has a good probability not to belong to $B_{3n}$ at times of order $n^2 (\log \op{Gr}(n))^{\kappa}$, so that we can obtain tubes from $S_n$ to $S_{3n}$ using segments of the resulting ironed paths. The analysis of the number, thickness, and lengths we can take these walks to have with the resulting tubes being disjoint is similar to the annular case above. (Indeed, the analysis is somewhat simpler since we use single walks instead of coupled pairs of walks. This also makes the dependence on the growth better than in the annular case.)
\end{proof}

\section{Quasi-polynomial growth II: Analysis of percolation}

\label{sec:low_growth_multiscale}


In this section, we analyze percolation in the low-growth regime, i.e., on scales $n$ where $\op{Gr}(n) \leq e^{[\log n]^D}$ for a constant $D$. Our goal is to show that a lower bound on (full-space) connection probabilities at some low-growth scale $n$ implies a lower bound on connection probabilities within a tube \emph{at a much larger scale} after sprinkling. This analysis will employ both the polylog-plentiful tubes condition from \cref{prop:existence_of_interval_with_controlled_tubes} and the outputs of the ghost field technology developed in \cref{prop:snowballing}. 

\medskip

Recall that $\mathcal{U}_d^*$ is the space of infinite, connected, unimodular transitive graphs of degree $d$ that are not one-dimensional. 

\begin{prop}[Inductive analysis of percolation in the low-growth regime.]
\label{prop:low_growth_main}
For each $d,D\geq 1$ there exist positive constants $\lambda_0=\lambda_0(d,D)$ and $c=c(d,D)$ such that for each $\lambda \geq \lambda_0$ there exist constants $K_1=K_1(d,D,\lambda)$ and $n_0=n_0(d,D,\lambda)$ such that if $K\geq K_1$ and $n\geq n_0$ then the following holds:
If $G\in \mathcal{U}_d^*$ and for each $1\leq b \leq n$ we define
\[
\delta(b,n)=\delta_K(b,n)= \left(\frac{K \log \log n}{\min\{\log n, \log \op{Gr}(b)\}}\right)^{1/4}
\]
 then the implication
\begin{multline}
\label{eq:goal_of_the_low_growth_argument}
\Bigl(n \in \mathscr{L}(G,D),\; \kappa_{p_1}(n,\infty)\geq e^{-(\log \log n)^{1/2}}, \, \p_{p_1} ( \op{Piv}[4b , n^{1/3} ] ) \leq (\log n)^{-1}, \, \text{ \emph{and} } \delta(b,n) \leq 1 \Bigr)
\\
\Rightarrow \Bigl(p_2 \geq  p_c \, \text{ \emph{or} } \, \kappa_{p_2}\left( e^{(\log n)^{c \lambda}},n \right) \geq e^{-3(\log \log n)^{1/2}}\Bigr)
\end{multline}
holds for every $n\geq n_0$, $b\leq \frac{1}{8}n^{1/3}$, and $p_2 \geq p_1\geq 1/d$ with $p_2 \geq \sprinkle(p_1,\delta(b,n))$.
\end{prop}

The parameter $\lambda$ controls how far a connection probability lower bound is propagated by sprinkling, namely from scale $n$ to scale $e^{(\log n)^{c\lambda}}$, where $c$ is independent of $\lambda$. \Cref{prop:low_growth_main} morally says that for any choice of $\lambda$, if we sprinkle by a sufficient amount at a sufficiently large scale, we can achieve this propagation from scale $n$ to scale $e^{(\log n)^{c\lambda}}$. (Note that since the conclusion of \Cref{prop:low_growth_main} gets stronger as $\lambda$ increases, the $\lambda \geq \lambda_0$ condition is redundant; we have nevertheless included it because some of our working is simpler if we can assume that $\lambda$ is large.)

\medskip

 The overall strategy of this section is closely inspired by \cite{contreras2022supercritical}. 
A key insight of that paper was that if one is working in the supercritical phase of percolation, then one can sometimes deduce positive information (e.g.\ a lower bound on set-to-set connection probabilities) from negative information (e.g.\ an upper bound on set-to-set connection probabilities). Thus, one can analyze the supercritical phase by a case analysis according to whether or not certain point-to-point connection lower bounds hold: both assumptions are useful.
 This only makes sense because being in the supercritical phase is already a positive hypothesis. For our purposes, we cannot assume that we are in the supercritical phase. However, we can assume that we have a two-point lower bound at some large scale $n$, which lets us pretend that we are supercritical when working with scales much smaller than $n$. The positive information that this lets us deduce by working with scales smaller than $n$ is so strong that it implies set-to-set connection lower bounds even at scales much larger than $n$. 

\medskip

 Besides this, there are two main complications we need to address when adapting the methods of \cite{contreras2022supercritical} in this section. First, 
 our quasi-polynomial growth and plentiful tubes conditions are rather different than the polynomial growth and connectivity of exposed spheres used in \cite{contreras2022supercritical}, and many details of the argument must change to accommodate this.
 Second, and more seriously, we must use methods that use the growth upper bound \emph{at one scale only}, since that is all we can assume; this is very different from \cite{contreras2022supercritical} where their graphs are assumed to have polynomial growth at all scales and many of the arguments work by inducting up from scale $1$. 



\medskip


We now begin to work towards the proof of \cref{prop:low_growth_main}. We begin by recording various notations and important constants that will be used throughout the proof. First, we let $\mathcal{A}(b,n,p_1)=\mathcal{A}_{D,K}(b,n,p_1;G)$ be the statement on the left hand side of the implication \eqref{eq:goal_of_the_low_growth_argument}:
\[
\mathcal{A}(b,n,p_1) = \Bigl(n \in \mathscr{L}(G,D),\; \kappa_{p_1}(n,\infty)\geq e^{-(\log \log n)^{1/2}}, \, \p_{p_1} ( \op{Piv}[4b , n^{1/3} ] ) \leq (\log n)^{-1} \, \text{ \emph{and} } \delta(b,n) \leq 1 \Bigr).
\]
We will always assume that $\delta(p_1,p_2)$ is larger than the quantity $\delta:= \delta(b,n)=\delta_K(b,n)$ introduced in \cref{prop:low_growth_main}. We regard the choices of $d \in \mathbb{N}$, $D \geq 1$, $K\geq 1$, and $\lambda \geq 20$, as well as the graph $G\in \mathcal{U}_d^*$, as being fixed for the remainder of the section.
We write $p_{3/2}:=\sprinkle(p_1;\delta(p_1,p_2)/2)$, so that $p_2=\sprinkle(p_{3/2};\delta(p_1,p_2)/2)$ by the semigroup property of our sprinkling operation, and let $c_1=c_1(d,D)>0$ and $N=N(d,D,\lambda)$ be the constants guaranteed to exist by \cref{prop:existence_of_interval_with_controlled_tubes}.

\medskip

Inspired by \cite{contreras2022supercritical}, we will split the proof of \cref{prop:low_growth_main} into two cases according to how easy it is to connect points at certain well-chosen intermediate scales. 
  For each $1 \leq m \leq n$ we define the \textbf{two-point zone}
\[
      \op{tz}(m)=\op{tz}(m,n) := \sup \left\{ r \geq 0 : \tau_{p_{3/2}}^{B_m}(B_r) \geq (\log n)^{-1} \right\},
\]
where we recall that $\tau^B_p(A)$ is the quantity defined by
$\tau^B_p(A) := \inf_{x,y \in A}\p_p(x \xleftrightarrow{B} y )$.
We stress that the two-point zone $\op{tz}(m)$ is defined using the intermediate parameter $p_{3/2}$.
Knowing the value of $\op{tz}(m)$ for some $m$ provides us with both positive information (on scales smaller than $\op{tz}(m)$) and negative information (on scales larger than $\op{tz}(m)$). 
%
%
%

\medskip

For each $n\in \mathscr{L}(G,D)$ with $n\geq N$, let $n^{1/3} \leq m_1 \leq m_2 \leq n^{1/(1+c_1)}$ be such that $m_2 \geq m_1^{1+c_1}$ and 
%
 $G$ has $(c_1,\lambda)$-polylog-plentiful tubes at every scale $m_1 \leq m \leq m_2$ (such $m_1$ and $m_2$ existing by \cref{prop:existence_of_interval_with_controlled_tubes}), 
  and let $\mathscr{S}(n)=\mathscr{S}_{d,D,\lambda}(n) =\{m \in \mathbb{N} : m_1^{1+(c_1/4)} \leq m \leq m_1^{1 +(3c_1/4)} \}$. From now on we will mostly work at scales $m \in \mathscr{S}(n)$, so that $G$ has plentiful tubes not just at these scales but at a large range of consecutive scales on either side of every such scale. Let $c_2=c_2(d)$ be the minimum of the four constants appearing in \cref{prop:snowballing} with `$D$' equal to 1 (known in the statement of that proposition as $c_1$, $c_2$, $c_3$, and $h_0$; the proposition becomes weaker if we replace all four constants by their minimum), define $c_3=c_3(d) := 2^{-9}c_2$, and
   consider the statement
  \begin{multline*}
\mathcal{B}(n,p_{3/2})=\mathcal{B}_{d,D,\lambda,K}(n,p_{3/2}) \\= \Bigl(\text{there exists $m\in \mathscr{S}(n)$ such that }
\op{tz}\left(\frac{m}{2}\right) [ \log n]^{c_3K} \geq m (\log n)^{3\lambda / c_1 }
\Bigr).
  \end{multline*}
  We think of $\mathcal{B}(n,p_{3/2})$ as our ``positive assumption'' about percolation on scale $n$: it means that there is a ``good'' scale $m\in \mathscr{S}(n)$ such that points in a ball of radius not much smaller than $m$ can be connected with reasonable probability within the ball of radius $m/2$ when $p=p_{3/2}$.

  \medskip

  \medskip
  \hrule


  \medskip

  \noindent \emph{Notational conventions and standing assumptions:}
  \textbf{Recall that the choices of $d \in \mathbb{N}$, $D \geq 1$, $K\geq 1$, $\lambda \geq 20$, and $G\in \mathcal{U}_d^*$ are considered to be fixed for the remainder of the section.
  The constants $N$, $c_1$, $c_2$, and $c_3$ used to define $\mathscr{S}(n)$ and $\mathcal{B}(n,p)$ will be used with the same meaning throughout this section. We also define $c_{-1}=c_{-1}(d)$ to be a positive constant such that $\sprinkle(p;\delta)\geq p + c_{-1}\delta$ for every $p\geq 1/d$ and $0<\delta \leq 1$, which exists by \eqref{eq:natural_sprinkling_lower_bound}. Finally, we fix $K_0=K_0(d,D)=2^{11}(c_3^{-1}\vee c_{-1}^{-5}\vee 1)$ throughout the section: all subsequent lemmas in this section will include $K\geq K_0$ as an implicit hypothesis.  These conventions do not apply outside of this section (\cref{sec:low_growth_multiscale}).}
  \medskip

  \hrule

\medskip

\medskip

  \cref{prop:low_growth_main} follows trivially from the following two lemmas.

\begin{lem}[Concluding from a positive assumption]
\label{lem:done_if_can_cross_tubes}
There exists a constant $\lambda_0=\lambda_0(d,D)$ such that if $\lambda \geq \lambda_0$ then there exists a constant $n_0=n_0(d,D,\lambda) \geq N$ such that the implication
\begin{equation}
\label{eq:goal_of_the_low_growth_argument_B}
\mathcal{A}_{D,K}(b,n,p_1) \wedge \mathcal{B}_{d,D,\lambda,K}(n,p_{3/2})  \\ 
\Rightarrow \Bigl(p_2 \geq  p_c \, \text{ \emph{or} } \, \kappa_{p_2}\left( e^{(\log n)^{c_1 \lambda/4}},n \right) \geq e^{-3(\log \log n)^{1/2}}\Bigr)
\end{equation}
holds for every $n\geq n_0$, $b\leq \frac{1}{8}n^{1/3}$, and every pair of probabilities $p_2 \geq p_1 \geq 1/d$ with $\delta(p_1,p_2)\geq \delta_K(b,n)$.
\end{lem}

\begin{lem}[Concluding from a negative assumption]
\label{lem:complement_of_done_if_can_cross_tubes}
There exists a constant $\lambda_0=\lambda_0(d,D)$ such that if $\lambda \geq \lambda_0$ then there exist constants $K_1=K_1(d,D,\lambda)\geq K_0$ and $n_0=n_0(d,D,\lambda) \geq N$ such that if $K\geq K_1$ and $n\geq n_0$ then the implication
\begin{equation}
\label{eq:goal_of_the_low_growth_argument_B}
\mathcal{A}_{D,K}(b,n,p_1) \wedge (\neg\, \mathcal{B}_{d,D,\lambda,K}(n,p_{3/2})) \\ 
\Rightarrow \Bigl(p_2 \geq  p_c \, \text{ \emph{or} } \, \kappa_{p_2}\left( e^{(\log n)^{c_1 \lambda/4}},n \right) \geq e^{-3(\log \log n)^{1/2}}\Bigr)
\end{equation}
holds for every $n\geq n_0$, $b\leq \frac{1}{8}n^{1/3}$, and every pair of probabilities $p_2 \geq p_1 \geq 1/d$ with $\delta(p_1,p_2)\geq \delta_K(b,n)$.
\end{lem}

\subsection{Concluding from a positive assumption}

Our next goal is to prove \cref{lem:done_if_can_cross_tubes}. We begin with the following lemma, which will later allow us to use the positive assumption $\mathcal{B}(n,p)$ to deduce lower bounds on the corridor function after sprinkling. Note that we are not yet using the polylog-plentiful tubes condition or the assumption $\mathcal{B}(n,p)$, so that the parameter $\lambda$ does not appear in this lemma. (Recall our standing assumption throughout this section that $K\geq K_0=K_0(d,D)$.)

\medskip

We define $p_{7/4}=\sprinkle(p_1;\frac{3}{4}\delta(p_1,p_2))=\sprinkle(p_{3/2};\frac{1}{4}\delta(p_1,p_2))$, and remind the reader that the two-point zone $\op{tz}(m)$ was defined with respect to the parameter $p_{3/2}$.

\begin{lem} \label{lem:can_always_cross_tz_tubes}
 There exist a constant $n_0=n_0(d,D) \geq N$ and a universal constant $c_4$ such that if $n \geq n_0$ then the implication
\begin{equation*}
\mathcal{A}_{D,K}(b,n,p_{1})
\Rightarrow \Bigl(
      \kappa_{p_{7/4}}\left( \op{tz}\left(\frac{m}{2}\right) [ \log n]^{c_3K} , m \right) \geq c_4 e^{-2(\log \log n)^{1/2}} \text{ for every $n^{1/3} \leq m\leq n$}\Bigr)
\end{equation*}
holds for every $n\geq n_0$, $b\leq \frac{1}{8}n^{1/3}$, and $p_1 \geq 1/d$.
\end{lem}

\begin{proof}[Proof of \cref{lem:can_always_cross_tz_tubes}]
Fix $n^{1/3}\leq m \leq n$ and a path $\gamma$ starting at some vertex $u$, ending at some vertex $v$, and with $\op{len} \gamma \leq  \op{tz}\left(\frac{m}{2}\right) (\log n)^{c_3K}$. Pick a subsequence $u = u_1, u_2,\ldots, u_k = v$ of $\gamma$ with $k \leq 5 (\log n)^{c_3K}$ and $\op{dist}(u_i,u_{i+1}) \leq \frac{1}{4} \op{tz}\left(\frac{m}{2}\right)$ for every $1\leq i<k$. (There are rounding issues that may prevent such a sequence existing when $\op{tz}\left(\frac{m}{2}\right)$ is small, but this is not a problem when $n$ is large.) We claim that $\frac{1}{4} \op{tz}\left(\frac{m}{2}\right) \geq b$. Indeed, the hypothesis $b\leq \frac{1}{8}n^{1/3}$ guarantees that $ 4b \leq m/2 $, and we have by a union bound that
\begin{multline} 
      \tau_{p_{3/2}}^{B_{m/2}} \left( B_{4b} \right) \geq \tau_{p_{1}}^{B_{m/2}} \left( B_{4b} \right) \geq \kappa_{p_{1}}(n,\infty) - \p_{p_{1}}( \op{Piv}[4b,n^{1/3}] ) \\
      \geq e^{-(\log \log n)^{1/2}} - [\log n]^{-1} \geq \frac{1}{2} e^{-[\log \log n]^{1/2}}\label{eq:good_connection_at_dist_z}
\end{multline}
for every $n$ larger than some universal constant. 
It follows in particular that $\tau_{p_{3/2}}^{B_{m/2}}(B_{4b}) \geq (\log n)^{-1}$ for every $n$ larger than some universal constant $n_0$, and hence that $\frac{1}{4} \op{tz}\left(\frac{m}{2}\right) \geq b$ when $n\geq n_0$ by maximality of $\op{tz}\left(\frac{m}{2}\right)$.

\medskip

Let $h=c_2 \max\{\op{Gr}(b)^{-1},n^{-1/15}\}$, so that if $n_0 \geq 3$ (to guarantee that $n^{1/15} \geq \log n$) then
\begin{equation}
\label{eq:h_in_always_cross_tubes}
h^{c_2 (\delta/4)^3} \leq h^{c_2 (\delta/4)^4} \leq \exp\left(\log h \cdot \frac{c_2}{2^{8}} \cdot \frac{ K \log \log n }{ \log h} \right) = (\log n)^{-2 c_3 K} \leq (\log n)^{-2}
\end{equation}
by the definition of $\delta_K(b,n)$ and the assumption that $K\geq K_0$. 
We want to apply \cref{prop:snowballing} where `$n$' is $k$, the sets `$A_i$' are the balls $(B_b(u_i))_{i=1}^k$, the superset `$\Lambda$' is the tube $B_{m/2}(\gamma)$, the ghost field intensity `$h$' is equal to $h$, `$p_1$' is $p_{3/2}$, and the sprinkling amount `$\delta$' is $\delta/4$. To do this, it suffices to verify that if $n$ is sufficiently large then
\[
h^{c_2 (\delta/4)^3} \leq c_2 k^{-1} \qquad \text{ and } \qquad
      \tau_{p_{3/2}}^{B_{m/2}(\gamma)}( B_{b}(u_i) \cup B_b(u_{i+1}) ) \geq 
       4 h^{c_2 (\delta/4)^4 } \quad  \text{for every $1 \leq i \leq k-1$.}
\]
 (The assumption that $h$ is sufficiently small holds automatically since we defined $c_2$ to be the minimum of the constants appearing in \cref{prop:snowballing} and set $h\leq c_2$.) The inequality \eqref{eq:h_in_always_cross_tubes} implies that the first required inequality $h^{c_2 (\delta/4)^3} \leq c_2 k^{-1}$ holds whenever $n$ is larger than some constant depending only on $d$, since $h^{c_2 (\delta/4)^3} \leq (\log n)^{-2c_3K}$, $k \leq 5(\log n)^{c_3K}$ and $c_3K \geq 2$. 
%
%
For the second required inequality, note that for every such $i$ and for every $u \in B_b(u_i)$ and $v \in B_{b}(u_{i+1})$, we have $u,v \in B_{\op{tz}(m/2)-1}(u_{i})$ and hence that
\[
      \tau_{p_{3/2}}^{B_{m/2}(\gamma)}( B_b(u_i) \cup B_b(u_{i+1}) ) \geq \tau_{p_{3/2}}^{B_{m/2}}( B_{\op{tz}(m/2)-1} ) \geq (\log n)^{-1}
\]
for all $n$ larger than some universal constant by \eqref{eq:good_connection_at_dist_z}. Let $n_0 =n_0(d,D) \geq N$ be the maximum of these two constants and $N$.

\medskip

Let $r$ be the minimum positive integer such that $\p_{q}(\op{Piv}[1,r h]) < h$ for every $q \in [p_{3/2},\sprinkle(p_{3/2};\delta/4)]$.
 The previous paragraph shows that if $n\geq n_0$ then the hypotheses of \cref{prop:snowballing} are met. Thus, since $p_{7/4} \geq \sprinkle(p_{3/2};\delta/4)]$, we obtain from that proposition that for a universal constant $c > 0$,
\begin{align*}
      \p_{p_{7/4}} \left( u \xleftrightarrow{B_{m/2+2r}(\gamma)} v \right) &\geq c \tau_{p_{3/2}}^{B_{m/2}(\gamma)}(B_{b}(u)) \cdot \tau_{p_{3/2}}^{B_{m/2}(\gamma)}(B_b(v)) \geq \frac{c}{4} e^{-2 (\log \log n)^{1/2}}.
\end{align*}
Since $\gamma$ was arbitrary, this implies that if $n\geq n_0$ then
\[
\kappa_{p_{7/4}}\left( \op{tz}\left(\frac{m}{2}\right) [ \log n]^{c_3K} , \frac{m}{2}+2r \right) \geq \frac{c}{4} e^{-2(\log \log n)^{1/2}}.
\]
All that remains is to verify that $2r \leq m/2$ when $n$ is sufficiently large.

\medskip

 Let $q \in [p_{3/2},\sprinkle(p_{3/2};\delta/4)]$ be arbitrary. Since $2/5<1/2$, 
\cref{prop:two_arm_bound_in_box} yields that there exists a constant $C_{d}$ such that
\begin{align*}
\p_q\left( \op{Piv}[1, (m/4)  h   ] \right) &\leq C_{d} \left[ \frac{ \log \op{Gr}((m/4) h )}{(m/4) h } \right]^{2/5}
\leq 
 C_{d} \left[ \frac{ 4\log \op{Gr}(m)}{m h} \right]^{2/5}.
\end{align*}
Since $h \geq n^{-1/15}$, $n^{1/3} \leq m \leq n$, $\log \op{Gr}(m)\leq (\log m)^D$, and $(2/5)\cdot (4/15) > 1/15$, it follows that there exists a constant $n_2=n_2(d,D)$ such that if $n\geq n_2$ then
\begin{align*}
\p_q\left( \op{Piv}[1, (m/4)  h   ] \right) &\leq
 C_{d} \left[ \frac{ 4(\log n)^D}{n^{1/3}\cdot n^{-1/15}} \right]^{2/5} \leq n^{-1/15} \leq h,
\end{align*}
and hence that $r \leq m/4$ as claimed. This completes the proof. \qedhere

%
%
\end{proof}

We now apply \cref{lem:can_always_cross_tz_tubes} together with the polylog-plentiful tubes condition to prove \cref{lem:done_if_can_cross_tubes}.




\begin{proof}[Proof of \cref{lem:done_if_can_cross_tubes}]
 Let $n_0=n_0(d,D)$ and the universal constant $c_4$ be as in \cref{lem:can_always_cross_tz_tubes}. Suppose that $n\geq n_0$, $b \leq \frac{1}{8}n^{1/3}$, and $p_1 \in [1/d,1]$ are such that $\mathcal{A}_{D,K}(b,n,p_1)$ holds, and write $\delta=\delta(b,n)$. We have by \cref{lem:can_always_cross_tz_tubes} that
\[
\kappa_{p_{7/4}}\left( \op{tz}\left(\frac{m}{2}\right) [ \log n]^{c_3K} , m \right) \geq c_4 e^{-2(\log \log n)^{1/2}} \text{ for every $n^{1/3} \leq m\leq n$}
\]
and in particular that if $\mathcal{B}_{d,D,\lambda,K}(n,p)$ holds then there exists $m\in \mathscr{S}(n)=\mathscr{S}_{d,D,\lambda}(n)$ such that
      \begin{equation} \label{eq:sufficient_tube_crossing_to_get_conclusion}
            \kappa_{p_{7/4}} \bigl( m(\log n)^{3\lambda/c_1} , m \bigr) \geq (\log n)^{-1}
      \end{equation}
      whenever $n_0$ is larger than a suitable universal constant.
(Note that we have not yet put any restrictions on the parameter $\lambda$.)

\medskip

      Suppose that $m \in \mathscr{S}(n)$ satisfies \eqref{eq:sufficient_tube_crossing_to_get_conclusion} and let $r= m[\log n]^{3\lambda/(2c_1)}$, so that if $n \geq 2$ we have the inclusion of intervals
      \begin{equation} \label{eq:correct_choice_of_wiggled_parameter_for_tubes}
            \Bigl[\frac{9}{10} r (\log r)^{-\lambda/c_1} , r (\log r)^{\lambda/c_1}\Bigr] \subseteq \Bigl[m,m(\log n)^{3 \lambda/c_1}\Bigr].
      \end{equation}
      Since $\frac{9}{10} r \in [m_1,m_2]$ (where $m_1$ and $m_2$ are as in the definition of $\mathscr{S}(n)$ above), $G$ has $(c_1,\lambda)$-polylog-plentiful radial tubes at scale $\frac{9}{10} r$. Let $\Gamma$ be a family of paths witnessing this fact (which cross the annulus from $S_{0.9r}$ to $S_{4\cdot (0.9r)}=S_{3.6r}$) and let $(T_\gamma)_{\gamma \in \Gamma}$ be the associated family of tubes given by $T_\gamma :=B(\gamma,\frac{9}{10} r[\log (\frac{9}{10}r)]^{-\lambda/c_1})$. Since each tube $T_\gamma$ has thickness at least $m$ and length $\op{len} \gamma \leq r(\log r)^{\lambda/c_1} \leq m(\log n)^{3\lambda/c_1}$, it follows from \eqref{eq:sufficient_tube_crossing_to_get_conclusion} that
      \[
            \p_{p_{7/4}}(S_{0.9 r} \xleftrightarrow{\,T_\gamma\,} S_{3.6 r } ) \geq (\log n)^{-1}
      \]
      for every $\gamma \in \Gamma$ and hence that there exists a constant $n_1=n_1(d,D,\lambda) \geq n_0$ such that if $c_1 \lambda  > 4$ and $n\geq n_1$ then
      \begin{align}
            \p_{p_{7/4}}( S_{0.9 r} \leftrightarrow S_{3.6 r }) &\geq 1 - \prod_{\gamma\in \Gamma} \p_{p_{7/4}} ( S_{0.9 r} \nxleftrightarrow{T_i} S_{3.6r}  ) 
            \nonumber         \geq 1 - (1-(\log n)^{-1})^{(\log (n^{1/3}))^{c_1 \lambda}} 
            \\
            &\geq 1 - \exp\bigl( 3^{-c_1\lambda} (\log n)^{c_1 \lambda-1}\bigr)
            \geq 1 - \exp(-(\log n)^{3c_1 \lambda/4}),
      \end{align}
      where we used that $1-x\leq e^{-x}$ in the second line.

      \medskip

      We next use the fact that $G$ also has $(c_1,\lambda)$-polylog-plentiful \emph{annular} tubes at scale $r$. (We will no longer use the paths and tubes from the radial case that we defined in the previous paragraph, so it is not a problem to reuse the same notation for the annular tubes we consider in the rest of the proof.) We will work with the standard monotone coupling $(\omega_p)_{p\in [0,1]}$ of percolation at different parameters. To lighten notation, we write $\omega_{1}=\omega_{p_1}$, $\omega_{3/2}=\omega_{p_{3/2}}$, $\omega_{7/4}=\omega_{p_{7/4}}$, and $\omega_2=\omega_{p_2}$. Let $u,v \in S_r$ and consider the event $A_{u,v} =\{u \leftrightarrow S_{3r}$ and $v \leftrightarrow S_{3r}$ in the configuration $\omega_{7/4}\}$. Define $C_u := K_u(\omega_{7/4})$ and $C_v := K_v(\omega_{7/4})$, so that the event $A_{u,v}$ is entirely determined by the pair $(C_u,C_v)$. Whenever $A_{u,v}$ holds, let $\Gamma$ be a family of paths from $C_u$ to $C_v$ that is guaranteed to exist by the fact that $G$ has $(c_1,\lambda)$-polylog-plentiful annular tubes at scale $r$ (choosing these paths as a function of $(C_u,C_v)$), and let $(T_\gamma)_{\gamma\in \Gamma}$ be the associated family of tubes, noting that $\bigcup_{\gamma \in \Gamma} T_\gamma \subseteq B(2r[\log r]^{\lambda/c_1})$. Define the configurations 
      \[\alpha := \omega_{2} \cap (\partial_E C_u \cup \partial_E C_v) \cap B(2 r [\log r]^{\lambda/c_1}) \qquad \text{and} \qquad \beta := (\omega_{7/4} \backslash \overline {C_u \cup C_v})\cap B(2 r [\log r]^{\lambda/c_1}),\]
      where we recall that $\overline {C_u \cup C_v}$ denotes the set of edges with at least one endpoint in $C_u \cup C_v$. In order for $C_u$ to be connected to $C_v$ in the configuration $(\alpha \cup \beta) \cap T_\gamma$, it suffices that in at least one of the tubes $T_{\gamma}$, there is a $\beta$-path from an endpoints of an $\alpha$-open edge in $\partial_E C_u$ to an endpoint of an $\alpha$-open edge in $\partial_E C_v$. The estimate \eqref{eq:sufficient_tube_crossing_to_get_conclusion} together with the interval inclusion \eqref{eq:correct_choice_of_wiggled_parameter_for_tubes} therefore yield that
      \begin{align*}
            \p( C_u \xleftrightarrow{\,\alpha \cup \beta\,} C_v \mid C_u,C_v) &\geq \mathbf 1(A_{u,v}) \cdot \left[ 1 - \prod_{\gamma \in \Gamma} \p( C_u \nxleftrightarrow{T_i \cap (\alpha \cup \beta)} C_v \mid C_u,C_v )\right]\\
            &\geq \mathbf 1(A_{u,v}) \cdot \left[ 1 - \left(1 -  \left[\frac{c_{-1}\delta}{4}\right]^2 [\log n]^{-1} \right)^{[\log (n^{1/3})]^{c_{1} \lambda}} \right] \\
            &\geq \mathbf 1(A_{u,v}) \cdot \left[ 1 - \exp\left(-3^{-c_1\lambda}(\log n)^{c_1\lambda-2}\right) \right],
    \end{align*}
      where we used that $c_{-1} \delta \geq 4(\log n)^{-1/2}$ (which holds by definition of $\delta$ and $K_0$) in the final inequality. As such, there exist constants $\lambda_0=\lambda_0(d,D)$ and $n_2=n_2(d,D,\lambda) \geq n_1$ such that if $\lambda \geq \lambda_0$, $n \geq n_2$, and $c_1 \lambda \geq 4$ then
\begin{align*}
            \p( C_u \xleftrightarrow{\,\alpha \cup \beta\,} C_v \mid C_u,C_v) 
            &\geq \mathbf 1(A_{u,v}) \cdot \left[ 1- e^{-[\log n]^{3c_1\lambda/4}} \right].
    \end{align*}
      Since $\alpha \cup \beta \subseteq \omega_{2} \cap B(2r [\log r]^{\lambda/c_1})$, it follows that
      \[
            \p( C_u \xleftrightarrow{\, \omega_{2} \cap B(2r [\log r]^{\lambda/c_1})\,} C_v \mid A_{u,v}) \geq 1 - e^{-[\log n]^{3c_1\lambda/4}}
      \]
      under the same conditions.
      Letting $\mathcal U$ be the event that all $\omega_{7/4}$-clusters that intersect both $S_r$ and $S_{3r}$ are contained in a single $\omega_{2} \cap B(2r [\log r]^{\lambda/c_1})$-cluster, it follows by a union bound that there exists $\lambda_1=\lambda_1(d,D) \geq \lambda_0$ such that if $\lambda \geq \lambda_1$ and $n \geq n_2$ then
      \begin{align}
            \p( \mathcal U )  & \geq 1 - \sum_{u,v \in S_r} \p\Bigl(A_{u,v}\cap \{C_u \nxleftrightarrow{ \omega_{2} \cap B(2r [\log r]^{\lambda/c_1})} C_v\}\Bigr)
            \nonumber \\
            &\geq 1- \op{Gr}(r)^2 e^{-[\log n]^{3c_1\lambda/4}} 
            \geq 1 - e^{[\log n]^{2D}} e^{-[\log n]^{3c_1\lambda/4}} 
            \geq 1-e^{-[\log n]^{2c_1\lambda/3}}.
    \end{align}

      For each vertex $x$, let $\mathcal E_x$ be the event that $S_{0.9 r}(x)$ is $\omega_{7/4}$-connected to $S_{3.6r}(x)$ and let $\mathcal{U}_x$ be the event that  all $\omega_{7/4}$-clusters that intersect both $S_r(x)$ and $S_{3r}(x)$ are contained in a single $\omega_{2} \cap B(x,2r [\log r]^{\lambda/c_1})$-cluster.
      Observe that if $\zeta$ is a path in $G$ then we have the inclusion of events
      \[
      \{u \xleftrightarrow{\omega_{2} \cap B_{n}(\zeta)} v \} \supseteq \{u \xleftrightarrow{\omega_{7/4}} S_{n}(u)\}\cap \{v \xleftrightarrow{\omega_{7/4}} S_n(v)\} \cap \bigcap_{t=0}^{\op{len}\zeta} (\mathcal U_{\zeta_t} \cap \mathcal E_{\zeta_t}).
      \]
         Thus, it follows by the Harris-FKG inequality and a union bound that there exists a constant $n_3=n_3(d,D,\lambda) \geq n_2$ such that if
         $\lambda \geq \lambda_1$, $n \geq n_3$, and the path $\zeta$ satisfies  $\op{len} \zeta \leq e^{[\log n]^{c_1\lambda/4}}$ then
      \begin{align*}
            \p_{p_2}( u \xleftrightarrow{B_{n}(\zeta)} v ) &\geq \p_{p_{7/4}}(u \leftrightarrow S_{n}(u)) \cdot \p_{p_{7/4}}(v \leftrightarrow S_{n}(v)) \cdot \prod_{t=1}^{\op{len}\zeta} \p(\mathcal E_{\zeta_t}) - \sum_{t=1}^{\op{len} \zeta} \p( \mathcal U_{\zeta^t}^c )\\
            &\geq \left(e^{-[\log \log n]^{1/2}}\right)^2 \left[ 1 - \op{len} (\zeta) e^{-[\log n]^{3c_1 \lambda/4 }} \right] - \op{len}(\zeta) e^{-[\log n]^{3c_1\lambda/4}} \\
            &\geq e^{-3[\log \log n]^{1/2}}. 
      \end{align*}
      The claimed lower bound on the corridor function follows since $\zeta$ was an arbitrary path of length at most $e^{[\log n]^{c_1\lambda/4}}$.
\end{proof}

\subsection{Concluding from a negative assumption}

The next lemma explains how to use the negative information encapsulated in an \emph{upper bound} on the two-point zone $\op{tz}(m)$ to find a set of vertices for which point-to-point connection probabilities within a ball are uniformly small. This lemma plays an analogous role to that of Section 7.2 in \cite{contreras2022supercritical}, but our proof is completely different and relies on the machinery developed in \cref{sec:snowballing}. Note that this lemma does not use the plentiful tubes condition.

\begin{lem} \label{lem:existence_of_well-sep_set}
 There exists a constant $n_0=n_0(d,D)$ such that if $n\geq n_0$ and $b \leq \frac{1}{8} n^{1/3}$ satisfy $\mathcal{A}_{D,K}(b,n)$ then
 for every $n^{1/3} \leq m \leq n$, there exists a subset $U \subseteq B_{\op{tz}(m)}$ with $\abs{U} \geq (\log n)^{c_3K}$ such that 
      \[
            \p_{p_1}( u \xleftrightarrow{B_{m/2}} v ) \leq (\log n)^{-c_3K}
      \]
      for all distinct $u,v \in U$.
\end{lem}

(Reminder: The implicit assumption $K \geq K_0$ remains in force.)

\begin{proof}[Proof of \cref{lem:existence_of_well-sep_set}]
Fix $n$ and $b\leq \frac{1}{8} n^{1/3}$ such that $\mathcal{A}_{D,K}(b,n)$ holds. We will assume that the claim is \emph{false} for some particular $n^{1/3} \leq m\leq n$, and show that this implies a contradiction when $n$ is sufficiently large. By definition of the two-point zone $\op{tz}(m)=\op{tz}(m,n)$, there exist vertices $u,v \in B_{\op{tz}(m)}$ such that $\p_{p_{3/2}}( u \xleftrightarrow{\,B_m\,} v ) < (\log n)^{-1}$. Let $\gamma$ be a geodesic in $B_{\op{tz}(m)}$ from $u$ to $v$. Recursively pick a sequence of indices $0=i_0 < i_1 < \ldots < i_k = \op{len} \gamma$ starting with $i_0 := 0$ and for each $j\geq 0$ with $i_j < \op{len} \gamma$ setting
\[
i_{j+1} = \max\Bigl\{\op{len} \gamma,\; 1 + \max\Bigl\{i \geq i_j: \p_{p_1}( u_{i_{j}} \xleftrightarrow{B_{m/2}} u_i ) \geq (\log n)^{-c_3K}\Bigr\}\Bigr\}.
\]
%
To lighten notation, define $v_j := u_{i_j}$ for every $0 \leq j \leq k$.  This sequence has the property that
$\p_{p_1}(v_j \leftrightarrow v_{j+1}) \geq p_1 (\log n)^{-c_3K}$ for every $0 \leq j <k$ and $\p_{p_1}(v_j \leftrightarrow v_\ell)\leq (\log n)^{-c_3K}$ for every distinct $0 \leq j ,\ell < k$. (The last connection probability, from $v_{k-1}$ to $v_k=v$, may be larger.) 
As such, our assumption guarantees that $k \leq  (\log n)^{c_3K}$.

\medskip

To conclude the proof, it suffices to show that $\p_{p_{3/2}}(u \leftrightarrow v) \geq (\log n)^{-1}$ when $n$ is sufficiently large. To this end, we would like to apply \cref{prop:snowballing} where the sets `$A_i$' are the balls $B_b(v_j)$, the superset `$\Lambda$' is the bigger ball $B_{m/2}$, the ghost field intensity `$h$' is $h:=\min\{\log n,\op{Gr}(b)\}^{-1}$, and the sprinkling amount `$\delta$' is $\delta/2$. To do this, we need to verify that 
\begin{equation}
\label{eq:well-sep-need-to-verify}
      h^{c_2 (\delta/2)^3} \leq c_2 k^{-1}
      \qquad \text{ and } \qquad
      \tau_{p_1}^{B_{m/2}}( B_{b}(v_j) \cup B_b(v_{j+1}) ) \geq 4 h^{-c_2 (\delta/2)^4}
\end{equation}
for every $0 \leq j \leq k-1$ when $n$ is sufficiently large.
We have by definition of $\delta$  (and the assumption $K \geq K_0$) that if $n$ is larger than some universal constant then
\begin{equation}
      h^{c_2 (\delta/2)^3} \leq 4h^{c_2 (\delta/2)^4} \leq 4(\log n)^{-2 c_3 K} \leq c_2 \left[(\log n)^{-c_3K}+1\right].\label{eq:sep_verify2}
\end{equation}
Since $k\leq (\log n)^{c_3K}$, this is easily seen to imply that the first inequality of \eqref{eq:well-sep-need-to-verify} holds. Moreover, we have by the same calculation performed in \eqref{eq:good_connection_at_dist_z} that
\begin{equation}
\label{eq:good_connection_at_dist_z2}\tau_{p_1}^{B_{m/2}}(B_{b}) \geq \frac{1}{2}e^{-(\log \log n)^{1/2}} \geq (\log n)^{-1}\end{equation} for all $n$ larger than some universal constant, and it follows by the Harris-FKG inequality that if $n$ is larger than some constant depending only on $d$ then 
\begin{align}
      \tau_{p_1}^{B_{m/2}}( B_{b}(v_j) \cup B_b(v_{j+1}) ) &\geq \tau_{p_1}^{B_{m/2}}(B_{b}(v_j)) \cdot \p_{p_1}( v_j \xleftrightarrow{B_{m/2}} v_{j+1}  ) \cdot \tau_{p_1}^{B_{m/2}}(B_{b}(v_{j+1})) \nonumber\\
      &\geq (\log n)^{-1} \cdot \Bigl[ \frac{1}{2d} (\log n)^{-c_3K} \Bigr] \cdot (\log n)^{-1} \geq (\log n)^{-c_3K-3}
      \label{eq:sep_verify1}
\end{align}
for every $0 \leq j \leq k-1$.
The estimates \eqref{eq:sep_verify2} and \eqref{eq:sep_verify1} together yield that there exists a constant $n_0=n_0(d)\geq N$ such that the required estimates \eqref{eq:well-sep-need-to-verify} hold whenever $n \geq n_0$.
Thus, \cref{prop:snowballing} and \eqref{eq:good_connection_at_dist_z2} yield that there  exists a constant $n_1 =n_1(d) \geq n_0$ such that if $K \geq K_0$ and $n\geq n_1$ and we define $r$ to be the minimum positive integer such that $\p_{q}(\op{Piv}[1,r h]) < h$ for every $q \in [p_{1},p_{3/2}]$ then (since $p_{3/2} \geq \sprinkle(p_1;\delta/2)$)
\[
      \p_{p_{3/2}} \left( u \xleftrightarrow{B_{m/2+2r}} v \right) \geq c_2 \tau_{p_1}^{B_{m/2}}(B_b(u)) \cdot \tau_{p_1}^{B_{m/2}}(B_b(v)) 
      \geq \frac{c_2}{4} e^{-2[\log \log n]^{1/2}}.
\]
Finally, the same argument as in the proof of \cref{lem:can_always_cross_tz_tubes} yields that $2r \leq m/2$ when $n$ is larger than some constant depending on $d$ and $D$, and it follows that there exists a constant $n_2 =n_2(d,D) \geq n_1$ such that if $n\geq n_2$ then $\p_{p_{3/2}}(u \xleftrightarrow{B_m} v) \geq (\log n)^{-1}$, a contradiction.
\end{proof}

We now use the existence of this set of poorly connected vertices (negative information) to prove that $S_{\op{tz}(m)}$ is very likely to be connected to the boundary of $S_{m/2}$ (positive information). This only works because we are working  under the positive hypothesis of a two-point lower bound at scale $n$. This step is essentially the same as Section 7 in \cite{contreras2022supercritical}, with our two-point lower bound at scale $n$ playing the role of `being in the supercritical regime' in their setting. 

\begin{lem} \label{lem:super_strong_connection_from_little_to_big_sphere}
 There exists a constant $n_0=n_0(d,D)$ such that if $n\geq n_0$ and $b \leq \frac{1}{8} n^{1/3}$ satisfy $\mathcal{A}_{D,K}(b,n)$ then
      \[
            \p_{p_{3/2}} ( S_{ \op{tz}(m)} \leftrightarrow S_{m/2} ) \geq 1 - e^{-(\log n)^{c_3K-1}}
      \]
for every $m \in \mathscr{S}(n)$.
\end{lem}

Note that we are not actually assuming a negative information assumption (such as $\neg \mathcal{B}(n,p)$) in the hypotheses of this lemma. The lemma holds without any such assumption, but is stronger when $\op{tz}(m)$ is small.
The proof of \cref{lem:super_strong_connection_from_little_to_big_sphere} will apply the following proposition.


\begin{prop} \label{prop:hamming}
      Let $G$ be a finite connected graph, let $p \in [0,1]$, and let $A,B \subseteq V(G)$. If $\theta\in (0,1)$ is such that
      \[
            \min_{x \in A} \p_p( x \leftrightarrow B) \geq \theta \geq 2 \abs{A} \max_{\substack{x,y \in A\\ x \not= y}} \p_p( x \leftrightarrow y ),
      \]
      then $\p_{q}(A \leftrightarrow B) \geq 1 - e^{-\delta(p,q) \theta \abs{A}}$ for every $q \in (p,1)$.
\end{prop}

\begin{proof}[Proof of \cref{prop:hamming}]
This proposition is essentially the same as \cite[Proposition 7.2]{contreras2022supercritical}, except that it is stated in terms of our sprinkling coordinates introduced in \cref{sec:induction_step} (which are natural from the perspective of Talagrand's inequality) and we get a factor $\delta$ rather than $2\delta$ in the exponential in the conclusion. Both versions of the proposition are elementary consequences of the differential inequality
\[
\frac{d}{dp} (-\log \p_p(A)) \geq \frac{1}{p(1-p)}\e_p \bigl[\text{Hamming distance from $\omega$ to $A$}\bigr],
\]
which holds for every finite graph and every decreasing event $A$ \cite[Theorem 2.53]{MR2243761}.  In our coordinates, this inequality reads
    \[
\frac{d}{dt} (-\log \p_{\sprinkle(p;t)}(A)) \geq \frac{\log1/(1-\sprinkle(p;t))}{\sprinkle(p;t)}\e_{\sprinkle(p;t)} \bigl[\text{Hamming distance from $\omega$ to $A$}\bigr].
\]  
In our case the prefactor $-(\log(1-\sprinkle(p;t)))/\sprinkle(p;t)$ is at least $1$ whereas in the original inequality the prefactor $1/(p(1-p))$ is at least $2$, leading to the difference between our two conclusions.{}
\end{proof}





\begin{proof}[Proof of \cref{lem:super_strong_connection_from_little_to_big_sphere}]
Let $n_0=n_0(d,D)$ be as in \cref{lem:existence_of_well-sep_set}, and suppose that  $n\geq n_0$ and $b\leq \frac{1}{8}n^{1/3}$ are such that $\mathcal{A}(b,n)$ holds.
      Let $U \subseteq B_{\op{tz}(m)}$ be the set of vertices guaranteed to exist by \cref{lem:existence_of_well-sep_set}, and let $A$ be a subset of $U$ with $\abs{A} = \lceil [\log n]^{c_3K-1/2} \rceil$. Since $\mathcal{A}(b,n)$ holds, we have that
      \[
            \min_{x \in A} \p_{p_1}( x \leftrightarrow S_{m/2} ) \geq \kappa_{p_1}(n,\infty) \geq e^{-[\log \log n]^{1/2}},
      \]
      while our choice of $A$ guarantees that
      \[
            2 \abs{A} \max_{\substack{x,y \in A\\ x \not= y}} \p_{p_1}( x \xleftrightarrow{B_{m/2}} y ) \leq 2 \lceil(\log n)^{c_3K-1/2 }\rceil (\log n)^{-c_3K} \leq e^{-(\log \log n)^{1/2}}
      \]
      whenever $n$ is larger than some constant $n_1=n_1(d,D) \geq n_0$. (This constant does not depend on $K$ since $c_3 K \geq c_3 K_0 \geq 2>1/2$.)
      Thus, applying \cref{prop:hamming} with $\theta = e^{-[\log \log n]^{1/2}}$ yields that (since $p_{3/2} \geq \sprinkle(p_1,\delta/2)$)
      \begin{equation}\label{eq:A_to_Sm/2}
           \p_{p_{3/2}}(B_{\op{tz}(m)}\leftrightarrow S_{m/2}) \geq \p_{p_{3/2}}( A \leftrightarrow S_{m/2} ) \geq 1 - \exp\left(-\frac{\delta}{2} e^{-[\log \log n]^{1/2}} \abs{A}\right).
      \end{equation}
      On the other hand, the definition of $\delta$ ensures that $\delta \geq [\log n]^{-1/4}$ and hence that there exists a constant $n_2=n_2(d,D) \geq n_1$ such that
      \[
             \frac{\delta}{2} e^{-[\log \log n]^{1/2}}  \abs{A} \geq \frac{1}{2} e^{-[\log \log n]^{1/2}} (\log n)^{c_3K-1/2-1/4} \geq (\log n)^{c_3K-1}
      \]
      whenever $n \geq n_2$, which implies the claim in conjunction with \eqref{eq:A_to_Sm/2}.
\end{proof}

The next lemma is completely elementary. It tells us that we can find two nearby reals $m$ and $m'$ where $\op{tz}(m)$ is close to $\op{tz}(m')$. 

\begin{lem} \label{lem:exists_m_where_tz_doesnt_jump_much}
Let $R \geq 1$. There exists a constant $n_0=n_0(d,D,R)$ such that if $n \geq n_0 \vee N$ then there exists $m \in \mathscr{S}(n)$ such that $m (\log n)^{-R} \in \mathscr{S}(n)$ and
      \[
            \frac{\op{tz}(m)}{\op{tz}(m(\log n)^{-R})} \leq (\log n)^{8R/c_1}.
      \]
\end{lem}

\begin{proof}[Proof of \cref{lem:exists_m_where_tz_doesnt_jump_much}]
Let $n\geq N$ so that $\mathscr{S}(n)$ is defined.
Let $s$ and $t$ denote the left and right endpoints of $\mathscr{S}(n)$, and define
\[
      k := \left\lfloor \log_{[\log n ]^{R}}(t/s) \right \rfloor = \left \lfloor \frac{c_1 \log (m_1)}{2 R \log \log n} \right \rfloor \geq \left \lfloor \frac{c_1 \log n}{6 R \log \log n} \right \rfloor
\]
where $n^{1/3} \leq m_1 \leq n$ is as in the definition of $\mathscr{S}(n)$.  
If a suitable $m\in \mathscr{S}(n)$ does \emph{not} exist then, using the trivial inequalities $\op{tz}(t)\leq t \leq n$ and $\op{tz}(s)\geq 1$, we must have that
\begin{multline*}
      n \geq \op{tz}(t) \geq \op{tz}(s) \prod_{i=1}^k \frac{\op{tz}(s [\log n]^{iR}) }{ \op{tz}(s [\log n]^{(i-1)R} ) } 
      \\ \geq \op{tz}(s) \left([\log n]^{8R/c_1}\right)^{k} 
      \geq \exp\left( \frac{8R}{c_1}\left \lfloor \frac{c_1 \log n}{6 R \log \log n} \right \rfloor \cdot \log \log n \right).
\end{multline*}
Since $8>6$, this yields a contradiction when $n$ is larger than some constant $n_0=n_0(d,D,R)$ (allowing us to approximately remove the effect of rounding down).
\end{proof}

We will now combine our lemmas to prove \cref{lem:complement_of_done_if_can_cross_tubes}. The idea here is inspired by the \emph{uniqueness via sprinkling} argument from \cite[Section 8]{contreras2022supercritical}, which itself used ideas from \cite{MR3634283}. Our approach is different because we do not know that exposed spheres are well-connected. Instead, we have the polylog-plentiful tubes condition. This is a much weaker geometric control because the tubes are not constrained to lie within narrow annuli. The main step is to use the strong connectivity bound from \cref{lem:super_strong_connection_from_little_to_big_sphere} to deduce that with high probability, every $\omega_{7/4}$-cluster crossing a thick annulus is contained in a single $\omega_{2}$-cluster. (As before we abbreviate $\omega_2=\omega_{p_2}$ and so on.) In \cite{contreras2022supercritical}, the analogous step was carried out by dividing the thick annulus into thinner annuli before showing that if two clusters cross multiple annuli, then, after sprinkling in those annuli, the clusters will merge with high probability. This works because in every thin annulus, there is some good probability that the clusters will merge after sprinkling in the annulus, thanks to the connectivity of exposed spheres. In our case, we also track how many clusters survive un-merged as they cross through multiple annuli. The difference is that we will have to sprinkle \emph{everywhere} each time we cross a thin annulus. Nevertheless, we will sprinkle so little at each stage that the net effect is to sprinkle by less than $\delta/4$, as required.


\begin{proof}[Proof of \cref{lem:complement_of_done_if_can_cross_tubes}]
Let $K_1=K_1(d,D,\lambda)=\max\{K_0,40 \lambda (c_1 \vee 1)^{-2}c_3^{-1},4c_3^{-1}(D \vee 1)\}$ and suppose that $K\geq K_1$ and $n\geq N$. Fix $R := 5 \lambda /c_1$ and let $n_0=n_0(d,D,R)$ be the constant from \cref{lem:exists_m_where_tz_doesnt_jump_much}, which by our choice of $R$ depends only on $d$, $D$, and $\lambda$. 
 We also let $n_1=n_1(d,D)$ and $c_4$ be the constants from \cref{lem:can_always_cross_tz_tubes}. 

\medskip

 Suppose that $n\geq n_2=n_2(d,D,\lambda)=n_0 \vee n_1 \vee N \vee e^{2^R}$ and $b\leq \frac{1}{8}n^{1/3}$ are such that $\mathcal{A}(b,n,p)$ holds, and let $m \in \mathscr{S}(n)$ be the element guaranteed to exist by \cref{lem:exists_m_where_tz_doesnt_jump_much} applied with this value of $R$. Taking $n_2 \geq e^{2^R}$ guarantees that $m/2 \geq m(\log n)^{-R}$ and hence that $m/2\in \mathscr{S}(n)$ when $n\geq n_2$. 
Since $n \geq n_1$, we have by \cref{lem:can_always_cross_tz_tubes} that
\[
      \kappa_{p_{7/4}} \left( \op{tz}(m (\log n)^{-R}) (\log n)^{c_3 K}, 2 m (\log n)^{-R} \right) \geq c_4 e^{-2 [\log \log n]^{1/2}}.
\]
On the other hand, our choice of $R$ and $m$ ensure that there exists a constant $n_3=n_3(d,D,\lambda) \geq n_2$ such that if $n\geq n_3$ then we have the inclusion of intervals
\[
      \Bigl( 2 m [\log n]^{-R} ,  \op{tz}\bigl(m [\log n]^{-R}\bigr) [\log n]^{c_3K} \Bigr) \supseteq \Bigl( \frac{m}{3} [\log m]^{-4\lambda/c_1} , 2\op{tz}(m/2) +2 \Bigr), 
\]
so that
\begin{equation} \label{eq:can_cross_tz_balls_well_in_final_proof_of_low_growth}
      \kappa_{p_{7/4}}\Bigl( 2\op{tz}(m/2)+2 ,  \frac{m}{3} [\log m]^{-4\lambda/c_1}  \Bigr) \geq c_4 e^{-2[\log \log n]^{1/2}}.
\end{equation}
Note that this estimate holds as a consequence of $\mathcal{A}(b,n,p)$ alone: we have not yet made us of the negative information $\neg \mathcal{B}(n,p)$.

\medskip

Now suppose that $\neg \mathcal{B}(n,p)$ holds. Since $K\geq K_1\geq 12 c_3^{-1} c_1^{-1}\lambda$, there exists a constant $n_3=n_3(d,D,\lambda) \geq n_3$ such that if $n\geq n_3$ then
\begin{equation} \label{eq:tz_is_small_in_final_proof_of_low_growth}
      \op{tz}(m/2) \leq m [\log n]^{3\lambda/c_1-c_3 K} \leq  m (\log n)^{-3c_3K/4} \leq \frac{m}{17}(\log m)^{-c_3K/2}
\end{equation}
for every $m\in \mathscr{S}(n)$.

\medskip

Our next goal is to prove a good upper bound on the probability of the non-uniqueness event $\op{Piv}_{p_{7/4},p_2}[ m/16,m/8 ]$, 
where $\op{Piv}_{p,q}[m,n]$ denotes the event that there are at least two distinct $\omega_p$-clusters that each intersect both $B_m$ and $S_n$ but that are not connected to each other by any path in $B_n \cap \omega_{q}$. We will do this using a variation on the ``orange peeling'' argument of \cite{contreras2022supercritical}, where we iteratively sprinkle and zoom in closer to $m/16$ over a number of steps.

\begin{figure}
\center
\includegraphics[width=0.9\textwidth]{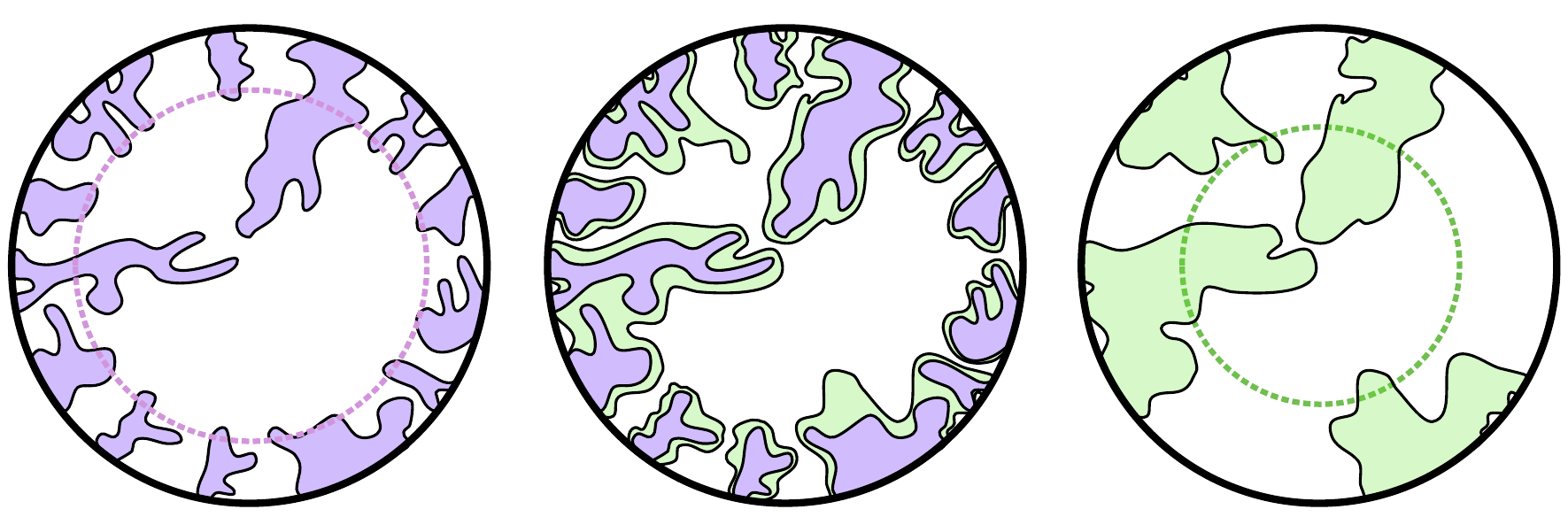} 
\caption{Schematic illustration of the construction of $\mathscr C_{i+1}$ from $\mathscr C_i$: The black, purple, and green circles represent the spheres $S_{r_0}$, $S_{r_{i}}$, and $S_{r_{i+1}}$ respectively. The purple regions in the left figure represent the clusters making up $\mathscr C_i$. By construction each of these is a $B_{r_0} \cap \omega_{(i)}$-cluster that connects $S_{r_0}$ to $S_{r_i}$. The green regions in the middle figure represent the $B_{r_0} \cap \omega_{(i+1)}$ clusters that contain some cluster in $\mathcal C_i$. Finally, the green regions remaining in the third figure represent the subset of these clusters that happen to intersect $S_{r_{i+1}}$; these make up $\mathscr C_{i+1}$.
}
\label{fig:orange_peeling}
\end{figure}

\medskip

 Let $k := 2 \lfloor (\log n)^D \rfloor$, $\eps := (\log n)^{-(D+1)}$, and for each $i \in \{0,...,k\}$ set
\[
      r_i := \frac{m}{8} - \frac{im}{40} [\log n]^{-D} \qquad \text{and} \qquad q_i := \sprinkle(p_{7/4}; i \eps).
\]
Note that $r_i \in [m/16,m/8]$ and $q_i \in [p_{7/4},p_2]$ for every $0\leq i\leq k$. We work with the standard monotone coupling $(\omega_q)_{q\in [0,1]}$, and write $\omega_{(i)}=\omega_{q_i}$ for each $i \in \{ 0, \ldots , k\}$. (Be careful not to confuse this with our previous notational shorthand $\omega_1=\omega_{p_1}$, $\omega_{2}=\omega_{p_2}$.)
Given the family of configurations $(\omega_q)_{q \in [0,1]}$, recursively define a set of $B_{m/8} \cap \omega_{(i)}$-clusters $\mathscr C_i$ for each $i \in \{0,...,k\}$ as follows: 
\begin{enumerate}
    \item
Let $\mathscr C_0$ be the set of all $B_{m/8} \cap \omega_{(0)}$-clusters that contain a vertex in $S_{r_0}$. 
\item Given $\mathscr C_i$ for some $i < k$, let $\mathscr C_{i+1}$ be the set of $B_{m/8} \cap \omega_{(i+1)}$-clusters $C$ such that there exists $C' \in \mathscr C_i$ with $C' \subseteq C$ and $C' \cap S_{r_{i+1}} \not= \emptyset$. 
\end{enumerate}
See \cref{fig:orange_peeling} for an illustration.
This definition ensures that the cardinality $|\mathscr{C}_i|$ is a decreasing function of $i$ and that we have the inclusion of events
\begin{equation}
\op{Piv}_{p_{7/4},p_2}[ m/16,m/8 ] \subseteq \bigl\{\mathscr C_k \in \{0,1\}\bigr\}.
\end{equation}
Roughly speaking, our goal is to show that, for each $i$, if $\mathscr{C}_i$ is not a singleton then the cardinality $\abs{\mathscr C_{i+1}}$ is 
smaller than $\abs{\mathscr C_i}$ by a factor of roughly $1/2$  with high probability under $\p$.

\medskip

Ideally, we would like to show that, with high probability on the event $\{|\mathscr{C}_i|>1\}$, every cluster in $\mathscr C_i$ that intersects $S_{r_{i+1}}$ is $B_{m/8} \cap \omega_{(i+1)}$-connected to a distinct cluster in $\mathscr C_i$. This is what we will prove, except that we will allow one distinguished cluster to go un-merged. 
%
%
%
For every non-empty set of clusters $\mathscr F$, permanently fix a choice of element $\op{min}(\mathscr F) \in \mathscr F$ such that $\op{dist}(o,\op{min}(\mathscr F)) = \op{dist}(o,\bigcup \mathscr F)$.
  For each $0\leq i \leq k-1$, consider the event 
\[\mathcal E_i = \{\mathscr C_i = \emptyset\} \cup \Bigl\{C \xleftrightarrow{B_{m/8} \,\cap\, \omega_{(i+1)}} \bigcup ( \mathscr C_i \backslash \{ C \} ) \text{ for every } C \in \mathscr C_i \backslash \{ \op{min}(\mathscr C_i) \}\text{ with }\op{dist}(o,C) \leq r_{i+1}\Bigr\}\]
 We will prove the following lemma at the end of the section after explaining how it may be used to conclude the proof of \cref{lem:complement_of_done_if_can_cross_tubes} (and hence of \cref{prop:low_growth_main}).

 \begin{lem}[Merging clusters]
\label{lem:merging}
There exists a constant $n_{11}=n_{11}(d,D,\lambda)\geq n_4$ such that if $n\geq n_{11}$ then
\begin{equation}\label{eq:merging}
      \p( \mathcal E_i ) \geq 1- e^{-[\log n]^{c_1\lambda/2}}.
\end{equation}
for every $0 \leq i \leq k-1$.
\end{lem}

(We name this constant $n_{11}$ to leave room for the constants $n_6$ through $n_{10}$ that will appear in the proof of this claim.)

\medskip

Let us now conclude the proof of \cref{lem:complement_of_done_if_can_cross_tubes} given \cref{lem:merging}. 
 On the event $\mathcal E_i$ we have that
\[
      \abs{\mathscr C_{i+1}} \leq \left\lfloor \frac{\abs{\mathscr C_i} - 1}{2} \right\rfloor + 1.
\]
Since $\abs{\mathscr C_0} \leq \abs{S_{r_0}} \leq e^{(\log n)^D}$ and $k := 2 \lfloor [\log n]^D \rfloor$, it follows that $|\mathscr C_k| \in \{0,1\}$ on the event $\bigcap_{i =0}^{k-1} \mathcal E_i$. It follows by a union bound that there exists a constant $n_6=n_6(d,D,\lambda) \geq n_5$ such that if $n \geq n_6$ then
\begin{equation}
\label{eq:nearly_done1}
      \p \left( \op{Piv}_{p_{7/4},p_2}[ m/16,m/8 ]  \right) 
      \leq \sum_{i=0}^{k-1} \p_p(\mathcal E_i^c) \leq 
        2(\log n)^D e^{- (\log n)^{c_1 \lambda/2} } \leq e^{-(\log n)^{c_1 \lambda/3}}.
\end{equation}
On the other hand, \eqref{eq:tz_is_small_in_final_proof_of_low_growth} ensures that if $n\geq n_6$ then $\op{tz}(m/2) \leq m/17$, and it follows from \cref{lem:super_strong_connection_from_little_to_big_sphere} applied to $m/2$ that
\begin{equation*}
      \p_{p_{7/4}} \left( S_{m/17} \leftrightarrow S_{m/4} \right) \geq 1 - e^{-(\log n)^{c_1 \lambda/3}}
\end{equation*}
if $n \geq n_6$. It follows in particular that there exists $n_7=n_7(d,D,\lambda)\geq n_6$ such that
\begin{equation}
\label{eq:nearly_done2}
\p_{p_{7/4}} \left( S_{m/17} \leftrightarrow S_{m/4} \right)^{e^{[\log n]^{c_1 \lambda /4}}} \geq \frac{1}{2}
\end{equation}
if $n\geq n_7$; this will be used to form a chain of connected annuli using the Harris-FKG inequality.

\medskip

We now apply \eqref{eq:nearly_done1} and \eqref{eq:nearly_done2} to bound the corridor function.
Let $\gamma$ be a path of length at most $e^{[\log n]^{c_1\lambda/4}}$, starting at some vertex $u$ and ending at some vertex $v$, and observe that we have the inclusion of events
\begin{multline*}
\bigl\{u \xleftrightarrow{B_n(\gamma) \cap \omega_{2}} v\bigr\} \subseteq
\bigl\{u \xleftrightarrow{\omega_{7/4}} S_n(u)\bigr\} \,\cap\, \bigr\{v \xleftrightarrow{\omega_{7/4}} S_n(v)\} \,\cap \\\bigcap_{t=1}^{\op{len} \gamma} \left(\bigl\{S_{m/17}(\gamma_t) \xleftrightarrow{\omega_{7/4}} S_{m/4}(\gamma_t)\bigr\}\,\cap\,\op{Piv}_{p_{7/4},p_2}[m/16,m/8](\gamma_t) \right).
\end{multline*}
  Applying the Harris-FKG inequality and a union bound, we deduce that there exists a constant $n_8=n_8(d,D,\lambda)$ such that if $n \geq n_8$ then
\begin{align*}
      \p_{p_2}( u \xleftrightarrow{B_n(\gamma)} v ) &\geq \p_{p_{7/4}}(o \leftrightarrow S_n )^2 \cdot \p_{p_{7/4}} \left( S_{m/17} \leftrightarrow S_{m/4} \right)^{\op{len} (\gamma)} \\& \hspace{5cm} - \op{len}(\gamma) \cdot (1-\p\left( \op{Piv}_{p_{7/4},p_{2}}[ m/16,m/8 ])  \right) \\
      &\geq \frac{1}{2}e^{-2(\log \log n)^{1/2}} - e^{(\log n)^{c_1\lambda/4}} e^{-(\log n)^{c_1 \lambda /3}} \\
      &\geq e^{-3 [\log \log n]^{1/2}},
\end{align*}
where we used the assumption that $\mathcal{A}(b,n,p)$ holds to bound $\p_{p_{7/4}}(o \leftrightarrow S_n ) \geq \kappa_{p_1}(n,\infty) \geq e^{-(\log \log n)^{1/2}}$.
The claimed lower bound on the corridor function follows since $\gamma$ was an arbitrary path of length at most $e^{[\log n]^{c_1\lambda/4}}$.
\end{proof}

It remains only to prove \cref{lem:merging}.

\begin{proof}[Proof of \cref{lem:merging}] We continue to use the notation from the proof of \cref{lem:complement_of_done_if_can_cross_tubes}, and in particular will use the constants $K_1=K_1(d,D,\lambda)$ and $n_4=n_4(d,D,\lambda)$ defined in that proof.
Consider the event $\Omega$ defined by
\[\Omega := \bigcap_{u \in B_{m/8}} \{S_{\op{tz}(m/2)}(u) \xleftrightarrow{\omega_{3/2}} S_{m/8}\}.\]
It follows by \cref{lem:super_strong_connection_from_little_to_big_sphere} (applied to $m/2$) and a union bound that there exists a constant $n_5 =n_5(d,D,\lambda) \geq n_4$ such that if $n\geq n_5$ then 
\begin{align}
      \p(\Omega) &\geq 1 - \op{Gr}(m) \sup_{u\in B_{m/8}} \p_{p_{3/2}}(S_{\op{tz}(m/2)}(u) \leftrightarrow S_{m/8})  \nonumber
      \\&\geq 1 - \op{Gr}(m) \sup_{u\in B_{m/8}} \p_{p_{3/2}}(S_{\op{tz}(m/2)}(u) \leftrightarrow S_{m/4}(u))\nonumber
\\&\geq 1-e^{(\log m)^D} e^{-(\log n)^{c_3K-1}} \geq 1- e^{-(\log n)^{c_3K/2}},
       \label{eq:every_small_ball_is_connected_to_big_sphere}
\end{align}
where we used that $K \geq K_1 \geq 4 c_3^{-1} (D \vee 1)$ in the final inequality.

\medskip

For each $0\leq i \leq k-1$, let $\mathbb F_i$ be the set $\mathbb F_i=\{ \mathscr{F}:$ $\p( \mathscr C_i = \mathscr F \mid \Omega ) > 0\}$. (Note that $\mathbb F_i$ is a set of sets of sets of vertices.) It follows from the definitions that
\begin{equation}\label{eq:Omega_balls_cover}
d\Bigl(u, \bigcup_{C\in \mathscr{F}} C\Bigr) \leq \op{tz}(m/2) \text{ for every $u \in B_{m/8}$ and $\mathscr F \in \mathbb F_i$.}
\end{equation}
By \eqref{eq:every_small_ball_is_connected_to_big_sphere} and a union bound, it suffices to prove that there exists $n_{11}=n_{11}(d,D,\lambda) \geq n_5$ such that
\begin{equation}\label{eq:conditional_merging}
      \p ( \mathcal E_i \mid \mathscr C_i = \mathscr F ) \geq 1- \frac{1}{2} e^{-[\log n]^{c_1\lambda/2}}
\end{equation}
for every $0\leq i \leq k-1$ and every $\mathscr F \in \mathbb F_i$. 

\medskip

Before doing this, we will need to prove a purely geometric preliminary claim.
      Let $0\leq i \leq k-1$, let $\mathscr F \in \mathbb F_i$, and let $C \in \mathscr F \backslash \{ \min(\mathscr F) \}$. We claim that there exists a constant  $n_8 =n_8(d,D,\lambda)\geq n_5$  such that if $n\geq n_8$ and $\op{dist}(o,C) \leq r_{i+1}$ then there exists a set of vertices $U \subseteq B(r_{i+1} + m(\log m)^{-\lambda/c_1})$ with $\abs{U} \geq \frac{1}{2} (\log m)^{c_1\lambda}$ such that the following hold:
      \begin{enumerate}
      \item[(i)]
       $U$ is $m(\log m)^{-4\lambda/c_1}$-separated. That is, pairwise distances between distinct points in $U$ are at least $m(\log m)^{-4\lambda/c_1}$.
       \item[(ii)] For each $u\in U$, the ball
            $B_{\op{tz}(m/2)+1}(u)$ intersects both $C$ and $\bigcup ( \mathscr F \backslash \{ C \} ) $.
  \end{enumerate}
  (As before, we name this constant $n_8$ to leave room for the constants $n_6$ and $n_7$ that will appear in the proof of this claim.)
We let $r := m (\log m)^{-2\lambda/c_1}$ and split the proof of this claim into two cases according to whether $\op{dist}(o,C)$ is smaller or larger than $r$. 

\medskip

\noindent \textbf{Case 1:} ($\op{dist}(o,C) \leq r$.)
 Since $C \not= \op{min}(\mathscr F)$, we have that $\op{dist}(o, \bigcup (\mathscr F \backslash \{C\})) \leq \op{dist}(o,C) \leq r$ also. Let $\Gamma$ be the family of paths from $B_{3r} \cap C$ to $B_{3r} \cap \bigcup ( \mathscr F \backslash \{C\} )$ that is guaranteed to exist by the fact that $G$ has $(c_1,\lambda)$-polylog-plentiful annular tubes at scale $r$. We now observe that for each $\gamma \in \Gamma$, there exists a vertex $u_\gamma$ on the path $\gamma$ satisfying the ball-intersection condition (ii):
 \begin{itemize}
 \item If $\max_t \op{dist}(C, \gamma_t) \leq \op{tz}(m/2)$ then we may take $u_\gamma$ to be the final vertex of $\gamma$. 
 \item Otherwise, if $\max_t \op{dist}(C, \gamma_t) >\op{tz}(m/2)$, we may take $u_\gamma= \gamma_{t_\gamma}$ where $t_\gamma$ is the maximum index such that $\op{dist}(C,\gamma_t) \leq \op{tz}(m/2)$. To see that this choice of $u_\gamma$ satisfies (ii), note that $\op{dist}(\gamma_{t_\gamma+1},C) > \op{tz}(m/2)$ and hence by \eqref{eq:Omega_balls_cover} that $\op{dist}(\gamma_{t_\gamma+1},\bigcup (\mathscr F \backslash \{C\})) \leq \op{tz}(m/2)$.
 \end{itemize} Now define $U := \{ u_\gamma : \gamma \in \Gamma\}$. Since the family $\Gamma$ is $2r(\log r)^{-\lambda/c_1}$-separated (and $r$ was defined to be $r=m(\log m)^{-2\lambda/c_1}$), it follows that there exists a constant $n_6 =n_6(d,D,\lambda) \geq n_5$ such that if $n\geq n_6$ then
  $U$ is $m(\log m)^{-4\lambda/c_1}$-separated and
\[
      \abs{U} = \abs{\Gamma} \geq (\log r)^{c_1\lambda} \geq \frac{1}{2} (\log m)^{c_1\lambda}.
\]
Moreover, since every path in $\Gamma$ was contained in $B(3r+r[\log r]^{\lambda/c_1})$, there exists a constant $n_7 =n_7(d,D,\lambda) \geq n_6$ such that if $n\geq n_7$ then $3r+r[\log r]^{\lambda/c_1} \leq m/16 \leq r_{i+1} + m(\log m)^{-\lambda/c_1}$ and hence $U \subseteq B(r_{i+1} + m(\log m)^{-\lambda/c_1})$.

\medskip

\noindent \textbf{Case 2:} ($\op{dist}(o,C) > r$.)  Let $v \in C$ be such that $\op{dist}(o,v) = \op{dist}(o,C)$, let $\gamma^a$ be a path in $C$ from $v$ to $S_r(v)$, and let $\gamma^b$ be the portion of a geodesic from $v$ to $o$ starting at a neighbour $u$ of $v$ with $\op{dist}(o,u)<\op{dist}(o,v)$ and ending at the first intersection with $S_r(v)$. These path are both finite, start in $S_1(v)$ and end in $S_{r}(v)$, and are contained in $C$ and disjoint from $C$ respectively. 
%
 Let $\Gamma$ be the family of paths from $\gamma^a$ to $\gamma^b$ that is guaranteed to exist by the fact that $G$ has $(c_1,\lambda)$-polylog-plentiful annular tubes at scale $r/3$. We can construct the desired set $U$ by picking a vertex $u_\gamma$ in $\gamma$ satisfying the condition (ii) for each $\gamma \in \Gamma$; the fact that such a vertex exists for each $\gamma$ follows by the same argument used in case 1 above. Moreover, it follows by the same argument used in the first case that there exists a constant $n_8=n_8(d,D,\lambda)\geq n_7$  such that if $n\geq n_8$ then the required bounds on the cardinality and separation of the set $U=\{u_\gamma:\gamma\in \Gamma\}$ hold, as well as the containment $U \subseteq B(r_{i+1} + m(\log m)^{-\lambda/c_1})$.
This concludes the proof of the geometric claim.

\medskip

We now use this geometric claim to establish the estimate \eqref{eq:conditional_merging}, which will complete the proof of the lemma. Let $i \in \{0,\ldots,k-1\}$ and let $\mathscr F \in \mathbb F_i$ be arbitrary. Consider also an arbitrary $C \in \mathscr F \backslash \{ \min(\mathscr F) \}$ with $\op{dist}(o,C) \leq r_{i+1}$, and let $U$ be the corresponding set of vertices guaranteed to exist by the geometric claim above. For each $u \in U$, let $\beta_u := B(u;\frac{m}{2} (\log m)^{-4 \lambda/c_1} ) \cap \omega_{(i+1)}$. Consider an arbitrary vertex $u\in U$. By construction of $U$, there exists a path $\gamma$ that starts in $C$, ends in $\bigcup( \mathscr F \backslash \{C\} )$, is contained in $B(u;\op{tz}(m/2) + 1)$, and has length at most $2 \op{tz}(m/2)+2$.
The estimate \cref{eq:tz_is_small_in_final_proof_of_low_growth} yields the existence of a constant $n_9 =n_9(d,D,\lambda)\geq n_8$ such that if $n\geq n_9$ then
\[
     [\op{tz}(m/2)+1] + \frac{m}{3} (\log m)^{-4\lambda/c_1} \leq \frac{m}{2} (\log m)^{-4\lambda/c_1},
\]
 so that the tube $B(\gamma;\frac{m}{3} (\log m)^{-4\lambda/c_1})$ associated to this path $\gamma$ is contained in the ball $B(u;\frac{m}{2} (\log m)^{-4\lambda/c_1})$. Thus, if $n\geq n_9$, the estimate  \eqref{eq:can_cross_tz_balls_well_in_final_proof_of_low_growth} yields that
\[
      \p\left( C \xleftrightarrow{\,\beta_u\, } \bigcup ( \mathscr F \backslash \{C\} ) \right) \geq c_4 e^{-2[\log \log n]^{1/2}}.
\]
We stress that the $C$ and $\mathscr F$ appearing in this inequality are deterministic, and do not depend on the configuration $\beta_u$. Under $\p$, the conditional law of $\beta_u$ given that $\mathscr C_i = \mathscr F$ is simply (inhomogeneous) bond percolation on $B(u;\frac{m}{2} (\log m)^{-4 \lambda/c_1} )$ where every edge has probability at least $c_{-1} \eps$ of being open, and every edge that does not touch $\bigcup \mathscr F$ has probability $q_{i+1}$ of being open.
In particular, recalling that $\eps=\delta(q_{i+1},q_i) = (\log n)^{-(D+1)}$, it follows that there exists a constant $n_{10}=n_{10}(d,D,\lambda)\geq n_9$ such that if $n\geq n_{10}$ then
\[
      \p\left( C \xleftrightarrow{\,\beta_u\, } \bigcup ( \mathscr F \backslash \{C\}) \mid \mathscr C_i = \mathscr F \right) \geq c_{-1}^2 \eps^2 c_4 e^{-2[\log \log n]^{1/2}} \geq (\log m)^{-2D-3}.
\]
Notice that under $\p$, the configurations $(\beta_u)_{u \in U}$ are independent. Moreover, this still holds after conditioning on the event that $\mathscr C_i = \mathscr F$. So, by independence, there exist constants $\lambda_0=\lambda_0(d,D)=8c_1^{-1}(2D+3)$ and $n_{11}=n_{11}(d,D,\lambda)\geq n_{10}$ such that if $\lambda \geq \lambda_0$ and $n \geq n_{11}$ then
\[\begin{split}
      \p\left( C \xleftrightarrow{ B_{m/8} \cap \omega_{(i+1)}} \bigcup ( \mathscr F \backslash \{C\} ) \mid \mathscr C_i = \mathscr F \right) &\geq 1 - \prod_{u \in U} \p\left( C \nxleftrightarrow{\beta_u } \bigcup ( \mathscr F \backslash \{C\} ) \mid \mathscr C_i = \mathscr F \right) \\
      &\geq 1 - ( 1 - (\log m)^{-2D-3} )^{\frac{1}{2}(\log m)^{c_1 \lambda} } \\
      &\geq 1 - e^{-2(\log n)^{c_1\lambda/2}}.
\end{split}\]
Finally, since $\abs{\mathscr F} \leq \abs{S_{r_0}} \leq e^{(\log n)^{D}}$, we have by a union bound that if $\lambda \geq \lambda_0$ and $n\geq n_{11}$ then
\[\begin{split}
      \p(\mathcal E_i \mid \mathscr C_i = \mathscr F) 
      &\geq 1 - e^{(\log n)^{D}} e^{-2(\log n)^{c_1 \lambda/2}} \geq 1 - \frac{1}{2} e^{-(\log n)^{c_1 \lambda/2}}.
\end{split}\]
Since $\mathscr F$ was arbitrary, this implies the claimed bound \eqref{eq:conditional_merging}, completing the proof.
\end{proof}

\subsection{Completing the proof of the main theorem: The implications (\ref{implication:corridor0}) and (\ref{implication:corridor})}

In this section we apply \cref{prop:low_growth_main} to complete the proof of the \cref{prop:complicated_induction_statement} and hence of \cref{thm:main}. Given \cref{prop:implication_F}, what remains is to verify the implications (\ref{implication:corridor0}) and (\ref{implication:corridor}).

\medskip

\begin{proof}[Proof of \cref{prop:complicated_induction_statement}]
Let $D := 20$, and accordingly let $\lambda_0(d,D)$ and $c=c(d,D)$ be the constants from \cref{prop:low_growth_main} with this value of $D$. Let $\lambda := \max\{ \lambda_0 , 10/c \}$, and let $K(d,D,\lambda)$ and $M(d,D,\lambda)$ be the corresponding constants from \cref{prop:low_growth_main} that are there called $K_1$ and $n_0$. (We want to avoid reusing the label $n_0$.) We claim that if we define $\delta_0$ using this value of $K$, then the implications (\ref{implication:corridor0}) and (\ref{implication:corridor}) (for all $i \geq 1$) hold whenever $p_0 \geq 1/d$, $\delta_0 \leq 1$, and $n_0$ is sufficiently large with respect to $d$, which in particular guarantees that $n_0\geq \max\{16,M\}$. The implication (\ref{implication:corridor0}) is immediate (i.e., is a direct consequence of \cref{prop:low_growth_main} after unpacking the definitions), so we will just explain how to prove the implication (\ref{implication:corridor}).

\medskip
Fix $i \geq 1$ and assume that $\upsc{Full-space}(i)$ holds and that $\upsc{Corridor}(k)$ holds for all $1 \leq k \leq i$. Our goal is to establish that $\upsc{Corridor}(i+1)$ holds provided that $n_0$ is sufficiently large with respect to $d$. This follows immediately from \cref{prop:low_growth_main} if we can show that for every $n \in \mathscr{L}(G,20) \cap [n_{i-1},n_{i}]$ there exists some $b \leq \frac{1}{8}n^{1/3}$ such that
\begin{equation} \label{eq:conditions_to_apply_low_growth_sprinkle}
     \left(\frac{K \log \log n}{\min\{\log n, \log \op{Gr}(b)\}}\right)^{1/4} \leq \left( \log \log n_i \right)^{-1/2} \quad \text{and} \quad \p_{p_i}\left( \op{Piv}[4b,n^{1/3}] \right) \leq (\log n)^{-1}.
\end{equation}
Consider an arbitrary $n \in \mathscr{L}(G,20) \cap [n_{i-1},n_{i}]$ (assuming one exists; the claim is vacuous if not). We split into two cases according to whether $(\log n)^{2/3} \in \mathscr L(G,20)$. First suppose that $(\log n)^{2/3} \not\in \mathscr L(G,20)$, so in particular
\[
      \op{Gr}((\log n)^{2/3}) \geq e^{((\log n)^{2/9})^{20}}.
\]
By \cref{cor:trivial_a_priori_uniqueness_zone}, provided $n_0$ is sufficiently large with respect to $d$, we know that
\[
      \p_{p_i}\left( \op{Piv}[4 (\log n)^{2/3} ,n^{1/3}] \right) \leq (\log n)^{-1}.
\]
So in this case both conditions in \eqref{eq:conditions_to_apply_low_growth_sprinkle} are satisfied for $b = (\log n)^{2/3} \leq \frac{1}{8} n^{1/3}$ provided that $n_0$ is sufficiently large with respect to $d$.

\medskip

Now instead suppose that $(\log n)^{2/3} \in \mathscr L(G,20)$. Note that we can always find some $k \in \{ 1 , \ldots, i\}$ such that $(\log n)^{2/3} \in [n_{k-2},n_{k-1}]$. (This is why we took $n_{-1}=(\log n_0)^{1/2}$ as small as we did in the statement of the proposition.) So by our hypothesis that $\upsc{Corridor}(k)$ holds for this particular value of $k$, we have that
\begin{equation}
\label{eq:almost_done_kappa_bound}
      \kappa_{p_{k}}\left( e^{ [ \log ( [\log n]^{2/3} ) ]^{10} } , [\log n]^{2/3} \right) \geq e^{-(\log \log n_k)^{1/2}}.
\end{equation}
We now claim that $b := \frac{1}{5}\min\{e^{(\log \log n)^{9}}, \op{Gr}^{-1}(e^{(\log n)^{1/10}}) \} $ satisfies both conditions from \cref{eq:conditions_to_apply_low_growth_sprinkle} provided that $n_0$ is sufficiently large with respect to $d$. The inequality $b \leq \frac{1}{8}n^{1/2}$ again holds trivially when $n_0$ is large. The definition of $b$ ensures that $\log \op{Gr}(b)\geq \frac{1}{5} \log \op{Gr}(5b) \geq \frac{1}{5} (\log \log n)^{9}$, so that the first condition also trivially holds when $n_0$ is large. To see that the second condition holds, we apply \cref{lem:nearby_clusters_using_two_point} and \cref{prop:two_arm_bound_in_box} to obtain that there exists a constant $C$ such that
\begin{multline*}
      \p_{p_i}( \op{Piv}[4b,n^{1/3}] ) \leq \p_{p_i}( \op{Piv}[1,n^{1/3}/2] ) \cdot \frac{  \abs{S_{4b}}^2 \op{Gr}(5b) }{ \min_{a,b \in S_{4b}}(a \xleftrightarrow{B_{5b}} b ) } 
    \\\leq C\left(\frac{(\log n)^{20}}{n}\right)^{1/4} e^{3 (\log n)^{1/10}+(\log \log n_k)^{1/2}} \leq  (\log n)^{-1}
\end{multline*}
whenever $n_0$ is sufficiently large with respect to $d$, where we used the estimate \eqref{eq:almost_done_kappa_bound} and the fact that $4b \leq e^{ [ \log ( [\log n]^{2/3} ) ]^{10}}$ and $5b \geq (\log n)^{2/3}$ (when $n_0$ is sufficiently large) to bound the term $\min_{a,b \in S_{4b}}(a \xleftrightarrow{B_{5b}\,} b )$. This completes the proof.
\end{proof}

\section{Closing discussion and open problems} 

\subsection{Joint continuity of the supercritical infinite cluster density}
\label{sec:joint_continuity}

Recall that $\mathcal G^*$ is the space of all infinite, connected, transitive graphs that are not one-dimenional, which we endow with the local topology, and recall that for all $p \in (0,1)$ and $G \in \mathcal G^*$, the \textbf{infinite cluster density} is defined to be $\theta(G,p) := \p_{p}^G(o \leftrightarrow \infty)$. Consider a sequence $(G_n)_{n \geq 1}$ in $\mathcal G^*$ converging to some $G\in \cG^*$. The main result of the present paper \cref{thm:main} states that $p_c(G_n) \to p_c(G)$. One could ask the following more refined question: does $\theta^{G_n} \to \theta^{G}$ pointwise? (One can observe from the mean-field lower bound, say, that a positive answer to this question would imply our result that $p_c(G_n) \to p_c(G)$.) For $p<p_c(G)$, it follows immediately from the lower semi-continuity of $p_c$ that $\theta(G_n,p) \to \theta(G,p)=0$, so the only non-trivial cases are when $p = p_c(G)$ and $p > p_c(G)$. The case $p=p_c(G)$ appears to be hard. Indeed, if one could prove this result in the case of toroidal slabs $(\mathbb Z^{2} \times \mathbb Z/n\mathbb Z)_{n \geq 1} $ converging to the cubic grid $\mathbb Z^3$, then it would follow from the main result of \cite{MR3503025} that $\theta(\mathbb{Z}^3,p_c(\mathbb Z^3)) = 0$, a notorious open question. We have nothing interesting to say about this case. However, in our other work \cite{easo2021supercritical2} we prove the following theorem, which together with \cref{thm:main} completely resolves the case when $p > p_c(G)$.

\begin{thm}[\cite{easo2021supercritical2}] \label{thm:locality_of_density}
      Let $(G_n)_{n \geq 1}$ be a sequence in $\mathcal G^*$ that converges in the local topology to some $G \in \mathcal G^*$. Then $\theta(G_n,p) \to \theta(G,p)$ as $n\to\infty$ for every $p > \limsup_{n \to \infty} p_c(G_n)$.
\end{thm}

Note that the main theorem of \cite{easo2021supercritical2} is much more general than this and also establishes a form of locality for the density of the \emph{giant cluster} on \emph{finite} transitive graphs that may have divergent degree.
Together with \cref{thm:main} this theorem yields the following elementary corollary:

\begin{cor}
      $\theta(G,p)$ is continuous on the open set $\{(G,p):G\in \cG^*, p > p_c(G)\}$. 
      Moreover, if $(G_n)_{n \geq 1}$ is a sequence in $\mathcal G^*$ that converges in the local topology to some $G \in \mathcal G^*$ then $\theta(p,G_n)\to \theta(p,G)$ as $n\to\infty$ for each $p \in [0,1] \backslash \{ p_c(G)\}$.
\end{cor}

Let us roughly indicate how the tools built in \cref{sec:snowballing} could be used to give an alternative proof of \cref{thm:locality_of_density}. This alternative proof is less general and (arguably) more involved and  than the one given in \cite{easo2021supercritical2}, but the result is quantitatively stronger: it can be used to prove that $\theta(G,p)$ is not just continuous but even locally H\"older continuous on the supercritical set (with the power in the definition of H\"older continuity possibly degenerating near the boundary of the set). 

\medskip

Let  $\mathcal G_d^*$ denote the set of infinite transitive graphs with vertex degree exactly $d$ that are not one-dimensional. 
As explained in detail in \cite{easo2021supercritical2}, it suffices to prove a tail estimate on the size of finite clusters in supercritical percolation that is uniform over $\mathcal G_d^*$ for each $d$, i.e.\! it suffices to prove that
\begin{equation} \label{eq:uniform_tail}
      \lim_{m \to \infty} \sup_{G \in \mathcal G_d^*} \sup_{p \geq p_c(G) + \eps} \p_{p}^{G}( m \leq \abs{K_o} < \infty ) = 0 \quad \text{for all $\eps >0$ and $d \geq 1$,}
\end{equation}
where we recall that $K_o$ denotes the cluster of the root vertex $o$. We will focus on proving an estimate of this form for the set of unimodular graphs $\cU_d^*$ instead of $\cG_d^*$; in the nonunimodular case much stronger results (with optimal dependence on $p-p_c$ and $m$) can be proven by invoking the results of \cite{hutchcroft2020nonuniqueness,MR4500202} as explained in detail in \cite{easo2021supercritical2}.
Our proof will yield quantitatively that
for every $\eps>0$ and $d\geq 1$ there exist constants $C$ and $c$ such that
\[
\sup_{G \in \mathcal U_d^*} \sup_{p \geq p_c(G) + \eps} \p_{p}^{G}( m \leq \abs{K_o} < \infty ) \leq Cm^{-c};
\]
running the proof of continuity with this quantitative estimate yields the aforementioned local H\"older continuity of $\theta(G,p)$.  This bound is quantitatively much better than the bound coming from the proof in \cite{easo2021supercritical2}. On the other hand, it is also much worse than the conjectured optimal bounds, which are stretched exponential in $m$ (see Section 5.3 of \cite{MR4243018}). Having any \emph{superpolynomial} tail bound would imply that $\theta(G,p)$ is a smooth function of $p\in (p_c(G),1]$ for each fixed $G$; for $\mathbb{Z}^d$ and for nonamenable graphs it is known that the density is not just smooth but real analytic on this set \cite{MR4243018,MR4631964}.

\medskip

Let $\eps >0$, $d \geq 1$, $G \in \mathcal U_d^*$, $p \geq p_c(G) + \eps$, $m \geq 1$, and $\eta >0$ be arbitrary, and suppose that $\p_{p}^{G}( m \leq \abs{K_o} < \infty ) \geq \eta$. It suffices to prove that $m$ is necessarily bounded above by some constant $M(\eps,d,\eta) < \infty$. Let $\p$ denote the canonical monotone coupling $(\omega_q)_{q \in [0,1]}$ of the percolation measures $(\p_q^G)_{q \in [0,1]}$. By the mean-field lower bound and transitivity, one can find vertices $u,v \in V(G)$ such that
\[
      \p\Bigl( \abs{K_u(\omega_{p-\eps/2})} \geq m \text{ and } \abs{K_v(\omega_{p})} \geq m \text{ but } u \nleftrightarrow v \Bigr) \geq \frac{\eta \eps}{2},
\]
where $K_u(\omega_{p-\eps/2})$ denotes the cluster of $u$ in $\omega_{p-\eps/2}$.
In particular, writing $\mathbb G_{1/m}$ for the law of a ghost-field $\mathscr G$ of intensity $1/m$ on the whole vertex set $V(G)$,
\begin{equation} \label{eq:assume_large_cluster_pair}
      \mathbb G_{\frac{1}{m}} \otimes \p ( u \xleftrightarrow{\omega_{p-\eps/2}} \mathscr G \xleftrightarrow{ \omega_{p} } v \text{ but } u \nxleftrightarrow{\omega_{p}} v ) \geq \bra{1-\frac{1}{e}}^2\frac{\eta \eps}{2} \geq \frac{\eta\eps}{8}.
\end{equation}
Assume for now that $G$ is amenable so that there is at most one infinite cluster $\p_q^G$-almost surely for every $q \in [p-\eps/2,p]$. Then by \cref{lem:ghost_in_box_transitivity} with $(X,A,Y):= ( \{ u\}, V(G) , \{v\} )$, one can deduce that for some constants $c_3(\eps,d) > 0$ and $C(\eps,d) < \infty$,
\begin{equation} \label{eq:no_large_cluster_pair}
      \mathbb G_{\frac{1}{m}} \otimes \p ( u \xleftrightarrow{\omega_{p-\eps/2}} \mathscr G \xleftrightarrow{ \omega_{p} } v \text{ but } u \nxleftrightarrow{\omega_{p}} v ) \leq C m^{-c_3}.
\end{equation}
By combining \eqref{eq:assume_large_cluster_pair} and \eqref{eq:no_large_cluster_pair}, we deduce that $m \leq M(\eps,d,\eta) := (8C/(\eta \eps))^{1/c_3}$, as required.


\medskip

 Finally, to handle the case when $G$ is nonamenable (but still unimodular), one can still run essentially the same argument as above but with some technical modifications to handle the possible existence of multiple infinite clusters. The key difference is that instead of tracking connections $u \leftrightarrow v$ as usual, we instead track $\emph{wired}$ connections $u \xleftrightarrow{\,\text{wired}\,} v$ where
\[
      \{u \xleftrightarrow{\,\text{wired}\,} v\} := \{ u \leftrightarrow v \} \cup \{ u \leftrightarrow \infty \text{ and } v \leftrightarrow \infty \}.
\]
One can verify that the proof of \cref{lem:ghost_in_box_transitivity} works just as well with this alternative notion of connnectivity, without requiring the hypothesis about the uniqueness of the infinite cluster. The rest of the argument explained above can then be adapted to work with this wired notion of connectivity also.


\medskip

As explained in \cite{easo2021supercritical2}, the estimate \eqref{eq:uniform_tail} has the following nice interpretation. Consider the space $(0,1)\times \mathcal G^* \to (0,1)$ with the product topology, and consider the function $\theta : (0,1)\times \mathcal G^* \to (0,1)$ mapping $(p,G) \mapsto \theta^G(p)$. It is natural to ask whether $\theta$ is continuous as a function of two variables. One can show that a priori, $\theta$ is continuous if and only if the conclusion of \cref{thm:main} holds (locality of $p_c$), the estimate \eqref{eq:uniform_tail} holds (i.e., there is a uniform tail bound on supercritical finite clusters), and $\theta(G,p_c(G)) = 0$ for every $G \in \mathcal G^*$ (continuity of the phase transition), the last statement being one of the most important open conjectures in the general study of percolation in $\mathcal G^*$. (In this decomposition, \eqref{eq:uniform_tail} handles the interior of the supercritical region $\mathcal S := \{ (p,G) : p > p_c(G) \}$, \cref{thm:main} implies that this region is open, and the continuity conjecture handles the boundary values.)

\medskip

\noindent \textbf{Finite graphs.} As mentioned above, in \cite{easo2021supercritical2}  we prove versions of \cref{thm:locality_of_density} and \cref{eq:uniform_tail} that also apply to families of bounded-degree \emph{finite} transitive graphs. The above sketches work just as well in this context too; we have stated things in terms of infinite graphs purely for simplicity. 
Moreover, the above sketch can be used to give an alternative proof that for supercritical percolation on bounded-degree finite transitive graphs, the giant cluster is unique and has concentrated density, recovering the results of our two papers \cite{easo2021supercritical,easo2021supercritical2} in this case. Note however that all of the tools from \cref{sec:snowballing} break down rather badly when working with families of finite graphs that have large vertex degrees (e.g.\! vertex degrees that grow at least as a power of the total number of vertices), partly because for such graphs the emergence of a giant cluster can occur around values of $p$ close to $0$. To handle this more general setting of arbitrary finite transitive graphs, we know of no alternative proofs of the uniqueness or concentration of the supercritical giant cluster to those we give in \cite{easo2021supercritical,easo2021supercritical2}.

\label{sec:continuity_of_density}

\subsection{The $p_c$ gap and its witnesses}
\label{sec:pc_gap}

Since $\cG_d^*$ is compact for each $d\geq 1$, it is a consequence of \cref{thm:main} that $p_c$ attains its maximum on $\cG_d^*$ for each $d\geq 1$.
In \cite{panagiotis2021gap}, Panagiotis and Severo improved upon the results of \cite{MR4181032,hutchcroft2021nontriviality} to establish that there exists a \emph{universal} $\eps>0$ (independent of the degree) such that every Cayley graph with $p_c<1$ has $p_c\leq 1-\eps$; Lyons, Mann, Tessera, and Tointon \cite{MR4529894} give the explicit bound $\eps \geq \exp(-\exp(17\exp(100\cdot 8^{100})))$. Presumably a similar result holds for transitive graphs that are not Cayley. The following natural conjecture would strengthen this result, and would also imply by \cref{thm:main} that $p_c$ attains its global maximum on $\cG^*$.

\begin{conj}
\label{conj:pc_high_degree}
There exists a universal constant $C$ such that if $G$ is an infinite, connected, transitive, \emph{simple} graph of vertex degree $d$ that is not one-dimensional then $p_c(G) \leq C/d$.
\end{conj}

For many natural families of high-degree graphs we have the stronger statement that $p_c\sim 1/\deg$ as the degree diverges. For example, this holds for $\Z^d$ as $d\to\infty$ by a theorem of 
Kesten \cite{kesten1990asymptotics} (see also \cite{alon2004percolation}). Moreover, this is not just a high-dimensional phenomenon:
Penrose \cite{penrose1993spread} proved that a similar-estimate holds for the ``spread-out'' $d$-dimensional lattice, in which $x,y\in \Z^d$ are connected by an edge whenever $\|x-y\|\leq R$, when $d$ is fixed and $R\to \infty$. On the other hand, the example illustrated in \cref{fig:macromolecular} shows that high-degree transitive graphs do not always have $p_c\sim 1/\deg$ even when they are not one-dimensional in any sense.

\begin{figure}[t!]
\centering
\includegraphics[width=0.7\textwidth]{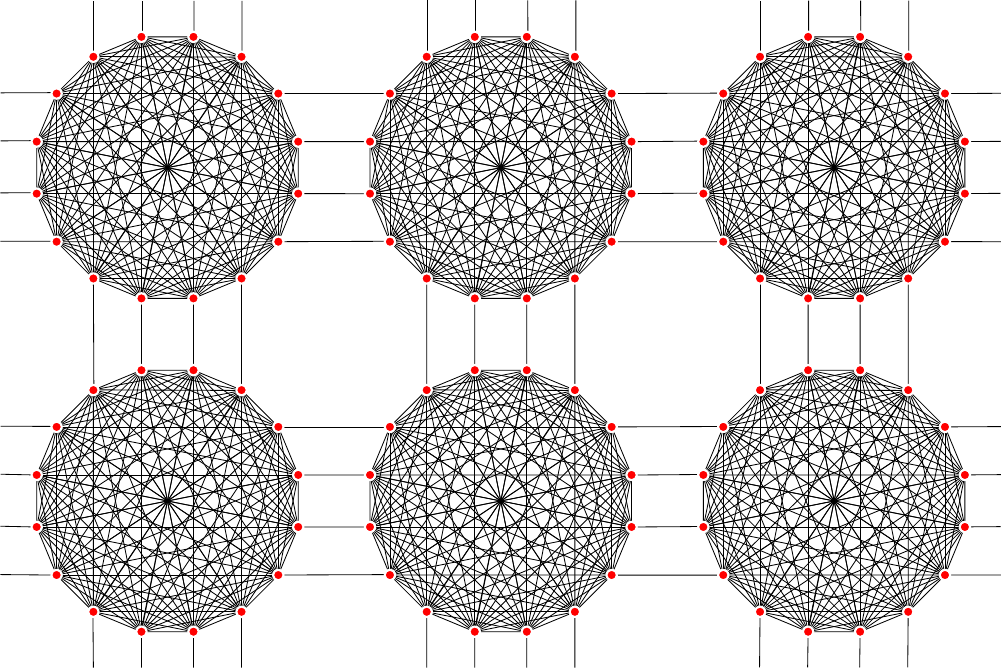}
\caption{The transitive graph formed by laying out copies of $K_{4n}$ in an infinite square grid as above has $p_c \geq (4 \log 2-o(1))/\deg$  as $n\to\infty$ as can be seen by coupling with bond percolation on $\Z^2$ and using a Poisson approximation.  Since $4\log 2\approx 1.2>1$, this shows that the asymptotic estimate $p_c\sim 1/\deg$ can fail for high-degree vertex-transitive graphs even when these graphs are not one-dimensional in any sense. An exact asymptotic estimate $p_c\sim C/\deg$  can be proven with a little further work. (Indeed, the constant  $C\approx 3.095$ is the unique solution to the equation $C(1+C^{-1}W[-e^{-C}C])^2 = 4 \log 2$ where $W$ is the Lambert $W$ function.)}
%
%
\label{fig:macromolecular}
\end{figure}

Once one knows that $p_c$ attains a maximum (either globally on $\cG^*$ or on $\cG^*_d$), it becomes interesting to understand which graphs attain this maximum.  Martineau and Severo \cite{martineau2019strict} proved that $p_c$ is strictly increasing under quotients, so that any maximal graph must have no non-trivial quotients in~$\cG^*$. It seems reasonable to believe that the maximal graph would be a lattice of low degree and in low dimension. Consulting tables of numerical values of $p_c$  for these lattices (as can be found on \url{https://en.wikipedia.org/wiki/Percolation_threshold}) leads to the highly speculative conjecture that $p_c$ is maximized by the so-called super-kagome lattice (a.k.a.\ 3-12 lattice); see \cref{fig:superKagome} for an illustration.

\begin{problem} Investigate the transitive graphs in $\cG^*_d$ that maximize $p_c$ for each degree $d$, as well as the global maximum in $\cG^*$ if this maximum exists. Are these maxima uniquely attained? Does 
$p_c:\mathcal{G}^* \to [0,1]$ attain its maximum uniquely at the $3$-$12$ lattice (a.k.a.\ super-kagome lattice), which has $p_c\approx 0.7404207$? When restricted to edge-transitive graphs, does $p_c$ attain its unique maximum at the hexagonal lattice, which has $p_c=0.65270\ldots=1-2 \sin (\pi/18)$?
\end{problem}

One may wish to restrict attention to simple graphs. Similar questions have been investigated for self-avoiding walk by Grimmett and Li \cite{grimmett2020cubic}.

\begin{figure}[t!]
\centering
\includegraphics[width=\textwidth]{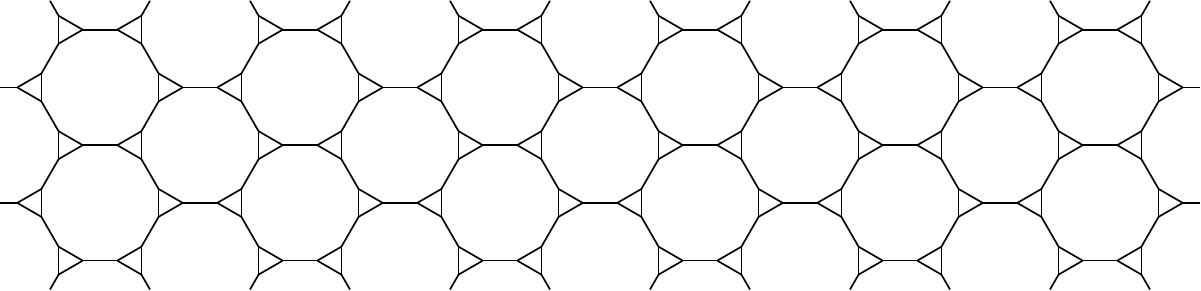}
\caption{The 3-12 (a.k.a. super-kagome) lattice is the current best candidate for the transitive graph with the highest non-trivial value of $p_c$ for bond percolation. Its critical value  $p_c\approx 0.7404\ldots$ has been estimated to great precision numerically in \cite{scullard2020bond}. The transitive graph with the next highest value of $p_c$ to have been investigated numerically is the truncated trihexagonal lattice, which has $p_c \approx 0.6937$.}
\label{fig:superKagome}
\end{figure}

\subsection*{Acknowledgments}
This work was supported by NSF grant DMS-2246494. TH thanks Ariel Yadin for discussions on the history of the ``cool inequality'' and thanks Vincent Tassion for discussions on the plausibility of \cref{conj:pc_high_degree}. Both authors thank Geoffrey Grimmett,  Russ Lyons, and S\'ebastien Martineau for helpful comments on an earlier version of the manuscript.

\addcontentsline{toc}{section}{References}

\setstretch{1}

\footnotesize{
\printbibliography
}
%
\end{document}